\documentclass[
letterpaper,12pt]
{amsart}
\usepackage{amssymb,latexsym}
\usepackage{mathrsfs}
\usepackage{euscript}
\usepackage{hyperref}
\usepackage{amsmath}
\usepackage{upgreek}

\usepackage[cjk]{kotex}

\newtheorem{theorem}{Theorem}[section]
\newtheorem{definition}[theorem]{Definition}
\newtheorem{lemma}[theorem]{Lemma}
\newtheorem{corollary}[theorem]{Corollary}
\newtheorem{remark}[theorem]{Remark}
\newtheorem{proposition}[theorem]{Proposition}
\newtheorem{notation}[theorem]{Notation}

\def\cal{\mathcal}

\def\R{{\mathbb R}}

\def\<{\left \langle |}
\def\>{\right>}

\def\<{\langle \hspace{-.04cm} | \hspace{.01cm}}
\def\>{\hspace{.01cm} | \hspace{-.04cm} \rangle}

\def\Im{{\rm Im}\,}
\def\Re{{\rm Re}\,}

\def\({( \hspace{-0.335em}(}
\def\){) \hspace{-0.335em})}
\def\supp{{\rm supp\,}}

\def\fu2{\frac{n}{2}}

\def\l({\left(}
\def\r){\right)}
\def\be{\begin{enumerate}}
\def\ee{\end{enumerate}}

\def \dist {{\rm dist}}

\def \rdist {{\rm \, rdist}}
\def \inrdist {{\rm \, inrdist}}

\def \ec {{\rm ec}}
\def \child{{\rm ch}}

\def \ch{{\rm ch}}

\def \Int{{\rm Int}}

\newcommand\smland{\raise.2ex\hbox{$\scriptstyle\land$}\hspace{.01cm}}
\newcommand\smlor{\raise.1ex\hbox{$\scriptstyle\lor$}}

\title[New local $T1$ Theorems 
compactness 
]{New local $T1$ Theorems\\ 
on non-homogeneous spaces}
\author{Paco Villarroya}
\address{W. California Avenue, Sunnyvale, CA, 94086, USA}
\email{pieisiou@gmail.com}

\thanks{The author has been partially supported by Spanish Grant MTM2011-23164}


\begin{document}

\begin{abstract}
We develop new local $T1$ theorems to characterize
Calder\'on-Zygmund operators that
extend boundedly or compactly on $L^{p}(\mathbb R^{n},\mu)$  
with 
$\mu$ 
a measure of power growth.

The results, whose proofs do not 
require 
random grids,
allow the use of 
a countable collection
of testing functions.

As a corollary, we describe the measures $\mu$ of the complex plane 
for which 
the Cauchy integral defines a compact operator on $L^p(\mathbb C,\mu)$.

\end{abstract}

\maketitle

\section{Introduction}

%
The $T1$ theorem characterizes boundedness of Calder\'on-Zygmund operators $T$ in terms of the functions 
$T1$ and $T^*1$. 
On the other hand, the local $T1$ theorem attains similar characterization
using the action of $T$ and $T^*$ 
over a system of indicator functions $(\chi_{Q})_{Q\in \mathcal Q}$ of all cubes
with edges parallel to the coordinate axes. 

The idea of a local $T1$ theorem was first introduced in 1990 by M. Christ \cite{Christ2} 
in connection with the geometric description of removable compact sets for bounded analytic functions (known as Pailenv\'e's problem).
His motivation was that, in principle,
finding a system of local testing functions should be easier than identifying a single function over which the operator behaves well globally.
This approach was shown to be right 
at the turn of the century 
when 
F. Nazarov, S. Treil and A. Volberg proved 
the first local $T1$ and $Tb$ theorems for non-doubling measures
\cite{NTV4} and \cite{NTVTb} 
(see also \cite{NTV}  and \cite{NTV3}). 
 
%
Since then,  
research work on this subject has been continuously growing with special focus on  
more general criteria of boundedness (\cite{AHMTT}, \cite{HyNa},  \cite{HyMa}), variants that apply to new settings (\cite{Hy3}, \cite{Hy}, \cite{HyVa}) and applications 
to PDEs (\cite{Ho2b}, \cite{Ho2}). 
The articles
\cite{Hof}, \cite{Ho}, \cite{AR}, \cite{LaVa} 
and the books \cite{Christ}, \cite{David}, \cite{ke}, \cite{To2} 
 provide detailed accounts of the evolution of this theory. 
Few years ago, papers \cite{V}, \cite{V2}  
presented global $T1$ and $Tb$ theorems characterizing  compactness of Calder\'on-Zygmund operators.
These results can be used to prove 
compactness of many double layer potential operators
(see \cite{V2}). In turn, this allows 
via Fredholm Theory (see \cite{Fabes}) to deduce invertibility of the Laplacian on a large class of domains. 
Following this line of research, the current paper 
introduces a local $T1$ Theorem for non-doubling measures, that is, a
criterion of boundedness and compactness that relies on the action of the operator over 
a family of 
indicator functions of dyadic cubes 
(Theorems \ref{Mainresultrestrictedbddness2} and   
\ref{Mainresult2}). 

All known proofs of $T1$ theorems 
on non-doubling spaces employ randomization methods 
to deal with the fact that, when 
using the kernel decay, estimates of the dual pair 
$\langle T\chi_I, \chi_J\rangle $ 
grow like the logarithms of both the distance between the cubes $I,J$ and the ratio between their side lengths. To overcome this issue, the method considers grids of general cubes rather than only the grid of dyadic cubes. 
In the space of all these grids, cubes with close boundaries and
very different sizes are rare and thus, they can be assigned a small probability. Then, by averaging among all grids, the contribution of such cubes can be made arbitrarily small. 
The costs of this method include 
a delicate technique of decomposition called surgery and
the requirement by hypothesis of a non-contable family of testing functions (one per each cube in $\mathbb R^n$).

But there is a catch. Once randomization is applied and the proof of boundedness is complete, the 
estimate $|\langle T\chi_I, \chi_J\rangle |
\leq \|T\| \mu(I)^{\frac{1}{2}}\mu(J)^{\frac{1}{2}}$ shows that the dual pair with any two cubes
never blows up. 
Random cubes were used to tame ghost singularities that 
did not exist. 

We introduce a 
new 
approach that does not use random grids and so, it
avoids many technicalities of randomization.
The method 
allows the use of 
a countable family of testing functions. 

Regarding compactness, we provide an application to the Cauchy integral operator with a non-doubling measure, which is defined by
$$
C_{\mu }f(z)=\int_{\mathbb{C}} \frac{f(w)}{w -z} \, d\mu (w).
$$
It is known that if the measure is defined by the indicator function of a unit line segment $S$, that is
$d\mu_S= \chi_{S}d\mathcal H^{1}$ with $\mathcal H^{1}$ the one-dimensional Hausdorff measure in $\mathbb C$,
then $C_{\mu_S}$ 
is bounded but not compact on
$L^{2}(\mu_S)$.
On the other hand, if the measure is defined by the indicator function of a unit square $Q$, that is
$d\mu_Q = \chi_{Q}dm$ with $m$ the Lebesgue measure in $\mathbb C$, then $C_{\mu_Q}$ is compact on 
$L^{2}(\mu_Q)$. 
Theorem \ref{Cauchytheo} describes for which measures $\mu$ of the complex plane the Cauchy integral operator $C_{\mu }$ can be compactly extended on 
$L^{2}(\mu)$.

The outline of the paper is as follows. In section 2, 3, and 4 we introduce some notation, 
define the class of operators under study and state 
the main results respectively. In section 5, we study a smooth truncation of the kernel, while  
in section 6 we describe the Haar wavelet system.  Section 7 focuses on obtaining estimates for the action of the operator over Haar wavelets.  
In section 8 we deal with the paraproducts and   
section 9 is devoted to the proof of the main result Theorem \ref{Mainresult2}. 

I want to express my appreciation to 
\hspace{-.3cm}
\begin{CJK}{UTF8}{mj}
^^ea^^b9^^80^^ec^^9e^^90^^ec^^98^^81
\end{CJK}
\hspace{-.3cm} 
in Sunnyvale, USA. I would also like to thank Professors  Oscar Blasco, Anthony Carbery, Neil Lyall, and Gerardo Mendoza for their unreserved support.

\section{Notation}\label{ecandrdist}
\subsection{\bf Cubes and dyadic cubes}\label{grids}
Let ${\mathcal C}$ be the family of cubes in $\mathbb R^n$ defined by tensor products of intervals of the same length, namely, 
$I=\prod_{i=1}^{n}[a_{i},a_{i}+l)$ with $a_i,l\in \mathbb R$. 

For each cube $I\in \mathcal C$, we denote its center by $c(I)$, its side length by $\ell(I)$ and its boundary in the euclidean topology of $\mathbb R^n$ by
$\partial I$.


Let $\mathcal D_1$ be the family of dyadic cubes $I=\prod_{i=1}^{n}2^{-k}[j_{i},j_{i}+1)$ with $j_i,k\in \mathbb Z$. 
Let $\tilde{\mathcal D}_1$ be the family of open dyadic cubes $I=\prod_{i=1}^{n}2^{-k}(j_{i},j_{i}+1)$ with $j_i, k\in \mathbb Z$. 

Now, let $\lambda_1=0$ and $\lambda_2,\ldots, \lambda_n \in \mathbb R$ such that $\lambda_i\in \mathbb R\setminus (\cup_{j=1}^{i-1}(\lambda_j+\mathbb Q))$.
Let $a_i=\lambda_i (1,\ldots, 1)\in \mathbb R^n$. 
For $i\in\{ 1,\ldots, n\}$, 
we define the families of cubes
\begin{equation}\label{grids}
\mathcal{T}_i\mathcal{D}=a_i+\mathcal{D}_1=\{ a_i+I :  
I\in \mathcal{D}_1\},
\end{equation}
with $a_i+I\in \mathcal{C}$ such that 
$c(a_i+I)=a_i+c(I)$ and $\ell(a_i+I)=\ell(I)$. 


We write any particular instance 
of the families of cubes $\mathcal{T}_i\mathcal{D}$ (and 
$\mathcal D_{\partial}^{r}$ defined later in this section) simply as $\mathcal D$. And we will often denote any
particular instance 
of the families of cubes $\mathcal{T}_i\tilde{\mathcal{D}}$ (defined in similar way as $\mathcal{T}_i\mathcal{D}$) simply by $\tilde{\mathcal D}$.

Given a measurable set $\Omega \subset \mathbb R^{n}$, let $\mathcal D(\Omega )$ be the family of dyadic cubes $I\in \mathcal D$ such that $I\subsetneq \Omega$.



For $\lambda >0$, we write $\lambda I$ for the cube such that $c(\lambda I)=c(I)$ and $\ell(\lambda I)=\lambda \ell(I)$. 
We write $\mathbb B=[-1/2,1/2)^{n}$ and $\mathbb B_{\lambda }=\lambda \mathbb B$.
We also denote by $\lambda \mathcal D$ the family of cubes
$\lambda I$ with $I\in \mathcal D$. 

For $I\in \mathcal D$, we denote by 
$\child(I)$ the family of dyadic cubes $I'\subset I$ such that $\ell(I')=\ell(I)/2$, and 
by $I_{p}\in \mathcal D$ the parent of $I$, the 
only dyadic cube such that $I\in \child(I_{p})$.  
If $\Omega \in \mathcal D$ and 
$I\in \mathcal D(\Omega)$ then $I_p\subseteq \Omega$.

We also write 
$I^{\rm fr}$ for the collection of cubes $J\in {\mathcal D}$ such that $\ell(J)=\ell(I)$ and $\dist(I,J)=0$, 
where $\dist(I,J)$ denotes the set distance between $I$ and $J$.

\subsection{Pairs of cubes: eccentricity and relative distances}
Given $I,J\in \mathcal{C}$,
if $\ell(J)\leq \ell(I)$ we write $I\smland J=J$, $I\smlor J=I$,
while if $\ell(I)<\ell(J)$, we write
$I\smland J=I$, $I\smlor J=J$. 

We define $\langle I,J\rangle$ as the only cube containing $I\cup J$ with the smallest possible side length and 
such that  $\sum_{i=1}^{n}c(I)_{i}$ is minimum. 
We note that
$
\ell(\langle I,J\rangle )  
\approx  \dist(I,J)+\ell(I\smlor J),
$
where $\dist(I,J)$ denotes the euclidean set distance between $I$ and $J$.

We  
define $[I,J]$ as the unique cube satisfying $\ell([I,J])=\dist(I,J)$, $\lambda [I,J]\cap I\neq \emptyset$  and
$\lambda [I,J]\cap J\neq \emptyset$ for any $\lambda >0$, and such that $\sum_{i}c([I,J])_{i}$ is mininum. 

We define the eccentricity and the 
relative distance of $I$ and $J$ as
$$
\ec(I,J)=\frac{\ell(I\smland J)}{\ell(I\smlor J)},
\hskip20pt 
\rdist(I,J)=
1+\frac{\dist(I,J)}{\ell(I\smlor J)}.
$$

We define
the inner boundary of $I$ as $\mathfrak{D}_{I}=
\cup_{I'\in \child(I)}\partial I'$, and 
the inner relative distance 
of $J$ and $I$ by
$$
\inrdist(I,J)=1+\frac{\dist(I\smland J,{\mathfrak D}_{I\smlor J})}{\ell(I\smland J)}.
$$




\subsection{Lagom cubes}
For $M\in \mathbb N$, we define ${\cal C}_{M}$ as the family of cubes in $\mathcal C$ such that
$2^{-M}\leq \ell(I)\leq 2^{M}$ and
$\rdist(I,\mathbb B_{2^{M}})\leq M$.  
We name the cubes in ${\cal C}_{M}$ 
as lagom cubes.

We write ${\cal D}_{M}={\cal C}_{M}\cap {\mathcal D}$, 
${\cal D}_{M}^{c}={\cal D}\backslash {\cal D}_{M}$,
${\cal D}_{M}(\Omega )={\cal D}_{M}\cap {\mathcal D(\Omega )}$, and 
${\cal D}_{M}^c(\Omega )={\cal D}_{M}^c\cap {\mathcal D(\Omega )}$.

\subsection{Grids of dyadic cubes of lower  dimensions}\label{grids2}
We write $\mathcal D^n_\partial =\mathcal D_1$. 
Let $\partial \mathcal D^{n}$ the set defined by 
the union of $\partial I$ for all $I\in \mathcal D_1$. 
This set is the union of countably many affine euclidean spaces of dimension $n-1$. Then 
let $\mathcal D_{\partial}^{n-1}$ be the family of dyadic cubes in $\partial \mathcal D^{n}$, namely, the cubes of dimension $n-1$ of the form 
$$
I=\prod_{i=1}^{l-1}2^{-k}[j_{i},j_{i}+1)\times \alpha \times \prod_{i=l+1}^{n}2^{-k}[j_{i},j_{i}+1)
$$ 
where $l\in \{1, \ldots ,n\}$, $j_i, k \in \mathbb Z$ and $\alpha \in \{ 2^{-k}j_{l}, 2^{-k}(j_{l}+1)\}$, with the convention that if $b<a$ then 
$\prod_{i=a}^{b}2^{-k}[j_{i},j_{i}+1)=\emptyset $.

We now continue recursively. 
For $0<r<n$, we define
$\partial \mathcal D^{r+1}$ as the union of $\partial I$ for all $I\in \mathcal D_{\partial}^{r+1}$, where here 
$\partial I$ denotes the border of $I$ in the euclidean topology of $\mathbb R^{r+1}$. 
This way $\partial \mathcal D^{r+1}$ is the union of countably many affine euclidean spaces of dimension $r$.
Finally then, we define
$\mathcal D_{\partial}^{r}$ as the family of $r$-dimensional dyadic cubes in $\partial \mathcal D^{r+1}$.

\section{Measure, kernel, and operator.}

\subsection{Non-homogeneous measure}

We describe the class of measures for which the theory applies. 

\begin{definition}\label{suitable}
Let $\mu $ be a Radon measure on $\mathbb R^n$, which without loss of generality we assume to be positive. 

We say that $\mu$ has power growth if there is $0<\alpha \leq n$ such that 
$\mu(I)\lesssim \ell(I)^\alpha$ for 
all $I\in \mathcal{C}$.

We now define three densities of the measure: 
for $I\in \mathcal{C}$, let 
$$
\rho(I)=\frac{\mu(I)}{\ell(I)^\alpha},
$$
$$
\displaystyle{\rho_{\rm in}(I)
=\sup_{\substack{t\in I\\ 0<\lambda<\ell(I)}}
 \frac{\mu(I\cap B(t,\lambda))}{\lambda^{\alpha}}
 =\sup_{\substack{t\in I\\ \lambda>0}}
 \frac{\mu(I\cap B(t,\lambda))}{\lambda^{\alpha}}}
 $$
where $B(t,\lambda)=\{ x\in \mathbb R^n/|t-x|<\lambda \} $, 
and given $0<\delta \leq 1$,  
$$
\rho_{\rm out}(I)=
\sum_{m\geq 1}\frac{\mu(mI)}{\ell(mI)^{\alpha }}
\frac{1}{m^{\frac{\delta}{2}+1}}.
$$
With them, we denote 
\begin{equation}\label{suitable2}
\rho_{\mu}(I)=\rho_{\rm in}(I)+\rho_{\rm out}(I). 
\end{equation}
\end{definition}

\begin{remark}\label{measureconstants} In the definition of $\rho_{\rm in}$ one can substitute the balls $B(t,\lambda)$ by dyadic cubes just by taking the smallest $Q\in \mathcal D(I)$ with $B(t,\lambda)\subset Q$.

The sum in the definition of $\rho_{\rm out}$ is comparable to 
$$
\int_{1}^\infty \frac{\mu(tI)}{\ell(tI)^{\alpha }}
\frac{dt}{t^{\frac{\delta}{2}+1}}
\approx 
\sum_{k\geq 0}2^{-k\frac{\delta}{2}}
\frac{\mu(2^kI)}{\ell(2^kI)^{\alpha }}.
$$

\end{remark}



If $\mu$ satisfies the power growth for $n$ dimensional cubes on $\mathcal D$, then it satisfies the same power growth for $r$ dimensional cubes on $\mathcal D_{\partial}^r$. To show this, 
we note that each $r$ dimensional dyadic cube $I\in \mathcal D_{\partial}^r$ is in the border of an $n$-dimensional dyadic cube $Q\in \mathcal D$ with the same side length and then
$
\mu(I)\leq \mu(Q)\lesssim \ell(Q)^{\alpha}=\ell(I)^{\alpha}
$.

\subsection{Compact Calder\'on-Zygmund kernel and its  associated operator} We now describe the class of kernels and operators for which the theory applies. 


\begin{definition}\label{prodCZoriginal}
Let $\mu$ be a positive Radon measure on $\mathbb R^n$ with power growth $0<\alpha \leq n$.

A function $K:(\mathbb R^{n}\times \mathbb R^{n}) \setminus 
\{ (t, x)\in \mathbb R^{n}\times \mathbb R^{n} : t=x\} \to \mathbb C$ is a Calder\'on-Zgymund kernel if 
it is bounded on compact sets of its domain
and
there exist $0<\delta \leq 1$ 
and bounded functions $L,S,D:[0,\infty)\rightarrow [0,\infty )$ 
satisfying 
\begin{equation}\label{smoothcompactCZ}
|K(t,x)-K(t',x')|
\lesssim
\Big(\frac{(|t-t'|+|x-x'|)}{|t-x|}\Big)^{\delta}\frac{F(t,x)}{|t-x|^\alpha },
\end{equation}
with $F(t,x)=L(|t-x|)S(|t-x|)D(|t+x|)$, 
whenever $2(|t-t'|+|x-x'|)<|t-x|$.


We 
say that $Kd\mu\times d\mu$ is a compact Calder\'on-Zygmund kernel
if \eqref{smoothcompactCZ} holds 
and
\begin{align}\label{limits}
\lim_{\ell(I)\rightarrow \infty }L(\ell(I))\rho_{\mu}(I)
&=\lim_{\ell(I)\rightarrow 0}S(\ell(I))\rho_{\mu}(I)
\\
\nonumber
&=\lim_{\rdist(I,\mathbb B)\rightarrow \infty }D(\rdist(I,\mathbb B))\rho_{\mu}(I)=0.
\end{align}


\end{definition}

\begin{remark}\label{constants}
Since a dilation of a function satisfying any of the limits in (\ref{limits}) satisfies the same limit, namely
${\mathcal D}_{\lambda}(L\rho_{\mu})(a)=L(\lambda^{-1}a)\rho_{\mu}(\lambda^{-1}a)$ also satisfies the first limit, we omit universal constants in the argument of the functions.
\end{remark}

\begin{notation}\label{LSDF}

Given three cubes $I_{1},I_{2},I_{3}\in \mathcal C$, 
we denote
$$
F(I_{1}, I_{2}, I_{3})=L(\ell(I_{1}))S(\ell(I_{2}))D(\rdist(I_{3},\mathbb B))
$$
and $F(I)=F(I,I,I)$.
Then the limits in \eqref{limits} can be written as
$$
\lim_{M\rightarrow \infty }
\sup_{I\in \mathcal \mathcal D_M^{c}}
F(I)\rho_{\mu}(I)=0.
$$

Given $0<\delta \leq 1$, we define
\begin{equation}\label{Dtilde}
\tilde{D}
(I)=
\sum_{k\geq 0}2^{-k\frac{\delta}{2}}
D(\rdist(2^{k}I,\mathbb B)).
\end{equation}

We note that if $D$ satisfies \eqref{limits}, then by Lebesgue's Domination Theorem so does $\tilde{D}$.

\end{notation}

In \cite{V} it was proved, only for the one-dimensional case and when $\alpha =n=1$, that the smoothness condition
\eqref{smoothcompactCZ} and the mild 
assumption $\lim_{|t-x|\rightarrow \infty} K(t,x)=0$ 
imply the following pointwise decay condition:  
\begin{equation}\label{kerneldecay}
|K(t,x)|
\lesssim
\frac{F(t, x)}{|t-x|^{\alpha}},
\end{equation}
with $F(t,x)=L(|t-x|)S(|t-x|)D(|t+x|)$.
This is also the case when
$F\equiv 1$.

\begin{definition}\label{intrep}
A linear operator $T$ is associated with 
a Calder\'on-Zygmund kernel $K$ if the following representation holds
\begin{equation}\label{kernelrep}
Tf(x) =\int_{\R^{n}} f(t) K(t,x)\, d\mu(t) 
\end{equation}
for all functions $f$ bounded and compactly supported,  
and $x\notin \supp f$.
\end{definition}
By 
\eqref{kerneldecay} and the properties of $f$ and $x$, 
the double integral is absolutely convergent with
$$
\int_{\R^{n}} |f(t) K(t,x)\, |d\mu(t) \lesssim 
\|f\|_{L^\infty (\mu)}\frac{\mu(\supp f)}{\dist(x,\supp f)^\alpha}.
$$

\section{Statements of main results}\label{statemainres}


We denote $\delta(x)=1$ if $x=0$ and 
$\delta(x)=0$ otherwise. 

\begin{theorem}\label{Mainresultrestrictedbddness2}
Let $\mu$ be a positive Radon measure on $\mathbb R^n$ with
power growth $0<\alpha \leq n$.
Let
$T$ be a linear operator with a  Calder\'on-Zygmund kernel and measure $\mu $ as in \eqref{kernelrep}. 
Let $1<p<\infty$, $k=n-[\alpha ]+\delta(\alpha -[\alpha])$.

Then the following statements are equivalent:

a) $T$ extends to a bounded operator on $L^p(\mu )$ 

b) there exist $k$ grids of $n$-dimensional cubes, $\mathcal T_i\mathcal D$  
as in \eqref{grids} 
with $i\in \{1,2, \ldots, k\}$,  
such that the testing condition 
\begin{equation}\label{restrictcompact1-2bddness}
\| \chi_IT\chi_I\|_{L^{2}(\mu)}
+\| \chi_IT^{*}\chi_I\|_{L^{2}(\mu)}
\lesssim 
\mu(I)^{\frac{1}{2}}
\end{equation}
holds for all $I\in \mathcal T_i\mathcal D\cup 2\mathcal T_i\mathcal D$. 


c) for 
the $k$ grids $\mathcal{D}^r_\partial$ of $r$-dimensional cubes as defined in \ref{grids2}
with $r\in \{n, n-1, \ldots, n-k+1\}$, $T$ 
satisfies \eqref{restrictcompact1-2bddness}
for all $I\in \mathcal{D}^r_\partial\cup 2\mathcal D^r_\partial$. 
\end{theorem}

Theorem \ref{Mainresultrestrictedbddness2} follows from the proof of the following theorem. 
\begin{theorem}\label{Mainresult2}
Let $\mu$ be a positive Radon measure on $\mathbb R^n$ with
power growth $0<\alpha \leq n$.
Let
$T$ be a linear operator with a  Calder\'on-Zygmund kernel and measure $\mu $ as in \eqref{kernelrep}. 
Let $1<p<\infty$,  
$k=n-[\alpha ]+\delta(\alpha -[\alpha])$. 

Then the following statements are equivalent: 

a) $T$ extends to a compact operator on $L^p(\mu )$ 

b) $Kd\mu\times d\mu$ 
is a compact Calder\'on-Zygmund kernel
and there exist $k$ grids of $n$-dimensional cubes, $\mathcal T_i\mathcal D$ as in \eqref{grids} with $i\in \{1, 2, \ldots, k\}$, 
such that 
\begin{equation}\label{b-averageandbound1limit}
\| \chi_IT\chi_I\|_{L^{2}(\mu)}
+\| \chi_IT^{*}\chi_I\|_{L^{2}(\mu)}\lesssim 
\mu(I)^{\frac{1}{2}}F_T(I)
\end{equation}
for all $I\in \mathcal T_i\mathcal D\cup 2\mathcal T_i\mathcal D$, with $F_T$ bounded and satisfying 
\begin{align}\label{FT}
\lim_{M\rightarrow \infty }\sup_{I
\in (\mathcal T_i\mathcal D)_{M}^{c}\cup (2\mathcal T_i\mathcal D)_{M}^{c}}\hspace{-.2cm}
F_T(I)=0.
\end{align}


c) $Kd\mu\times d\mu$ 
is a compact Calder\'on-Zygmund kernel
and for the $k$ grids of $r$-dimensional cubes, ${\mathcal D}_{\partial}^r$ as  
in \ref{grids2} with
$r\in \{n,n-1, \ldots, n-k+1\}$, we have that $T$ 
satisfies \eqref{b-averageandbound1limit}
for all $I\in \mathcal{D}^r_\partial\cup 2\mathcal D^r_\partial$, 
with $F_T$ bounded and satisfying 
$$
\lim_{M\rightarrow \infty }\sup_{I
\in ({\mathcal D}_{\partial}^r)_{M}^{c}\cup (2{\mathcal D}_{\partial}^r)_M^{c}}F_T(I)=0. 
$$
\end{theorem}

For the proof of Theorem \ref{Mainresult2}, we provide the following notation: 
\begin{notation}
Let
\begin{align}\label{defFmu}
F_K (I,J)&=
L(\ell(I\smland J)S(\ell(I\smland J))
(D(\rdist(\langle I,J\rangle,\mathbb B)) 
\\
&
\nonumber 
+\tilde D(\inrdist(I,J)I\smland J)).
\end{align}

With  $\rho_{\mu}$, $F_K$, $F_T$ as defined in  \eqref{suitable2}, \eqref{defFmu}, and \eqref{FT} respectively, we now denote
$$
F_{\mu}(I,J)=\sup_{\substack{R\subset I\\ S\subset J}} F_K(R,S)\rho_{\mu}(R\smlor S)+F_T(R)+
F_T(S)
$$
and $F_{\mu}(I)=F_{\mu}(I,I)$. 
\end{notation}

When applied to the Cauchy integral operator, 
$$
C_{\mu }(f)(z)=\int \frac{f(w)}{w -z} \, d \mu (w)
$$
we obtain the following result. 
\begin{theorem}\label{Cauchytheo}
Let $\mu$ a positive Radon measure on the complex plane $\mathbb C$ such that 
$\mu(I)\lesssim \ell (I)$
for each $I\in \mathcal{D}$. Let $1<p<\infty $. 
Then the following statements are equivalent:

a) $C_{\mu }$ is bounded on $L^{p}(\mu)$, 

b) there exist two grids of $2$-dimensional cubes, $\mathcal T_i\mathcal D$ with $i\in \{1, 2\}$ as defined in \eqref{grids}, 
such that the testing condition 
\begin{equation}\label{Cmu}
\| \chi_{I}C_{\mu}\chi_{I}\|_{L^2(\mu)}\lesssim \mu(I)^{\frac{1}{2}},
\end{equation}
holds for all $I\in \mathcal T_i\mathcal D\cup 2\mathcal T_i\mathcal D$.

c) for the grids of dyadic squares $\mathcal{D}^2_\partial$ and line segments $\mathcal{D}^1_\partial$
as defined in \ref{grids2}, we have that 
\eqref{Cmu} holds
for all $I\in \mathcal{D}^r_\partial \cup 2\mathcal D^r_\partial$ with $r\in \{2,1\}$.

\vskip5pt
Furthermore, the following statements are also equivalent:

a) $C_{\mu }$ is compact on $L^{p}(\mu)$, 

b) there exist two grids of $2$-dimensional cubes, $\mathcal T_i\mathcal D$ with $i\in \{1, 2\}$ as defined in \eqref{grids}, 
such that \eqref{Cmu} holds and 
$$
\lim_{M\rightarrow \infty }
\sup_{I\in (\mathcal T_i\mathcal D)^{c}_M}
\rho_{\mu}(I)
=\lim_{M\rightarrow \infty }
\sup_{I\in (\mathcal T_i\mathcal D)^{c}_M\cup (2\mathcal T_i\mathcal D)_M^{c}}
\frac{\| \chi_{I}C_{\mu}\chi_{I}\|_{L^2(\mu)}}{ \mu(I)^{\frac{1}{2}}}
=0,
$$

c) for the grids of dyadic squares $\mathcal{D}^2_\partial$ and line segments $\mathcal{D}^1_\partial$
as defined in \ref{grids2}, we have
that \eqref{Cmu} holds and 
$r\in \{2,1\}$
$$
\lim_{M\rightarrow \infty }
\sup_{I\in ({\mathcal D}_{\partial}^r)_{M}^{c}}
\rho_{\mu}(I)
=\lim_{M\rightarrow \infty }\sup_{I\in ({\mathcal D}_{\partial}^r)^{c}_M\cup (2{\mathcal D}_{\partial}^r)^{c}_M}
\frac{\| \chi_{I}C_{\mu}\chi_{I}\|_{L^2(\mu)}}{ \mu(I)^{\frac{1}{2}}}
=0.
$$

\end{theorem}

\section{The truncated operators}\label{truncated} In this section 
we define and study the properties of a particular smooth truncation of Calder\'on-Zygmund operators. 
We start the section with a technical result. 
\begin{lemma}\label{integrable}
Let 
$I\in \mathcal D$, and $x\in I$. Then
\begin{align*}
\int_{I}
\frac{1}{|t-x|^{\alpha -1}}d\mu(t)
\lesssim \ell(I)\rho_{\rm in}(I).
\end{align*}
\end{lemma}
\begin{proof}
For $k\geq 0$, let 
$S_k=\{t\in I/ |t-x|\leq 2^{-k}\ell(I)\}$ and
$C_k=S_{k}\setminus S_{k+1}=\{t\in I/ 2^{-(k+1)}\ell(I)<|t-x|\leq 2^{-k}\ell(I)\}$. Then
\begin{align*}
\int_{I}
\frac{1}{|t-x|^{\alpha -1}}d\mu(t)
&\leq \sum_{k\geq 0}\frac{2^{(\alpha -1) (k+1)}}{\ell(I)^{\alpha-1}}
\mu(S_{k}\setminus S_{k+1})
\\
&\lesssim \ell(I)^{1-\alpha}\sum_{k\geq 0}2^{(\alpha -1) k}
(\mu(S_{k})-\mu(S_{k+1})).
\end{align*}
We write $a_{k}=\mu(S_k)$ and 
evaluate the previous expression using 
Abel's formula: for large $R$ we have 
\begin{align*}
\ell(I)^{1-\alpha}
\sum_{k=0}^{R} 
2^{(\alpha -1) k}(a_{k}-a_{k+1})
&=\ell(I)^{1-\alpha}\Big(a_{0}
-
a_{R+1}2^{(\alpha -1) R}
\\
&
+
\sum_{k=1}^{R}
a_{k}(2^{(\alpha -1) k}-2^{(\alpha -1) (k-1)})\Big).
\end{align*}

Since $a_{0}\leq \mu(I)=
\rho(I)\ell(I)^{\alpha}\leq \rho_{\rm in}(I)\ell(I)^{\alpha}$,  
for the first term we have
$\ell(I)^{1-\alpha}a_{0}
\leq \ell(I)\rho_{\rm in}(I).
$

Similarly, 
since $$a_{R+1}=\mu(S_{R+1})\leq 
\mu(I\cap B(x, 2^{-(R+1)}\ell(I)))\leq 
\rho_{\rm in} (I) 2^{(R+1)\alpha }\ell(I)^{\alpha},$$
the absolute value of the second term can be bounded by
$$
\ell(I)^{1-\alpha}a_{R+1}2^{(\alpha -1) R}
\lesssim \ell(I)\rho_{\rm in}(I) 2^{-R}\leq \ell(I)\rho_{\rm in}(I) .
 $$

Meanwhile, since since $a_{k}=\mu(S_{k})\leq 
\rho_{\rm in}(I) 2^{k\alpha }\ell(I)^{\alpha}$,
the absolute value of the last term is bounded by 
\begin{align*}
\ell(I)^{1-\alpha}
&
\sum_{k=1}^{R}
a_{k}2^{(\alpha -1) k}(1-2^{-(\alpha-1)})
\lesssim \ell(I)\rho_{\rm in}(I)\sum_{k=1}^{R}
2^{-k}\lesssim \ell(I)\rho_{\rm in}(I).
\end{align*}
\end{proof}

We now define the following smooth truncation of an operator associated with a Calder\'on-Zygmund kernel. 
\begin{definition}\label{TgammaQ}
Let $\phi$ be a smooth function such that 
$0\leq \phi(x)\leq 1$, $\sup \phi\subset [-2,2]$,  
$\phi(x)=1$ for all $|x|< 1$ and $0\leq \phi'(x)\leq 2$.

Let $Q=[-2^N,2^N]^n$ with $N\geq 0$
and $0<\gamma \leq 1$. 
We define the kernel 
\begin{align*}
K_{\gamma,Q}(t,x)&=
K(t,x)
(1-\phi(\frac{|t-x|}{\gamma}))
\phi(\frac{4|t|}{\ell(Q)})
\phi(\frac{4|x|}{\ell(Q)}).
\end{align*}
Let 
$T_{\gamma, Q}$ the operator with kernel $K_{\gamma,Q}$. 
\end{definition}

In the next two results, we prove that the truncated operators are bounded, have a Calder\'on-Zygmund kernel and satisfy the testing condition uniformly on $\gamma $ and $Q$. 
\begin{lemma}\label{truncaT}
The operator $T_{\gamma, Q}$ is bounded with bounds  depending on $\gamma $ and $Q$. Moreover,  
$K_{\gamma,Q}$ is a Calder\'on-Zygmund kernels with parameter $0<\delta\leq 1$ and constant independent of $\gamma $ and $Q$. 
\end{lemma}
\begin{proof}
We first show that $K_{\gamma,Q}$ is a bounded function: by 
\eqref{kerneldecay}, 
$$
|K_{\gamma,Q}(t,x)|\lesssim 
\frac{1}{|t-x|^\alpha}(1-\phi(\frac{|t-x|}{\gamma}))
\leq \frac{1}{\gamma^\alpha}.
$$
The last inequality is due to the fact that when $|t-x|\leq \gamma $ we have $\phi(\frac{|t-x|}{\gamma})=1$ and so, the second factor is zero. Then, since 
$K_{\gamma,Q}$ is supported on $Q\times Q$, by Cauchy-Schwarz, we have for $f,g\in L^2(\mu)$ 
\begin{align*}
|\langle T_{\gamma, Q}f,g\rangle|
&=|\int \int K_{\gamma ,Q}(t,x)f(t)g(x)d\mu(t)d\mu(x)|
\\
&\leq \Big(\int_{Q} \int_{Q} |K_{\gamma ,Q}(t,x)|^2d\mu(t)d\mu(x)\Big)^{\frac{1}{2}}
\| f\|_{L^2(\mu)}\| g\|_{L^2(\mu)}
\\
&\lesssim \frac{\mu(Q)}{\gamma^{\alpha}}
\| f\|_{L^2(\mu)}\| g\|_{L^2(\mu)}.
\end{align*}
This proves that the operators $T_{\gamma, Q}$  are bounded.

We now show that $K_{\gamma,Q}$ is a Calder\'on-Zygmund kernel. We prove the appropriate estimate for 
 $A=|K_{\gamma,Q}(t,x)-K_{\gamma,Q}(t',x)|$ being similar the work for  $|K_{\gamma,Q}(t,x)-K_{\gamma,Q}(t,x')|$. 
 Let $t,t',x$ such that $2|t-t'|<|t-x|$. We note that this inequality implies
$|t-x|\leq |t'-x|+|t-t'|\leq |t'-x|+|t-x|/2$
and so, 
$|t-x|\leq 2|t'-x|$.
 Then
\begin{align*}
A
&
\leq |K(t,x)-K(t',x)|\phi(\frac{4|t|}{\ell(Q)})
\phi(\frac{4|x|}{\ell(Q)})(1-\phi(\frac{|t-x|}{\gamma}))
\\&
+|K(t',x)| |\phi(\frac{4|t|}{\ell(Q)})-\phi(\frac{4|t'|}{\ell(Q)})|
\phi(\frac{4|x|}{\ell(Q)})(1-\phi(\frac{|t-x|}{\gamma}))
\\
&
+|K(t',x)| \phi(\frac{4|t'|}{\ell(Q)})
\phi(\frac{4|x|}{\ell(Q)})
|\phi(\frac{|t-x|}{\gamma})-\phi(\frac{|t'-x|}{\gamma})|.
\end{align*}
Since $2|t-t'|<|t-x|$ we can use the kernel smoothness condition \eqref{smoothcompactCZ} and that $\phi$ is bounded, to bound the first term by a constant times
\begin{align*}
\Big(\frac{|t-t'|}{|t-x|}\Big)^{\delta}\frac{F(t,x)}{|t-x|^\alpha }=\frac{|t-t'|^{\delta}}{|t-x|^{\alpha +\delta}}F(t,x).
\end{align*}
The second term is non zero if $x\in Q$ and $t\in Q$ or $ t'\in Q$. 
Then if $t\in Q$, we have
$|t-x|<|t|+|x|\leq \ell(Q)$.
Meanwhile if $t'\in Q$, we get 
$|t-x|\leq 2|t'-x|\leq 2(|t'|+|x|)
\leq 2\ell(Q)$.
Then, by the kernel decay\eqref{kerneldecay}, the fact that $\phi$ is bounded, and the Mean Value Theorem on $\phi$ with bounded derivative, the second term is bounded by a constant times
\begin{align*}
\frac{F(t,x)}{|t-x|^\alpha }
\frac{|t-t'|}{\ell(Q)}
\lesssim \frac{F(t,x)}{|t-x|^\alpha }
\frac{|t-t'|}{|t-x|}
\leq \frac{|t-t'|^{\delta}}{|t-x|^{\alpha+\delta} }F(t,x).
\end{align*}
The third term is non zero if $t',x\in Q$ and 
$|t-x|<2\gamma$ or $|t'-x|<2\gamma$. In 
the latter case we have  
$|t-x|\leq 2|t'-x|<4\gamma $.
Then, using again the kernel decay, that $\phi $ is bounded and the Mean Value Theorem on $\phi$, we can estimate 
this third term by a constant times
\begin{align*}
\frac{F(t,x)}{|t-x|^\alpha }
\frac{||t-x|-|t'-x||}{\gamma}
\lesssim \frac{F(t,x)}{|t-x|^\alpha }
\frac{|t-t'|}{|t-x|}
\leq \frac{|t-t'|^{\delta}}{|t-x|^{\alpha+\delta} }F(t,x).
\end{align*}



\end{proof}

\begin{lemma}\label{truncaT2}
The operator $T_{\gamma, Q}$ satisfies the testing condition with bounds independent of $\gamma $ and $Q$. 
\end{lemma}
\begin{proof}
From the symmetric expression of the kernel $K_{\gamma ,Q}$ it is clear that the same ideas used to prove the testing condition on $T_{\gamma, Q}$ also work for $T^*_{\gamma ,Q}$. Therefore, we write the computations only for 
$T_{\gamma, Q}$. We are going to show that for 
all $I\in \mathcal D\cup 2\mathcal D$ we have 
$\| \chi_IT_{\gamma ,Q}\chi_I\|_{L^2(\mu)}^2
\lesssim F_{\mu}(I)\mu(I)^{\frac{1}{2}}$ with 
$F_{\mu}$ bounded and such that  
$\lim_{M\rightarrow \infty }\sup_{I
\in \mathcal D_{M}^{c}\cup 2\mathcal D_{M}^{c}}\hspace{-.2cm}
F_{\mu}(I)=0$.

Since $0\leq \phi(x)\leq 1$ we have 
that $\| \chi_IT_{\gamma ,Q}\chi_I\|_{L^2(\mu)}^2$ is bounded by
\begin{align*}
\int_I|\int_I K(t,x)(1-\phi(\frac{|t-x|}{\gamma}))
\phi(\frac{4|t|}{\ell(Q)})
d\mu(t)|^2
\phi(\frac{4|x|}{\ell(Q)})^2
d\mu(x)={\rm Int}
\end{align*}


We first assume that $I\subset 2^{-1}Q$. In this case, 
since  
$\phi(\frac{4|t|}{\ell(Q)})=1$ for all $t\in 2^{-1}Q$, we have 
\begin{align}\label{T1}
{\rm Int}&=\int_I|\int_I K(t,x)(1-\phi(\frac{|t-x|}{\gamma}))
d\mu(t)|^2\phi(\frac{4|x|}{\ell(Q)})^2d\mu(x)
\\
\nonumber
&\lesssim \int_I(\int_{\substack{t\in I\\ \gamma \leq |t-x|\leq 2\gamma }}
|K(t,x)|
d\mu(t))^2d\mu(x)
\\
\nonumber
&+\int_I|\int_{\substack{t\in I\\ 2\gamma \leq |t-x|}}
K(t,x)
d\mu(t)|^2d\mu(x)
\end{align}
We note than from here we do not use 
$I\subset 2^{-1}Q$ anymore, but only $I\in \mathcal D$. 

By the kernel decay, the first term is bounded by a constant times 
\begin{align}\label{diago}
\int_I &(\int_{\substack{t\in I\\ \gamma<|t-x|\leq 2\gamma}} \frac{F(t,x)}{|t-x|^\alpha}
d\mu(t))^2d\mu(x)
\\
\nonumber
&\leq F_K(I)^2\int_I \frac{\mu(I\cap B(x,2\gamma ))^2}{\gamma^{2\alpha} }d\mu(x)
\\
\nonumber
&\leq F_K(I)^2\rho_{\rm in}(I)^2\mu(I)\leq F_{\mu}(I)^2\mu(I)
\end{align}

To deal with the second term, 
we denote
$D_x=\{t\in I/ |t-x|\leq 2\gamma \}$ 
and 
$D_x^c=I\setminus D_x$. Since 
$\chi_{D_x^c}=\chi_I-\chi_{D_x}$, 
the second term can be written as 
\begin{align*}
\int_I|T(\chi_{D_x^c})(x)|^2d\mu(x)
\lesssim 
\int_I|T(\chi_{I})(x)|^2d\mu(x)
+
\int_I|T(\chi_{D_x})(x)|^2d\mu(x)
\end{align*}

The new first term equals 
\begin{align*}
\| \chi_IT\chi_I\|_{L^2(\mu)}^2
\lesssim F_{T}(I)\mu(I)
\leq F_{\mu}(I)\mu(I)
\end{align*}
by the testing condition on $T$. On the other hand, 

We now denote
$D=\{(t,x)\in I\times I/ |t-x|\leq 2\gamma \}$. 
Then
we can rewrite the second term as 
\begin{align*}
\| \chi_I\int & K(t,\cdot )
\chi_{D}(t,\cdot )
d\mu(t)\|_{L^2(\mu)}^2.
\end{align*}
We are going to prove that 
\begin{align*}
\| \chi_I\int & K(t,\cdot )
\chi_{D}(t,\cdot )
d\mu(t)\|_{L^2(\mu)}
\lesssim F_{\mu}(I)\mu(I)^{\frac{1}{2}}
\end{align*}
or, equivalently, that 
for every $\Phi_I\in L^2(\mu)$ with support on $I$ and 
$\| \Phi_I\|_{L^2(\mu)}\leq \mu(I)^{\frac{1}{2}}$, we have
\begin{align*}
|\langle \Phi_I, \int & K(t,\cdot )
\chi_{D}(t,\cdot )
d\mu(t)\rangle |\lesssim F_{\mu}(I)\mu(I).
\end{align*}

Let $I_i\in \mathcal D(I)$ such that 
$I_i\times I_i\subset D$ is maximal inside $D$ with respect the inclusion. Therefore, $\ell(I_i)$ is the same for all cubes and comparable to $\gamma $. 
Let $I_{i,p}$ the parent of $I_i$ and let $2I_{i,p}\in \mathcal C$ such that $c(2I_{i,p})=c(I_{i,p})$ and 
$\ell(2I_{i,p})=2\ell(I_{i,p})$. 
Then there are alternating sub-collections of bi-cubes 
$I_{i,p}\times I_{i,p}$ and $2I_{i,p}\times 2I_{i,p}$
such that 
completely cover $D$ and a set $D'\subset \{(t,x)\in I\times I)/ \gamma <|t-x|\leq 
8\gamma \}$,  
the cubes $2I_{i,p}$ are pairwise disjoint, 
and the intersection of consecutive $I_{i,p}$ and $2I_{i,p}$ is contained on  $I_i$. Therefore, we can write
\begin{align*}
\chi_{D}(t,x)&=\sum_{i\in \mathcal O}\chi_{I_{i,p}}(t)\chi_{I_{i,p}}(x)+\sum_{i\in \mathcal E}\chi_{2I_{i,p}}(t)\chi_{2I_{i,p}}(x)
\\
&
-
\hspace{-.3cm}\sum_{\substack{i / I_i\subset I_{j,p}\\ j\in  \mathcal O}}\hspace{-.3cm}\chi_{I_i}(t)
\chi_{I_i}(x)-\chi_{D'}(t,x).
\end{align*}
where $\mathcal O$, $\mathcal E$ denote sets of indexes in $\mathbb Z^n$. 
With this,
we have 
\begin{align*}
 | \langle \Phi_I, \int K(t,\cdot )
 &\chi_{D}(t,\cdot )
 d\mu(t)\rangle |
\leq \sum_{i\in \mathcal O}|\langle \Phi_I\chi_{I_{i,p}}, T\chi_{I_{i,p}}\rangle |
\\
&
 +\sum_{i\in \mathcal E}|\langle \Phi_I\chi_{2I_{i,p}}, T\chi_{2I_{i,p}}\rangle |
 +\sum_{\substack{i / Q_i\subset Q_{j,p}\\ j\in \mathcal O}}
 |\langle \Phi_I\chi_{I_{i}}, T\chi_{I_{i}}\rangle |
 \\&
 +\| \Phi_I\|_{L^2(\mu)}\|\int K(t,\cdot )\chi_{D'}(t,\cdot )d\mu(t)\|_{L^2(\mu)}.
\end{align*}
Now, as we did in \eqref{diago}, we can estimate the 
last term by a constant times
\begin{align*}
\mu(I)^{\frac{1}{2}}\Big(\int_I(\int_{\substack{t\in I\\ 2\gamma<|t-x|\leq 8\gamma}}
|K(t,x)|d\mu(t))^2d\mu(x)\Big)^{\frac{1}{2}}
\lesssim 
F_{\mu}(I)\mu(I).
\end{align*}
On the other hand, by the testing condition for $T$ on the cubes $I_{i,p}$ and Cauchy's inequality, the first term is bounded by  
\begin{align*}
\sum_{i\in \mathcal O}
\| \Phi_I&\chi_{I_{i,p}} \|_{L^2(\mu)}
\| \chi_{I_{i,p}}T\chi_{I_{i,p}}\|_{L^2(\mu)}
\\
&\lesssim \sum_{i\in \mathcal O}
\| \Phi_I\chi_{I_{i,p}}  \|_{L^2(\mu)}
F_{T}(I_{i,p})\mu(I_{i,p})^{\frac{1}{2}}
\\
&\leq \Big( \sum_{i\in \mathcal O}
\| \Phi_I\chi_{I_{i,p}}  \|_{L^2(\mu)}^2\Big)^{\frac{1}{2}}
\Big(\sum_{i\in \mathcal O}
\mu(I_{i,p})\Big)^{\frac{1}{2}}F_{\mu}(I)
\leq F_{\mu}(I)\mu(I).
\end{align*}
The last inequality is due to the fact that, since the cubes 
$I_{i,p}\subset I$
have the same size length, they are pairwise disjoint. 
Then
$$
\sum_{i \in \mathcal O}
\| \Phi_I\chi_{I_{i,p}} \|_{L^2(\mu)}^2
=\sum_{i \in \mathcal O}\int_{I_{i,p}}|\Phi_I(x)|^2d\mu(x)
\leq \| \Phi_I \|_{L^2(\mu)}^2\leq \mu(I).
$$
Similar computations using the testing condition for $T$ on the cubes $2I_{i,p}$ (which are also pairwise disjoint) and $I_i$ respectively, show the same inequality for the second and third terms. 

This finishes the proof under the assumption that 
$I\subset 2^{-1}Q$. For the general case, since 
$K_{\gamma ,Q}$ is supported on $Q$ we assume without loss of generality that $I\cap Q\neq \emptyset $. 
Let $I'\in \mathcal D$ the smallest dyadic cube such that
$I\cap Q\subset I'$. Since $I\in \mathcal D$, by minimality we have 
$I'\subset I$. Let also $R$ the quadrant 
of $\mathbb R^n$ such that $I\subset R$ and 
let $Q'\in \mathcal D$ the smallest dyadic cube such that $Q\cap R\subset Q'$. Clearly $\ell(Q')\leq 2\ell(Q)$. 
Moreover, 
since 
$I\cap Q\subset R\cap Q\subset Q'$ we have by minimaliy that 
$\ell(I')\leq \ell(Q')\leq 2\ell(Q)$. 

Now, by the Mean Value Theorem, there exists $\xi \in (0,\frac{4||t|-|x||}{\ell(Q)})$ 
such that 
 \begin{align}
 \phi(\frac{4|t|}{\ell(Q)})
 &=(\phi(\frac{4|x|}{\ell(Q)})-\phi'(\xi)
 \frac{4||t|-|x||}{\ell(Q)})\chi_{I'}(t),
 %
 \end{align}
 since $\phi(\frac{4|t|}{\ell(Q)})\leq \chi_{I'}(t)$ 
for $t\in I\cap Q$. 
Then we can write 
\begin{align*}
{\rm Int}&=
\int_{I}|\int_{I} K(t,x)(1-\phi(\frac{|t-x|}{\gamma}))
\phi(\frac{4|t|}{\ell(Q)})
d\mu(t)|^2\phi(\frac{4|x|}{\ell(Q)})^2
d\mu(x)
\\
&
\lesssim \int_{I'}|\int_{I'} 
K(t,x)(1-\phi(\frac{|t-x|}{\gamma}))d\mu(t)|^2
\phi(\frac{2|x|}{\ell(Q)})^3d\mu(x)
\\
&+\int_{I'}|\int_{I'} 
K(t,x)(1-\phi(\frac{|t-x|}{\gamma}))
\phi'(\xi)
 \frac{4||t|-|x||}{\ell(Q)}
d\mu(t)|^2
d\mu(x)
\end{align*}
The first term is bounded by 
$$
\int_{I'}|\int_{I'} 
K(t,x)(1-\phi(\frac{|t-x|}{\gamma}))d\mu(t)|^2
d\mu(x),
$$
which satisfies the right estimates as it was proved in the previous case, when $I\subset Q$, starting at \eqref{T1}. We remind that 
\eqref{T1} we did not use $I'\subset 2^{-1}Q$ anymore, but 
$I'\in \mathcal D$. 

On the other hand, 
since $1\leq \phi(x)\leq 1$, $0\leq \phi'(x)\leq 2$, $\ell(Q)\geq 1$ and 
$\phi(x)=1$ for $|x|\leq 1$, the second term is bounded by a constant times 
\begin{align*}
\int_{I'}(
&\int_{I'}
|K(t,x)||1-\phi(\frac{|t-x|}{\gamma})|
\frac{||t|-|x||}{\ell(Q)}d\mu(t))^2
d\mu(x)
\\
&\lesssim \int_{I'}(
\frac{1}{\ell(Q)}\int_{\substack{t\in I'\\ |t-x|>\gamma }}
|K(t,x)||t-x|d\mu(t))^2
d\mu(x)
\\
&\lesssim \int_{I'}(
\frac{1}{\ell(Q)}\int_{\substack{t\in I'\\ |t-x|>\gamma }}
\frac{F(t,x)}{|t-x|^{\alpha }}|t-x|d\mu(t))^2
d\mu(x)
\\
&\lesssim F_K(I)^2\int_{I'}(
\frac{1}{\ell(Q)}\int_{\substack{t\in I\cap 2Q\\ |t-x|>\gamma }}
\frac{1}{|t-x|^{\alpha -1}}d\mu(t))^2
d\mu(x)
\end{align*}     
By Lemma \ref{integrable}, the last expression is bounded by a constant times 
\begin{align*}
F_K(I)^2\int_{I'}(\frac{1}{\ell(Q)}\rho_{\rm in}(I')\ell(I'))^2
d\mu(x)
&\lesssim F_K(I)^2\rho_{\rm in}(I)\mu(I)
\\
&
\lesssim F_\mu(I)^2\mu(I),
\end{align*}     
where we used that $\ell(I')\lesssim \ell(Q)$ and 
that, since $I'\subset I$, we have 
$\rho_{\rm in}(I')\leq \rho_{\rm in}(I)$. 

\end{proof}

\section{Haar wavelet systems and the characterization of compactness.}

\subsection{The Haar wavelet system}
%
%
%
\begin{definition} 
Let $\mu$ be a measure on $\mathbb R^n$. 
For $Q\in \mathcal D\cup \tilde{\mathcal D}$ with $\mu(Q)\neq 0$
we denote the average $\langle f\rangle_Q=\mu(Q)^{-1}\int_Qf(x)d\mu(x)$. 
For $Q\in \mathcal D$ with $\mu(Q)=0$, we set
$\langle f\rangle_Q=0$.

We define the averaging operator by $
E_{Q}f=\langle f\rangle _{Q}\chi_{Q}
$ and  the difference operator by
\begin{equation}\label{Deltaincoord}
\Delta_{Q}f=\Big(\sum_{I\in \child(Q)}E_{I}f\Big)-E_{Q}f
=\hspace{-.3cm}\sum_{I\in \child(Q)}\Big(\langle f\rangle _{I}
-\langle f\rangle _{Q}\Big)\chi_{I}.
\end{equation}

For $k\in \mathbb Z$, 
we define
$$
{\displaystyle E_{k}f=\hspace{-.3cm}
\sum_{\substack{Q\in \mathcal D\\\ell(Q)=2^{-k}}}
\hspace{-.2cm}E_{Q}f}
,
\hskip15pt {\it and } 
\hskip20pt
{\displaystyle \Delta_{k}f=E_{k}f-E_{k-1}f
=\hspace{-.3cm}\sum_{\substack{Q\in \mathcal D\\\ell(Q)=2^{-k}}}\hspace{-.2cm}\Delta_{Q}f}. 
$$
\end{definition}

\begin{definition}[Haar wavelets]\label{defpsi}
Let $I\in \mathcal {D}\cup \tilde{\mathcal D}$. For $\mu(I)\neq 0$ we define the Haar wavelet function associated with $I$ by
\begin{align*} 
\psi_{I}&=\mu(I)^{\frac{1}{2}}\big(
\frac{1}{\mu(I)}\chi_{I}-
\frac{1}{\mu(I_p)}\chi_{I_{p}}\big),
\end{align*}
where $I_p\in \mathcal{D}$ is such that $I\in \child(I_p)$.
For $\mu(I)=0$ we set $\psi_I=0$.

\end{definition}



\begin{lemma}\label{adaptedwavelets} For $Q\in \mathcal{D}\cup \tilde{\mathcal D}$ and $f, g$ 
locally integrable, we have
\begin{align*}
\Delta_{Q}f 
&=\sum_{I\in \ch(Q)}\langle f, \psi_{I}\rangle 
\psi_{I}
\end{align*}
almost everywhere with respect to $\mu$. 
\end{lemma}
\proof
If $\mu(Q)=0$ then $\mu(I)=0$ for every 
$I\in \ch(Q)$. Witht this, both $\Delta_{Q}=0$ and 
$\psi_I=0$ and so, the equality is trivial.

For $\mu(Q)\neq 0$, 
from \eqref{Deltaincoord} 
and 
${\displaystyle \mu(Q)\langle f\rangle _{Q}=
\sum_{I\in \ch(Q)}\mu(I)\langle f\rangle _{I}}$ 
we have
\begin{align}\label{delta}
\nonumber
\Delta_{Q}f
&=\sum_{I\in \ch(Q)}\langle f\rangle _{I}\chi_{I}-\langle f\rangle_{Q}\chi_{Q}
=\sum_{I\in \ch(Q)}
\langle f\rangle _{I}\Big(\chi_{I}
-\frac{\mu(I)}{\mu(Q)}\chi_{Q}\Big)
\\
&
=\sum_{I\in \ch(Q)}\mu(I)^{\frac{1}{2}}\langle f\rangle _{I}\psi_{I},
\end{align}
where the last equality holds even for those terms for which $\mu(I)=0$ since in that case $\langle f\rangle _{I}=0$. 

Also from \eqref{Deltaincoord} we have for each $I\in \child(Q)$, 
\begin{equation}\label{Deltaaverage}
\langle \Delta_{Q}f\rangle_{I}=\langle f\rangle _{I}-\langle f\rangle _{Q}
\end{equation}
and so 
\begin{align}
\nonumber
\Delta_{Q}f&
=\sum_{I\in \ch(Q)}\mu(I)^{\frac{1}{2}}\langle \Delta_{Q}f\rangle_{I}\psi_{I}
+\langle f\rangle _{Q}\sum_{I\in \ch(Q)}\mu(I)^{\frac{1}{2}}\psi_{I}.
\end{align}
For the first term, we compute the coefficients: 
for $\mu(I)=0$, we have $\psi_I=0$ and so, 
$\mu(I)^{\frac{1}{2}}\langle \Delta_{Q}f\rangle_{I}\psi_{I}=0=\langle f, \psi_{I}\rangle \psi_{I}$. Meanwhile, 
for $\mu(I)\neq 0$, we can use \eqref{Deltaaverage} to write
$$
\mu(I)^{\frac{1}{2}}\langle \Delta_{Q}f\rangle _{I}
=\mu(I)^{\frac{1}{2}}\int f(x) \Big(\frac{\chi_{I}(x)}{\mu(I)}-\frac{\chi_{Q}(x)}{\mu(Q)}\Big)d\mu(x)
=\langle f, \psi_{I}\rangle .
$$

We now denote by $Q'$ the union of cubes $I\in \ch(Q)$ such that $\mu(I)\neq 0$. Then for the second term we have
 \begin{align}\label{aezero}
 \nonumber
 \sum_{I\in \ch(Q)}\mu(I)^{\frac{1}{2}}\psi_{I}
 &=
\sum_{\substack{I\in \ch(Q)\\\mu(I)\neq 0}}\mu(I)^{\frac{1}{2}}
 \psi_{I}
 =
 \sum_{\substack{I\in \ch(Q)\\\mu(I)\neq 0}}\Big(\chi_{I}-
 \frac{\mu(I)}{\mu(Q)}\chi_{Q}\Big) 
 \\
 &
   =\chi_{Q'}-\frac{\mu(Q')}{\mu(Q)}\chi_{Q}
 =-\chi_{Q\setminus Q'}=0
 \end{align}
almost everywhere since $\mu(Q\setminus Q')=0$. 

\begin{lemma}\label{densityinL2}
Let $f$ bounded, compactly supported 
and with mean zero with respect to $\mu$. Then 
\begin{equation}\label{representationoff}
\int f(x)g(x)d\mu(x)=
\lim_{M\rightarrow \infty }\int 
\sum_{\substack{I\in \mathcal D\\ 2^{-M}\leq \ell(I)\leq  2^{M}}}\langle f, \psi_{I}\rangle 
\psi_{I}(x)g(x)d\mu(x)
\end{equation}
for 
$g$ bounded and compactly supported. 

\end{lemma}
%
\begin{proof}
By dividing $f$ into up to $2^n$ functions if necessary, we can assume that $\sup f$ and $\sup g$ are contained in one quadrant of $\mathbb R^n$. Then
we can choose $S\in {\mathcal D}$ with $\sup f\cup \sup g\subset S$.

We first note that the function inside the integral in the right hand side of \eqref{representationoff} can be written as 
$$
\sum_{\substack{I\in \mathcal D
\\ 2^{-M}\leq \ell(I)\leq 2^{M}}}\langle f, \psi_{I}\rangle 
\psi_{I}
=\sum_{\substack{Q\in \mathcal D
\\ 2^{-(M-1)}\leq \ell(Q)\leq 2^{M+1}}}\sum_{I\in \ch(Q)}\langle f, \psi_{I}\rangle 
\psi_{I}.
$$

By Lemma \ref{adaptedwavelets}, for every $g$ bounded and compactly supported on $S$, we have $\mu$-almost everywhere that 
$$
\sum_{I\in \ch(Q)}\langle f, \psi_{I}\rangle 
\psi_{I}
=\Delta_{Q}f.
$$ 
Then
we can write the right hand side of \eqref{representationoff} as 
\begin{align*}
\lim_{M\rightarrow \infty } 
&\int 
\sum_{\substack{Q\in \mathcal D(S)\\ 2^{-(M-1)}\leq \ell(Q)\leq 2^{M+1}}}
\sum_{I\in \ch(Q)}\langle f, \psi_{I}\rangle 
\psi_{I}(x)g(x)d\mu(x) 
\\
&
=\lim_{M\rightarrow \infty } \int 
\sum_{-M\leq k\leq M}\Delta_{k}f(x)g(x)
d\mu(x).
\end{align*}
Now 
we can choose $M\in \mathbb N$ such that
$2^{-M}\leq \ell(S)\leq 2^{M+1}$. 
And for $x\in S$,
we select $I,J\in \mathcal D$ such that $x\in J\subset S\subset I$, 
$\ell(J)=2^{-M}$, and $\ell(I)=2^{M+1}$. Then, by summing a telescopic series, we get 
\begin{align*}
\sum_{-M\leq k\leq M}\hspace{-.5cm}\Delta_{k}f(x)
&=E_{M}f(x)-E_{-(M+1)}f(x)
\\
&
=\langle f\rangle_{J}\chi_{J}(x)-\langle f\rangle_{I}\chi_{I}(x)
=\langle f\rangle_{J}\chi_{J}(x).
\end{align*}
since $S\subset I$ and $f$ has mean zero.

With this
\begin{align*}
\lim_{M\rightarrow \infty } 
\int 
\sum_{-M\leq k\leq M}\Delta_{k}f(x) g(x)
d\mu
&
=\lim_{M\rightarrow \infty } \int_S 
E_{M}f(x) g(x)
d\mu(x).
\end{align*}
Since $f$ is locally integrable, by Lebesgue's Differentiation Theorem we have that 
$E_{M}f$
converges to $f$
pointwise almost everywhere with respect to $\mu$ when $M$ tends to infinity. 
Moreover, 
since $|E_{M}f(x) g(x)|\lesssim \| f\|_{L^\infty (\mu)}
\| g\|_{L^\infty (\mu)}\chi_{S}(x)$, 
we can use Lebesgue's Dominated Convergence Theorem to 
conclude the result. 
\end{proof}

Similar work shows the validity of the following result:
\begin{lemma}\label{densityinL22}
We remind that $\tilde{\mathcal D}$ denotes the family of open dyadic cubes and $\partial \mathcal D$ denotes the union of the borders of all dyadic cubes. 
For $I\in \mathcal D$, we denote
$\tilde I=I\setminus \partial I\in \tilde{\mathcal D}$. 

Let $f$ be integrable, compactly supported, and with mean zero with respect to $\mu$.
We denote $f_1=f-f\chi_{\partial \mathcal D}$. Then  
the equality 
\begin{align*}
f_1=\sum_{I\in \tilde{\mathcal D}}\langle f, \psi_{I}\rangle
\psi_{\tilde I}
\end{align*}
holds  
$\mu$-almost everywhere.
\end{lemma}

\subsection{A variation of the Haar wavelet system}
We now define a new Haar wavelet system and show that the analog of lemma \ref{adaptedwavelets} 
holds.
These wavelets will be used when dealing with the paraproducts.
\begin{definition}\label{fullwavelet}
Let $Q, J_p\in \mathcal{D}$ and we denote  $c_{J_p}=c(J_p)$. For $I\in \mathcal D$ with 
$\mu(I)\neq 0$, we define
$$
\psi_{I,J_p}^{{\rm full}}(t)=\mu(I)^{\frac{1}{2}}\Big(
\frac{\chi_{I}(c_{J_p})}{\mu(I)} 
-\frac{\chi_{I_p}(c_{J_p})}{\mu(I_p)} 
\Big)\chi_{Q}(t). 
$$
If $\mu(I)=0$, we define $\psi_{I,J_p}^{{\rm full}}\equiv 0$. 

We omit the dependence of $\psi_{I,J_p}^{{\rm full}}$ on the cube $Q$.
We note that $\psi_{I,J_p}^{{\rm full}}=0$ if $J_p\cap I_p=\emptyset $, and that 
$\psi_{I,J_p}^{{\rm full}}\chi_{I}=\psi_I\chi_{I}$ when  $J_p\subseteq I$.




We define 
the localized averaging operators by 
$
\hat E_{R}(f)=\langle f\rangle _{R}\chi_{R}(c_{J_{p}})\chi_{Q}
$
and the corresponding localized differences
\begin{align*}
\nonumber
\hat \Delta_{R}f&=\Big(\sum_{I\in \child(R)}\hat E_{I}f\Big)-\hat E_{R}f
\\
&=\Big(\sum_{I\in \child(R)}\langle f\rangle_{I} \chi_{I}(c_{J_p})
\Big)\chi_{Q}
-\langle f\rangle_{R}\chi_{R}(c_{J_p})\chi_{Q}.
\end{align*}


\end{definition}

The following result is the analog of Lemma \ref{adaptedwavelets} for the localized difference operator. 
\begin{lemma}\label{decompfull}
Let 
$R\in {\mathcal D}$ and $J_p\in \mathcal D$ with $\ell(J_p)< \ell(R)$ and $\mu(J_p)\neq 0$.
Then 
$$
\hat\Delta_{R}(f) 
=\sum_{I\in \ch(R)}\langle f, \psi_{I}\rangle 
\psi_{I,J_p}^{{\rm full}}
$$
for $f$ bounded, compactly supported and with mean zero. \end{lemma}
\begin{proof} 
If $\mu(R)=0$ then both sides of the equality are zero. If $\mu(R)\neq 0$, we reason as follows. 
Since
$$
\mu(R)\langle f\rangle _{R}=\int_{R}f d\mu
=\sum_{\substack{I\in \ch(R)\\ \mu(I)\neq 0}}
\int_{I}f d\mu
=\sum_{\substack{I\in \ch(R)\\ \mu(I)\neq 0}}\mu(I)\langle f\rangle _{I},
$$ 
we have
\begin{align*}
\hat \Delta_{R}(f)
&=\! \sum_{\substack{I\in \ch(R)\\ \mu(I)\neq 0}}\! \langle f\rangle _{I}\Big(\chi_{I}(c_{J_p})
-\frac{\mu(I)}{\mu(R)}\chi_{R}(c_{J_p})\Big)
\chi_{Q}
=\! \sum_{\substack{I\in \ch(R)\\ \mu(I)\neq 0}}\mu(I)^{\frac{1}{2}}\! \langle f\rangle _{I}\psi_{I,J_p}^{\rm full}.
\end{align*}
Now, by \eqref{Deltaaverage}, we have
$
\langle f\rangle_{I}
=\langle \Delta_{R}f\rangle_{I} +\langle f\rangle _{R}
$
and so,
\begin{align*}
\hat \Delta_{R}(f)
= \sum_{I\in \ch(R)}\mu(I)^{\frac{1}{2}} \langle \Delta_{R}f\rangle _{I}\psi_{I,J_p}^{\rm full}
+\langle f\rangle_R\sum_{\substack{I\in \ch(R)\\ \mu(I)\neq 0}}\mu(I)^{\frac{1}{2}} \psi_{I,J_p}^{\rm full}.
\end{align*}
We have as before that 
$\langle \Delta_{R}f\rangle _{I}=\langle f, \psi_{I}\rangle$. 

On the other hand, 
let now $R'$ be the union of cubes $I\in \child(R)$ such that $\mu(I)\neq 0$. Then
\begin{align*}
\sum_{\substack{I\in \ch(R)\\\mu(I)\neq 0}}\mu(I)^{\frac{1}{2}}
\psi_{I,J_p}^{{\rm full}}
&= \sum_{\substack{I\in \ch(R)\\ \mu(I)\neq 0}}
\Big(\chi_{I}(c_{J_p})
-\frac{\mu(I)}{\mu(R)}
\chi_{R}(c_{J_p})
\Big)\chi_{Q}
\\
&=
(\chi_{R'}(c_{J_p})
-\chi_{R}(c_{J_p}))\chi_{Q}
=-\chi_{R\setminus R'}(c_{J_p})\chi_{Q}.
\end{align*}

With this, 
\begin{align*}
\hat \Delta_{R}(f)
&=\sum_{I\in \ch(I_{p})}\mu(I)^{\frac{1}{2}}
\langle f, \psi_{I}\rangle
\psi_{I,J_p}^{\rm full}
-\langle f\rangle_{R}\chi_{R\setminus R'}(c_{J_p})\chi_{Q}.
\end{align*}
%

If $\chi_{R\setminus R'}(c_{J_p})\neq 0$ then 
$c(J_p)\in R\setminus R'$. With this, 
since $R\cap J_p\neq \emptyset $ and $\ell(J_p)< \ell(R)$, 
we have $J_p\subsetneq R$. Moreover, since 
$\mu(J_p)\neq 0$, we also have  
$J_p\subset I\in \child{(R)}$ with $\mu(I)\neq 0$, that is, $J_p\subset R'$. But this is contradictory since it implies  
$\chi_{R\setminus R'}(c_{J_p})=0$.

\end{proof}

\subsection{Orthogonality and Bessel inequality of the Haar wavelet systems}

The following lemma summarizes the orthogonality properties of the Haar wavelets. 
\begin{lemma}\label{psiortho} Let $I,J\in \mathcal D$ or $I,J\in \tilde{\mathcal D}$. Then  
$
\int \psi_{I}(x)d\mu(x)
=0. 
$
If $\mu(I)=0$ then $\langle \psi_{I},\psi_{J}\rangle=0$,
while if $\mu(I)\neq 0$ then
\begin{equation}\label{psiortho2}
\langle \psi_{I},\psi_{J}\rangle =
\delta(I_p,J_p)\mu(I)^{\frac{1}{2}}\mu(J)^{\frac{1}{2}}
\Big(\frac{\delta (I,J)}{\mu(I)}-
\frac{1}{\mu(I_{p})}
\Big),
\end{equation}
where we denote $\delta(I,J)=1$ if $I=J$ and zero otherwise.
In addition, if $\mu(I)\neq 0$ we have 
$
\| \psi_{I}\|_{L^{q}(\mu)}
\lesssim \mu(I)^{-\frac{1}{2}+\frac{1}{q}}
$. 
\end{lemma}
\begin{proof} The first equality is trivial.
Equality \eqref{psiortho2} is also 
trivial when $I_{p}\cap J_{p}= \emptyset$.  
When $I_{p}\subsetneq J_{p}$, $\psi_{J}$ is constant on the support of $\psi_{I}$ and so the dual pair is zero due to the mean zero of $\psi_{I}$. Symmetrically, we have the same result when $J_{p}\subsetneq I_{p}$.

%
%
%

On the other hand, for $J_{p}=I_{p}$, we have
\begin{align*}
\langle \psi_{I},\psi_{J}\rangle 
&=\mu(I)^{\frac{1}{2}}\mu(J)^{\frac{1}{2}} 
\int\Big(\frac{\chi_{I}(x)}{\mu(I)}
-\frac{\chi_{I_{p}}(x)}{\mu(I_{p})}\Big)
\Big(\frac{\chi_{J}(x)}{\mu(J)}-\frac{\chi_{J_{p}}(x)}{\mu(J_{p})}\Big) d\mu(x)
\\
&=\mu(I)^{\frac{1}{2}}\mu(J)^{\frac{1}{2}}
\frac{1}{\mu(I)}
\Big(\frac{\mu(I\cap J)}{\mu(J)}
-\frac{\mu(I)}{\mu(I_{p})}\Big).
\end{align*}

For $I\neq J$, since $I\cap J= \emptyset$, 
we have
\begin{align*}
\langle \psi_{I},\psi_{J}\rangle
&=-\frac{\mu(I)^{\frac{1}{2}}\mu(J)^{\frac{1}{2}}}{\mu(I_{p})},
\end{align*}
while for $I=J$, we get 
\begin{align*}
\langle \psi_{I},\psi_{J}\rangle 
&
=
\mu(I)(\frac{1}{\mu(I)}-\frac{1}{\mu(I_{p})}).
\end{align*}


On the other hand, for $\mu(I)\neq 0$, 
\begin{align*}
\| \psi_{I}\|_{L^{q}(\mu)}
&\leq \mu(I)^{\frac{1}{2}}\Big(\frac{1}{\mu(I)}\| \chi_{I}\|_{L^{q}(\mu)}
+\frac{1}{\mu(I_{p})}\| \chi_{I_{p}}\|_{L^{q}(\mu)}\Big)
\\
&= \mu(I)^{\frac{1}{2}}\Big(\frac{1}{\mu(I)^{\frac{1}{q'}}}
+\frac{1}{\mu(I_{p})^{\frac{1}{q'}}}\Big)
\leq 2\mu(I)^{-\frac{1}{2}+\frac{1}{q}}
\end{align*}
since $\mu(I)\leq \mu(I_{p})$. 
\end{proof}

Despite the Haar wavelets we have chosen do not constitute an orthogonal system of functions, they still safisfy Parseval's identity as we see in the next lemma. 
\begin{lemma}\label{Planche}  
For $f\in L^2(\mu)$
$$
\sum_{I\in \mathcal D}
|\langle f,\psi_{I}\rangle |^{2}
= \| f\|_{L^{2}(\mu)}^{2}.
$$
\end{lemma}


\begin{proof} 
We have from Lemma \ref{densityinL2} 
\begin{align*}
\|f\|_{L^2(\mu)}^2
&=\int_{\mathbb R^n}f(x)f(x)d\mu(x)
\\
&
=\lim_{M\rightarrow \infty }\int f(x)
\sum_{\substack{I\in \mathcal{D}\\ 2^{-M}\leq \ell(I)\leq 2^{M}}}
\langle f, \psi_{I}\rangle \psi_{I}(x)d\mu(x)
\\
&
=\lim_{M\rightarrow \infty }\sum_{\substack{I\in \mathcal{D}\\ 2^{-M}\leq \ell(I)\leq 2^{M}}}
\langle f, \psi_{I}\rangle 
\langle f, \psi_{I}\rangle 
=\sum_{I\in \mathcal{D}}
|\langle f, \psi_{I}\rangle|^2 .
\end{align*}

\end{proof}

\subsection{Characterization of compactness}\label{charaofcomp} In this section, 
we explain how to use the Haar wavelets to characterize compactness on $L^2(\mu)$ of Calder\'on-Zygmund operators.

\begin{definition}\label{lagom}
Let $(\psi_{I})_{I\in {\mathcal D}}$ be a Haar wavelet system of $L^2(\mu)$. 
For every $M\in \mathbb N$
we define the lagom projection operator by
$$
P_{M}f=\sum_{I\in {\cal D}_{M}}\langle f, \psi_{I}\rangle \psi_{I},
$$
where $\langle f,\psi_{I}\rangle =\int_{\mathbb R^{n}}f(x)\psi_{I}(x)d\mu(x)$. 
We also define $P_{M}^{\perp }f=f-P_{M}f$.  

We note that $P_{M}^*f=P_Mf$. 
\begin{remark}
When we deal with boundedness, we can consider 
$M=0$ and so, $P_{M}f=0$ and $P_{M}^{\perp }f=f$. 
\end{remark}
\begin{lemma}\label{PMortho} For $f\in L^2(\mu)$
\begin{equation*}
\|P_M^{\perp }f\|_{L^{2}(\mu)}^2
\leq \sum_{I\in \mathcal D_{M}^c}|\langle f,\psi_I\rangle|^2.
\end{equation*}
\end{lemma}
\begin{proof}
By Parseval's identity of Lemma \ref{Planche}
we have 
\begin{align*}
\|P_M^{\perp }f\|_{L^{2}(\mu)}^2&=
\|f- P_Mf\|_{L^{2}(\mu)}^2
=\sum_{\substack{I\in \mathcal D\\ \mu(I)\neq 0}}|\langle f- P_Mf, \psi_I\rangle|^2
\end{align*}
For $I\in \mathcal D$ with $\mu(I)\neq 0$, by the definition of $P_M$, 
\begin{align*}
\langle P_Mf, \psi_I\rangle
&=\sum_{J\in \mathcal D_M}\langle f, \psi_J\rangle \langle \psi_J,\psi_I\rangle .
\end{align*}

From \eqref{psiortho2}, we know that $\langle \psi_J,\psi_I\rangle =0$ if $I_p\neq J_p$. Now we see that 
$J\in \mathcal D_M$ and $I_p=J_p$ imply $I\in \mathcal D_{M+1}$. 

Since $J\in \mathcal D_M$, we have
$2^{-M}\leq \ell(J)\leq 2^{M}$ and 
$\rdist(J,\mathbb B_{2^M})\leq M$. 
Moreover, $I_p=J_p$ implies $\ell(I)=\ell(J)$ and 
$\dist(I,J)=0$. With this,  
we have $2^{-M}\leq \ell(I)\leq 2^{M}$ and
$$
1+\frac{\dist(I, \mathbb B_{2^M})}{2^M}
\leq 1+\frac{\dist(J, \mathbb B_{2^M})+ \ell(I)}{2^M}
\leq \rdist(J,\mathbb B_{2^M})+1\leq M+1.
$$ 
Then 
$\rdist (I,\mathbb B_{2^{M+1}})=1+\frac{\dist(I, \mathbb B_{2^{M+1}})}{2^{M+1}}
\leq 1+\frac{\dist(I, \mathbb B_{2^M})}{2^M}
\leq  M+1$.    

With this, we now reason as follows. 
If $I\in \mathcal D_{M+1}^c$, then $I_p\neq J_p$ for all $J\in \mathcal D_M$ and so, by 
\eqref{psiortho2}, we have  
$\langle \psi_J,\psi_I\rangle=0$. Then 
$
\langle P_Mf, \psi_I\rangle=0$ and thus
$
\langle f- P_Mf, \psi_I\rangle
=\langle f,\psi_I\rangle
$.

On the other hand, if $I\in \mathcal D_{M+1}$ we have
again by \eqref{psiortho2} that 
\begin{align*}
\langle &P_Mf, \psi_I\rangle
=\sum_{J\in \child(I_p)}
\langle f, \psi_J\rangle
\mu(I)^{\frac{1}{2}}\mu(J)^{\frac{1}{2}}
(\frac{\delta (I,J)}{\mu(I)}-
\frac{1}{\mu(I_{p})})
\\
&
=\langle f, \psi_I\rangle
-\frac{\mu(I)^{\frac{1}{2}}}{\mu(I_{p})}
\sum_{J\in \child(I_p)}
\langle f, \psi_J\rangle \mu(J)^{\frac{1}{2}}
=\langle f, \psi_I\rangle
\end{align*}
since 
$$
\sum_{J\in \child(I_p)}
\langle f, \psi_J\rangle \mu(J)^{\frac{1}{2}}
=\int f(x)\sum_{J\in \child(I_p)}
\mu(J)^{\frac{1}{2}}\psi_J(x)d\mu(x)=0
$$
by \eqref{aezero}. 
Therefore
$\langle f- P_Mf, \psi_I\rangle
=0$.

With both results, and $\mathcal D_{M}\subset \mathcal D_{M+1}^c$ we get 
\begin{align*}
\|P_M^{\perp }f\|_{L^{2}(\mu)}^2&
=\sum_{I\in \mathcal D_{M+1}^c}|\langle f,\psi_I\rangle|^2
\leq \sum_{I\in \mathcal D_{M}^c}|\langle f,\psi_I\rangle|^2.
\end{align*}
\end{proof}
\begin{remark}
Previous work also shows that 
\begin{equation*}
\|P_Mf\|_{L^{2}(\mu)}^2
=\sum_{I\in \mathcal D_{M+1}}|\langle f,\psi_I\rangle|^2
\leq \|f \|_{L^2(\mu)}
\end{equation*}
and so, $\|P_M\|_{L^{2}(\mu)\rightarrow L^{2}(\mu)}\leq 1$ and $\|P_M^\perp\|_{L^{2}(\mu)\rightarrow L^{2}(\mu)}\leq 1$.
\end{remark}
\begin{corollary} Let $f\in L^2(\mu)$. Then
\begin{equation}\label{limPar} 
\lim_{M\rightarrow \infty}\|P_M^{\perp }f\|_{L^{2}(\mu)}
=0.
\end{equation}
\end{corollary}
\begin{proof}
By Lemma \ref{Planche} we have 
$
\sum_{I\in \mathcal D}
|\langle f,\psi_{I}\rangle |^{2}
= \| f\|_{L^{2}(\mu)}^{2}<\infty $. Then, by 
Lemma \ref{PMortho}, 
$$
\lim_{M\rightarrow \infty}\|P_M^{\perp }f\|_{L^{2}(\mu)}^2
=\lim_{M\rightarrow \infty}\sum_{I\in \mathcal D_{M+1}^c}|\langle f,\psi_I\rangle|^2
=0.
$$
\end{proof}


\end{definition}

To prove the main result Theorem \ref{Mainresult2}, we will show that the truncated operators 
$T_{\gamma, Q}$ are uniformly bounded or compact on 
$L^2(\mu)$ with estimates independent on $Q$ and $\gamma $. 

\begin{lemma}\label{orthorep} Let $T$ be 
a bounded operator on $L^{2}(\mu)$.
Let $(\psi_{I})_{I\in {\mathcal D}}$ be the Haar wavelet system. 
Then
\begin{align}\label{generaldec}
\langle Tf,g\rangle 
=\sum_{I,J\in \mathcal D}
\langle f,\psi_{I}\rangle 
\langle g,\psi_{J}\rangle  \langle T\psi_{I}, \psi_{J}\rangle
\end{align}
for all $f,g$ compactly supported and integrable. 
\end{lemma}
\begin{proof} Let 
$P_{M}$ be the lagom projection related to the Haar wavelet frame. Since $T$ is bounded, we have
\begin{align*}
|\langle Tf,g\rangle-\langle TP_Mf, P_Mg\rangle |
&\leq |\langle T(f- P_Mf), g\rangle |
+|\langle TP_Mf, g-P_Mg\rangle |
\\
&\leq \| T\| \|f- P_Mf\|_{L^{2}(\mu)}
\|g\|_{L^{2}(\mu)}
\\
&
+\| T\| \|P_Mf\|_{L^{2}(\mu)}
\|g- P_Mg\|_{L^{2}(\mu)}. 
\end{align*}
Since by \eqref{limPar} we have $\|f- P_Mf\|_{L^{2}(\mu)}$ and $\|g- P_Mg\|_{L^{2}(\mu)}$ tend to zero, so does the left hand side of previous chain of inequalities.
\end{proof}

\begin{corollary}\label{repofproj} With the same hypotheses of Lemma \ref{orthorep}, let 
$P_{M}$ be the lagom projection related to the Haar wavelet frame. 
Then
$$
\langle P_{M}^{\perp}TP_{M}^{\perp}f,g\rangle =\sum_{I,J\in \mathcal D_{M}^{c}}
\langle f,\psi_{I}\rangle 
\langle g,\psi_{J}\rangle  \langle T\psi_{I}, \psi_{J}\rangle
$$
for all $f,g$ compactly supported and integrable. 

\end{corollary}
\begin{proof} We have that 
\begin{align*}
\langle P_{M}^{\perp}TP_{M}^{\perp}f,g\rangle 
&=\langle TP_{M}^{\perp}f,P_{M}^{\perp}g\rangle
\\
&=\langle Tf,g\rangle 
-\langle Tf,P_{M}g\rangle 
-\langle TP_{M}f,g\rangle 
+\langle TP_{M}f,P_{M}g\rangle . 
\end{align*}
Then by \eqref{generaldec} the last expression coincides with the right hand side of the statement:
\begin{align*}
\langle P_{M}^{\perp}TP_{M}^{\perp}f,g\rangle 
&=\Big(\sum_{I,J\in \mathcal D}
-\sum_{\substack{I\in \mathcal D\\ J\in \mathcal D_M}}
-\sum_{\substack{I\in \mathcal D_M\\ J\in \mathcal D}}
+\sum_{I,J\in \mathcal D_M}\Big)
\langle f,\psi_{I}\rangle 
\langle g,\psi_{J}\rangle  \langle T\psi_{I}, \psi_{J}\rangle
\\
&=\Big(
\sum_{\substack{I\in \mathcal D\\ J\in \mathcal D_M^c}}
-\sum_{\substack{I\in \mathcal D_M\\ J\in \mathcal D_M^c}}
\Big)
\langle f,\psi_{I}\rangle 
\langle g,\psi_{J}\rangle  \langle T\psi_{I}, \psi_{J}\rangle .
\end{align*}
\end{proof}

As explained in \cite{V2}, to prove compactness of an operator on $L^{2}(\mu)$ 
it suffices to show that 
$
\langle TP_{M}^{\perp}f,P_{M}^{\perp}g\rangle 
$
tends to zero when $M$ tends to infinity uniformly for all 
functions $f,g$  in the unit ball of 
$L^{2}(\mu)$.
Furthermore, $f, g$ can be assumed to be bounded, compactly supported, and with mean zero with respect to $\mu$. 
To see this last point,   
let $(\psi_I)_{I\in  \mathcal D}$ be the Haar wavelets
frame of Definition \ref{defpsi} and $P_N$ be the associated projection operator. 
Let $Q\in \mathcal D$ fixed with 
$\ell(Q)>1$ 
and $f,g$ in the unit ball of $L^2(\mu)$ supported on $Q$. Let also $N_1\in \mathbb N$ 
such that $\|f-P_{N_1}f\|_{L^2(\mu)}
+\|g-P_{N_1}g\|_{L^2(\mu)}
<\frac{\epsilon \gamma^{\alpha}}{\mu(Q)}
$. Then 
\begin{align*}
|\langle &P_{M}^{\perp}T_{\gamma, Q,\mu} P_{M}^{\perp}f,g\rangle 
-\langle P_{M}^{\perp}T_{\gamma, Q,\mu} P_{M}^{\perp}(P_{N_1}f),P_{N_1}g\rangle |
\\
&
\leq \| P_{M}^{\perp}T_{\gamma, Q,\mu} P_{M}^{\perp}\|
(\|f-P_{N_1}f\|_{L^2(\mu)}
+\|g-P_{N_1}g\|_{L^2(\mu)})
\lesssim \epsilon ,
\end{align*}
where we used that $\| P_{M}^{\perp}T_{\gamma, Q,\mu} P_{M}^{\perp}\|\lesssim \| T_{\gamma, Q,\mu}\|\lesssim \frac{\mu(Q)}{\gamma^{\alpha}}$. 

This supports the claim 
since the functions $P_{N_1}f$ and $P_{N_1}g$ are bounded, compactly supported and have mean zero with respect to $\mu$.





We continue with the following technical result:
\begin{lemma}\label{smallF} 
We remind that
$$
F_{\mu}(I, J) \! = \!
\sup_{\substack{R\subset I, S\subset J}}F_{K}(R, S)\rho_{\mu}(R\smlor S)
+F_{T}(I) \delta(I,J),
$$ 
where $\rho_{\mu}$ is defined in \eqref{suitable2}, $F_K$ is defined in 
\eqref{defFmu}, $F_T$ is given in \eqref{b-averageandbound1limit}, $\delta $ is Dirac's delta.
Let also $L, S, \tilde{D}$ be the functions of Definition \ref{LSDF}.

By \eqref{limits}, 
given $\epsilon >0$, we can take $M>0$ so that
$L(2^M)\rho_{\mu}(I)<\epsilon $ if 
$\ell(I)>2^M$, 
$S(2^{-M})\rho_{\mu}(I)<\epsilon $ if 
$\ell(I)<2^{-M}$, and 
$D(M^{\frac{1}{8}})\rho_{\mu}(I)<\epsilon $
if $\rdist(I,\mathbb B_{2^M})>M^{\frac{1}{8}}$, 
and $F_{T}(I)<\epsilon $ for $I\in {\mathcal D}_{M}^{c}$.

Then
for all 
$I\in {\mathcal D}_{2M}^{c}$ and $J\in {\mathcal D}_{M}^{c}$ 
we have that: either $F_{\mu}(I,J)<\epsilon $, or  $|\log(\ec(I,J))|\gtrsim \log M$, or $\rdist(I, J)\gtrsim M^{\frac{1}{8}}$. 
\end{lemma}
\begin{proof} We start with $F_{T}(I) \delta(I,J)$, since the proof is trivial in this case. Since 
$I=J\in \mathcal D_{2M}^{c}\subset \mathcal D_{M}^{c}$, we have $F_{T}(I)<\epsilon /2 $ by the choice of $M$. 

We continue with $F_{K}$.
Since $I\in {\mathcal D}_{2M}^{c}$, we consider three cases: 

a) When $\ell(I)<2^{-2M}$, we have $\ell(I\smland J)<2^{-2M}$. Since $J\in {\mathcal D}_{M}^{c}$, we separate in two cases:
\begin{itemize}
\item[a.1)] If $\ell(J)<2^{-M}$ then
we have $\ell(I\smlor J)<2^{-M}$
and so, we get
$F_{\mu}(I, J )\lesssim S(\ell(I\smland J))\rho_{\mu}(I\smlor J)\leq S(2^{-M})\rho_{\mu}(I\smlor J)<\epsilon$.
\item[a.2)] If $\ell(J)\geq 2^{-M}$ then
$$\ec(I,J)=\frac{\ell(I\smland J)}{\ell(I\smlor J)}
=\frac{\ell(J)}{\ell(I)}\geq \frac{2^{-M}}{2^{-2M}}= 
2^{M}$$
and thus, $\log \ec(I,J)\geq M$. 
\end{itemize}
b) When $\ell(I)>2^{2M}$, since $J\in {\mathcal D}_{M}^{c}$ we distinguish two cases:
\begin{itemize}
\item[b.1)] When $\ell(J)>2^{M}$, we get $\ell(I\smlor J)\geq \ell(I\smland J)>2^{M}$. Then  
%
%
$F_{\mu}(I,J)\rho_{\mu}(I\smlor J)\lesssim L(\ell(I\smland J))\rho_{\mu}(I\smlor J)\leq L(2^M)\rho_{\mu}(I\smlor J)<\epsilon$.
\item[b.2)] When 
$\ell(J)\leq 2^{M}$,
we have that 
$$
\ec(I,J)=\frac{\ell(I\smland J)}{\ell(I\smlor J)}=\frac{\ell(J)}{\ell(I)}<\frac{2^{M}}{2^{2M}}=2^{-M}
$$
and thus, $\log \ec(I,J)\leq -M$. 
\end{itemize}

c) When $2^{-2M}\leq \ell(I)\leq 2^{2M}$ with $\rdist(I,\mathbb B_{2^{2M}})>2M$, we 
have $|c(I)|>(2M-1)2^{2M}$. We fix $\alpha =\frac{1}{8}$, $\beta =\gamma =\frac{1}{4}$. Then,
\begin{itemize}
\item[c.1)] When $\ell(J)>(2M)^{\alpha}2^{2M}$, since $\alpha >0$ we have 
$$\ec(I,J)=\frac{\ell(I)}{\ell(J)}<\frac{2^{2M}}{(2M)^{\alpha }2^{2M}}\lesssim 
M^{-\frac{1}{8}},$$
which implies $\log \ec(I,J)\lesssim \log M$. 

\item[c.2)] When $\ell(J)\leq (2M)^{\alpha }2^{2M}$, we have $\ell(I\smlor J)<(2M)^{\alpha }2^{2M}$. Now:
\vskip5pt
\noindent c.2.1) When $\rdist(\langle I,J\rangle,\mathbb B 
)>(2M)^{\beta }$, we also have 
$\rdist(I\smlor J,\mathbb B 
)>(2M)^{\beta }$. Then
\begin{align*}
F_{K}(I,J)\rho_{\mu}(I\smlor J)&\lesssim \tilde D(\rdist(\langle I,J\rangle,\mathbb B))\rho_{\mu}(I\smlor J)
\\
&
\leq \tilde D(M^\beta )\rho_{\mu}(I\smlor J)
<\epsilon.
\end{align*}

\noindent c.2.2) When $\rdist(\langle I,J\rangle,\mathbb B
)\leq (2M)^{\beta }$, we get 
$|c(\langle I,J\rangle)|\leq (2M)^{\beta }(
1+\ell(\langle I,J\rangle))$. Then, we examine the last two cases:
\begin{itemize}
\item When $\ell(\langle I,J\rangle)>(2M)^{\gamma}2^{2M}$,  
we get
$$
\hspace{.5cm}\rdist(I,J)=\frac{\ell(\langle I,J\rangle)}{\ell(I\smlor J)}>\frac{(2M)^{\gamma}2^{2M}}{(2M)^{\alpha }2^{2M}}
\gtrsim M^{\gamma -\alpha }=M^{\frac{1}{8}}.
$$

\item When $\ell(\langle I,J\rangle)\leq (2M)^{\gamma }2^{2M}$, we have instead
\begin{align*}
\hspace{1.5cm}
&|c(I)-c(J)|>|c(I)|-|c(\langle I,J\rangle)-c(J)|-|c(\langle I,J\rangle)|
\\
&\geq |c(I)|-2^{-1}\ell(\langle I,J\rangle)-(2M)^{\beta }(
1+\ell(\langle I,J\rangle))
\\
&\geq (2M-1)2^{2M}-(2M)^{\gamma }2^{2M}-(2M)^{\beta }(
1+(2M)^{\gamma }2^{2M})
\\
&\gtrsim (M\!-\!M^{\gamma}\!-\!M^{\beta}\!-\!M^{\beta +\gamma })2^{2M}
\gtrsim (M\!-\!3M^{\frac{1}{2}})2^{2M}\geq 2^{-1}M2^{2M}
\end{align*}
\hspace{-.5cm} for $M\geq 36$. Then 
\begin{align*}
\hspace{1.5cm}
\rdist(I,J)&\geq \frac{|c(I)-c(J)|}{\ell(I\smlor J)}
\gtrsim  \frac{M2^{2M}}{(2M)^{\alpha }2^{2M}}\gtrsim M^{1-\alpha }
=M^{\frac{7}{8}}.
\end{align*}
\end{itemize}

\end{itemize}
\end{proof}

\begin{definition}\label{FM} As showed in the proof, 
$F_{\mu}(I,J)<\epsilon $ holds when either
$\ell(I\smland J)>2^{M}$, or
$\ell(I\smlor J)<2^{-M}$, or $\rdist(\langle I,J\rangle ,\mathbb B)>M^{1/8}$. For this reason, 
we denote by $\mathcal F_{M}$ the family of
ordered pairs $(I,J)$ with $I,J\in \mathcal D_{M}^{c}$ satisfying some of these three inequalities.
\end{definition}

\section{The operator acting on bump functions}\label{bump}

We estimate the dual pair
 $\langle T\psi_{I},\psi_{J}\rangle $ 
  in terms of the space and
frequency location of the argument functions.
The computations are carried out in two different propositions.

We start with a technical lemma, which in turn requires some explicit properties of the auxiliary functions $L$, $S$, $D$, and $F$ provided in Notation \ref{LSDF}. 
 We first note that, without loss of generality, $L$ and $D$ can be assumed to be non-creasing
while $S$ can be assumed to be non-decreasing. Moreover, we also have the following equivalent expression of the kernel smoothness condition:

\begin{remark}\label{rinfinity}
In \cite{V}, it is proved that
the smoothness condition 
\eqref{smoothcompactCZ}
implies the modified smoothness condition
\eqref{LSD}, which we will often use: 
\begin{equation}\label{LSD}
|K(t, x)-K(t',x')|
\lesssim
\frac{(|t-t'|+|x-x'|)^\delta}{|t-x|^{\delta}}\frac{F(t,x,t',x')}{|t-x|^{\alpha}},
\end{equation}
whenever $|t-t'|+|x-x'|<|t-x|<|t'-x'|$, with $0<\delta <1$ and 
$$F(t,x,t',x')=L_{1}(|t-x'|)S_{1}(|t-t'|+|x-x'|)D_{1}\Big(1+\frac{|t+x'|}{1+|t-x'|}\Big),$$ 
where
$L_{1}, S_{1}, D_{1}$ satisfy the limits in
(\ref{limits}). 

\end{remark}

Now we state and prove the mentioned technical lemma.




\begin{lemma}\label{boundtoFK2}
Let $I_p,J_p\in \mathcal D$ such that $\ell(J_p)\leq \ell(I_p)$ and 
$\dist(I_p,J_p)\geq \ell(J_p)$. 
Let $t\in I_p$, $x\in J_p$, $c_{J_p}=c(J_p)$ and 
$$
F(t,x)=L(|t-c_{J_p}|)
S(|x-c_{J_p}|)D\Big(1+\frac{|t+c_{J_p}|}{1+|t-c_{J_p}|}\Big).
$$
Then 
$$
F(t,x)\leq L(\ell([I_p,J_p]))
S(\ell(J_p))D(\rdist (\langle I_p,J_p\rangle, \mathbb B))
$$ 
\end{lemma}
\begin{proof}
Since $L$ is non-increasing, $S$ is non-decreasing, 
$|t-c_{J_p}|> \dist(I_p,J_p)= \ell([I_p,J_p])$
and $|x-c_{J_p}|\leq \ell(J_{p})/2$,  
we get
$$
F(t,x)\leq L(\ell([I_p,J_p]))
S(\ell(J_{p}))D\Big(1+\frac{|t+c_{J_p}|}{1+|t-c_{J_p}|}\Big).
$$ 
From $t\in I_p$, $c_{J_p}\in J_{p}$, and $I_p\cap J_p=\emptyset $, we get $|t-c_{J_p}|\leq \dist(I_p, J_p)+\ell(I_p)+\ell(J_p)\leq 2\ell(\langle I_p,J_p\rangle)$.
Then, since $|t+c_{J_p}|\geq 2|c_{J_p}|-|t-c_{J_p}|$, we have 
\begin{align*}
2(1+\frac{|t+c_{J_p}|}{1+|t-c_{J_p}|})
&\geq 2+\frac{|t+c_{J_p}|}{1+|t-c_{J_p}|}
\\
&
\geq 
2+\frac{2|c_{J_p}|}{1+|t-c_{J_p}|}
-\frac{|t-c_{J_p}|}{1+|t-c_{J_p}|}
\\
&
\geq 1+\frac{|c_{J_p}|}{1+\ell(\langle I_p,J_p\rangle )}.
\end{align*}
Moreover, 
since $|c(I_p)|-|c(J_p)|\leq |c(I_p)-c(J_p)|\leq \ell(\langle I_p,J_p\rangle )$, we can bound below
the numerator in the last expression 
as follows:
\begin{align*}
1+\ell(\langle I_p,J_p\rangle )+|c_{J_p}| 
&\geq 1+\frac{\ell(\langle I_p,J_p\rangle ) }{2}+\frac{|c(I_p)|-|c(J_p)|}{2}+|c(J_p)|
\\
&\geq \frac{1}{2}\big(1+\ell(\langle I_p,J_p\rangle )+\frac{1}{2}|c(I_p)+c(J_p)|\big).
\end{align*}
Therefore, 
\begin{align*}
1+\frac{|c_{J_p}|}{1+\ell(\langle I_p,J_p\rangle )}
&\geq \frac{1}{2}\Big(1+\frac{|c(I_p)+c(J_p)|/2}{1+\ell(\langle I_p,J_p\rangle )}\Big)
\\
&\geq \frac{1}{3}\Big(\frac{3}{2}+\frac{|c(I_p)+c(J_p)|/2}{1+\ell(\langle I_p,J_p\rangle )}\Big).
\end{align*}
Now, since $(c(I_p)+c(J_p))/2\in \langle I_p,J_p\rangle$, we have
$|(c(I_p)+c(J_p))/2-c(\langle I_p,J_p\rangle)|\leq \ell(\langle I_p,J_p\rangle)/2$
and so, we can bound below previous expression by
\begin{align*}
\frac{1}{3}\Big(\frac{3}{2}+\frac{|c(\langle I_p,J_p\rangle )|}{1+\ell(\langle I_p,J_p\rangle )}-\frac{1}{2}\Big) 
&\geq \frac{1}{3}\Big(1+\frac{|c(\langle I_p,J_p\rangle )|}{2\max(\ell(\langle I_p,J_p\rangle),1)}\Big)
\\
&
\geq \frac{1}{6}\Big(2+\frac{|c(\langle I_p,J_p\rangle )|}{\max(\ell(\langle I_p,J_p\rangle),1)}\Big)
\\
&
\gtrsim 1+\frac{|c(\langle I_p,J_p\rangle )|+\max(\ell(\langle I_p,J_p\rangle),1)}{\max(\ell(\langle I_p,J_p\rangle),1)}
\\
&
\geq 1+\frac{\dist( \langle I_p,J_p\rangle, 
\mathbb B)}{\max(\ell(\langle I_p,J_p\rangle),1)}
=\rdist(\langle I_p,J_p\rangle ,\mathbb B )
\end{align*}
with $\mathbb B=[-1/2,1/2]^n$. 

Finally, by using that $D$ is non-increasing, we get
$$
F(t,x)\leq L(\ell([I_p,J_p]))
S(\ell(J_p))D(\rdist (\langle I_p,J_p\rangle, \mathbb B))
.
$$ 

\end{proof}

\begin{proposition}\label{twobumplemma1} 
Let $T$
be a linear operator with compact C-Z kernel $K$
and parameters $0<\delta <1$, $0<\alpha\leq n$. 
Let $\theta \in(0,1)$ and  $I,J\in \mathcal D$ be such that 
$\dist(I_p,J_p)>0$ and 
$\ec(I,J)^\theta (\inrdist(I_{p}, J_{p})-1)>1$.  
Then 
\begin{align*}
|\langle T\psi_{I},\psi_{J}\rangle |
&\lesssim \inrdist(I_p,J_p)^{-(\alpha+\delta )}
\frac{\mu(I)^{\frac{1}{2}}\mu(J)^{\frac{1}{2}}}
{\ell(I\smland J)^{\alpha}}
F_{1}(I,J),
\end{align*}
with
$F_{1}(I, J)
=L(\ell([I_p,J_p]))
S(\ell(I_p\smland J_p))D(\rdist (\langle I_p,J_p\rangle, \mathbb B))
$.
\end{proposition}
\begin{proof} By symmetry we can assume $\ell(J)\leq \ell(I)$. Let $e\in \mathbb N$ such that $\ec(I,J)^{-1}=\ell(I)/\ell(J)=2^{e}
\geq 1$.
Then 
$$
\frac{\dist(I_{p}, J_{p})}{\ell(J_p)}
=\inrdist(I_p,J_p)-1
\geq \ec(I,J)^{-\theta}
=2^{e\theta}
$$
that is, $\dist(I_{p}, J_{p})\geq 2^{e\theta}\ell(J_{p})\geq \ell(J_p)$. 
We can then use the kernel representation of $T$ and
the zero mean of $\psi_{J}$ 
to write
$$
\langle T\psi_{I},\psi_{J}\rangle
=\int \int \psi_{I}(t)\psi_{J}(x) (K(t,x)-K(t,c_{J_{p}}))\, d\mu(t)d\mu(x)
$$
with $c_{J_p}=c(J_p)$.
Since
$
\psi_{I}=\mu(I)^{\frac{1}{2}}(\mu(I)^{-1}\chi_{I}-\mu(I_{p})^{-1}\chi_{I_{p}})
$
and similar for $\psi_{J}$, 
we  have 
\begin{align}\label{sumf}
\nonumber
|\langle T\psi_{I},\psi_{J}\rangle | &\lesssim 
\mu (I)^{\frac{1}{2}} \mu (J)^{\frac{1}{2}}
\sum_{R\in \{ I,I_{p}\}}\sum_{S\in \{ J,J_{p}\}}
\mu (R)^{-1}\mu (S)^{-1}
\\
&
\hskip30pt
\int_{S} \int_{R} 
|K(t,x)-K(t,c_{J_{p}})|\, d\mu(t)d\mu(x) .
\end{align}

We fix $R\in \{ I,I_{p}\}$ and $S\in \{ J,J_{p}\}$. 
In the domain of integration of the double integral we have $t\in R\subset I_p$, $x\in S\subset J_p$ and so,
$$
2|x-c_{J_{p}}|\leq \ell(J_{p})
\leq \dist(I_p,J_p)
\leq |t-c_{J_{p}}|.
$$ 
Then, by the smoothness condition 
of a compact C-Z kernel \eqref{smoothcompactCZ}, 
the double integral in \eqref{sumf}
is bounded by
\begin{align*}
\int_{S}&
\int_{R
}\frac{|x-c_{J_p}|^\delta }{|t-x|^{\alpha +\delta }}F(t,x)d\mu(t) d\mu(x)
\end{align*}
with 
$
F(t,x)=L(|t-c_{J_p}|)
S(|x-c_{J_p}|)D\Big(1+\frac{|t+c_{J_p}|}{1+|t-c_{J_p}|}\Big)
$.
Now, 
by
Lemma \ref{boundtoFK2}, 
previous expression can be bounded by
\begin{align*}
\frac{\ell(J_p)^{\delta }}{\dist(S,R)^{\alpha +\delta }}\mu(R)\mu(S)
L(\ell([I_p,J_p]))S(\ell(J))D(\rdist (\langle I_p,J_p\rangle, \mathbb B)).
\end{align*}
Since $R\subset I_p$ and $S\subset J_p$, we have  $\dist(S,R)\geq \dist(I_p,J_p)
$. 
Furthermore, since $\dist(I_{p}, J_{p})\geq \ell(J_{p})$, we have
\begin{align*}
\dist(I_{p},J_{p})\geq 2^{-1}(\dist(I_{p},J_{p})+\ell(J_{p})).
\end{align*}
With this, 
we can continue the bound in \eqref{sumf} as
\begin{align}\label{separated}
|\langle T\psi_{I},\psi_{J}\rangle | &\lesssim 
\mu(I)^{\frac{1}{2}}\mu(J)^{\frac{1}{2}}
\sum_{R\in \{ I,I_{p}\}}\sum_{S\in \{ J,J_{p}\}}
\frac{\ell(J_p)^{\delta}}{\dist(I_p,J_p)^{\alpha +\delta }}
F_1(I,J)
\\
&
\nonumber
\lesssim 
\Big(\frac{\ell(J_p)}{\ell(J_p)+\dist(I_p,J_p)}\Big)^{\alpha+\delta}
\frac{\mu(I)^{\frac{1}{2}}\mu(J)^{\frac{1}{2}}}{\ell(J)^\alpha}
F_1(I,J).
\end{align}
\end{proof}
\begin{remark}\label{alterbump}
When $\ell(I_p\smlor J_p)\leq \dist(I_p,J_p)$ we will use the weaker inequality
\begin{align}\label{bumpwithec}
|\langle T\psi_{I},\psi_{J}\rangle |
&\lesssim \ec(I,J)^{\delta }
\rdist(I_p,J_p)^{-(\alpha+\delta)}
\frac{\mu(I)^{\frac{1}{2}}\mu(J)^{\frac{1}{2}}}
{\ell(I\smlor J)^{\alpha}}
F_{1}(I,J),
\end{align}
which we now justify. Assuming $\ell(J)\leq \ell(I)$, we have $\dist(I_{p}, J_{p})\geq \ell(I_{p})$. Then
\begin{align*}
\dist(I_{p},J_{p})\geq 2^{-1}(\dist(I_{p},J_{p})+\ell(I_{p}))\geq 
2^{-1}(\dist(I_{p},J_{p})+\ell(J_{p})).
\end{align*}
With this, we get from \eqref{separated}
\begin{align*}
|&\langle T\psi_{I},\psi_{J}\rangle | \lesssim 
\mu(I)^{\frac{1}{2}}\mu(J)^{\frac{1}{2}}
\Big(\frac{\ell(J_p)}{\dist(I_p,J_p)}\Big)^{\delta }
\frac{1}{\dist(I_p,J_p)^\alpha}
F_1(I,J)
\\
&\lesssim 
\Big(\frac{\ell(J_p)}{\ell(J_p)+\dist(I_p,J_p)}\Big)^{\delta}
\Big(\frac{\ell(I_p)}{\ell(I_p)+\dist(I_p,J_p)}\Big)^{\alpha}
\frac{\mu(I)^{\frac{1}{2}}\mu(J)^{\frac{1}{2}}}{\ell(I)^\alpha}
F_1(I,J)
\\
&=\inrdist(I_p,J_p)^{-\delta }
\rdist(I_p,J_p)^{-\alpha}
\frac{\mu(I)^{\frac{1}{2}}\mu(J)^{\frac{1}{2}}}
{\ell(I)^{\alpha}}
F_{1}(I,J).
\end{align*}
Finally,  
\begin{align*}
\inrdist(I_p,J_p)^{-\delta}
&\lesssim \Big(\frac{\ell(J)}
{\dist(I_p,J_p)}\Big)^{\delta}
\lesssim \Big(\frac{\ell(J)}{\ell(I)}\Big)^{\delta}
\Big(\frac{\ell(I_p)}{\ell(I_p)+\dist(I_p,J_p)}\Big)^{\delta}
\\
&=\ec(I,J)^{\delta}
\rdist(I_p,J_p)^{-\delta}.
\end{align*}

\end{remark}

\begin{remark}
We also note that, from
$\dist(I_p,J_p)\leq \dist(I,J)
\leq \dist(I_p,J_p)+\ell(I_p)$,
we have
$$
\frac{1}{3}(1+\frac{\dist(I,J)}{\ell(I)})
\leq 1+\frac{\dist(I_p,J_p)}{\ell(I_p)}
\leq 1+\frac{\dist(I,J)}{\ell(I)},
$$
that is, $\rdist(I_p,J_p)\approx\rdist (I,J)$. 
\end{remark}


For the next Lemma, we remind the following notation introduced in Definition \ref{fullwavelet}. For  
$I_p,J_p\in \mathcal D$, $Q\in 3\mathcal D$ with 
$I_p,J_p\subset 3^{-1}Q$, we write
$$
\psi_{I,J}^{{\rm full}}(t)=\mu(I)^{\frac{1}{2}}(\varphi_{I}( c_{J_p})
-\varphi_{I_{p}}(c_{J_{p}})
)\chi_{Q}(t),
$$
with $\varphi_{I}
=\frac{1}{\mu(I)}\chi_{I}$, $c_{J_p}=c(J_p)$. 




\begin{proposition}\label{twobumplemma2}
Let $T$
be a linear operator with compact C-Z kernel $K$
and parameter $0<\delta <1$. 
Let $I,J\in \mathcal D$ be such that
$\dist(I_{p}, J_{p})=0$ and
$\ec(I,J)^{\theta}(\inrdist(I_{p},J_{p})-1)\geq 1$. Then
\begin{align*}
|\langle T(\psi_{I}-{\psi}^{\rm full}_{I,J}),\psi_{J}\rangle |
&\lesssim
\inrdist(I_p,J_p)^{-\delta}
\sum_{R\in \{I,I_p\}}
\Big(\frac{\mu(R\cap J)}{\mu(R)}\Big)^{\frac{1}{2}}
F_{2,\mu }(I,J)
\\
&+\inrdist(I_p,J_p)^{-(\alpha+\delta)}
\frac{\mu(I)^{\frac{1}{2}}\mu(J)^{\frac{1}{2}}}{\ell(
I\smland J)^{\alpha }}\chi_{I_p\setminus I}(c_{J_p})F_{3}(I,J),
\end{align*}
where 
\begin{align*}
F_{2,\mu}(I, J)&=
L(\ell(I\smland J))S(\ell(I\smland J))
\sum_{k\geq 0}2^{-k\delta }\frac{\mu(2^kK)}{\ell(2^kK)}D(\rdist(2^kK,\mathbb B))
\end{align*} 
and
\begin{align*}
F_{3}(I, J)&=
L(\ell(I\smland J))S(\ell(I\smland J))
\sum_{k\geq 0}2^{-k\delta }D(\rdist(2^kK,\mathbb B))
\end{align*} 
with 
$K=\inrdist(I_p,J_p)(I\smland J)$.
\end{proposition}




\begin{proof} We assume $\ell(J)\leq \ell(I)$.  
Let $e\in \mathbb N$ such that $2^{e}=\ell(I)/\ell(J)\geq 1$. 
%
Since $\dist(I_{p}, J_{p})=0$ and
$\ec(I,J)^{\theta}\inrdist(I_{p},J_{p})> 1$, we have
$$
\frac{\dist(J_{p}, \mathfrak D_{I_p})}{\ell(J_p)}
=\inrdist(I_p,J_p)-1
\geq \ec(I,J)^{-\theta}
=2^{e\theta}
$$
Then
$\dist(J_{p},\partial I_{p})>2^{e\theta}\ell(J_p)$, which implies that 
$3J_{p}\subsetneq I_{p}$ with $\ell(J)\leq \ell(I)/8$ and so,  
$3J_{p}\subseteq I'$ for some $I'\in \child(I_{p})$. 

Now we note that 
 \begin{align*}
\psi_{I}(t)-{\psi}^{\rm full}_{I,J}(t)
=\mu(I)^{\frac{1}{2}}[\varphi_{I}(t)-
\varphi_{I}(c_{J_{p}})
\chi_{Q}(t)
-\varphi_{I_{p}}(t)+
\varphi_{I_{p}}(c_{J_{p}})
\chi_{Q}(t)].
\end{align*}

Then if $t\in 3J_p\subsetneq I_p$ we have   $\varphi_{R}(t)\chi_{3J_{p}}(t)=\varphi_{R}(c_{J_{p}})\chi_{3J_{p}}(t)$ for $R\in \{ I,I_{p}\}$ 
and so 
$\psi_{I}(t)-{\psi}^{\rm full}_{I,J}(t)=0$.  
With this,
\begin{align*}
\psi_{I}-\psi_{I,J}^{{\rm full}}&=(\psi_{I}-\psi_{I,J}^{\rm full})(1-\chi_{3J_{p}})
\end{align*}
We denote the last expression by $\psi_{I}^{\rm out}$, 
which is supported on $(I_p\cup Q)\setminus 3J_p$. 
 Since 
  $\dist((I_p\cup Q)\setminus 3J_{p}, J_{p})
  \geq \ell(J_{p})$, 
we can apply the reasoning we used in Proposition \ref{twobumplemma2} with some variations.
We describe again the argument
 because we aim for slightly different estimates. 

We improve previous argument. Since 
that $J_{p}\subseteq I'$ for some $I'\in \child(I_{p})$,
we have for $t\in I'$ that  $\varphi_{R}(t)=\varphi_{R}(c_{J_{p}})$
with $R\in\{I,I_p\}$, 
and so $\psi^{\rm out}_{I}(t)\equiv 0$.
That is, $\psi^{\rm out}_{I}(t)\neq 0$ implies 
$t\in ((I_p\cup Q)\setminus I')\cap (3J_{p})^{c}$. 
Then 
\begin{align*}
\hskip50pt |t&-c(J_{p})|\geq \frac{\ell(J_{p})}{2}+\dist(I_p\backslash I', J_{p})
\\&
= \frac{\ell(J_{p})}{2}+\dist(J_{p},{\mathfrak D}_{I_{p}}) 
\geq \frac{1}{2}\inrdist(I_p,J_p)\ell(J_p).
\end{align*}

Now we prove the following inequalities: 
for $J_p\subset I_p$, 
\begin{itemize}
\item[1)] if $J_p\subset I$ then $|\psi_I^{\rm out}|
\lesssim \mu(I)^{\frac{1}{2}}\frac{1}{\mu(I)}\chi_{Q\setminus I},
$
\item[2)] if $J_p\cap I=\emptyset$ then 
$|\psi_I^{\rm out}|
\lesssim \mu(I)^{\frac{1}{2}}(\frac{1}{\mu(I)}\chi_{I}
+\frac{1}{\mu(I_p)}\chi_{Q\setminus I}).
$
\end{itemize}
1) If $I'=I$, since $J_p\subset I\subset I_p$, we have seen that for all $t\in I$, $\psi_I^{\rm out}(t)=0$. Meanwhile for 
$t\in I_p\setminus I$ we 
have $\varphi_{I}(t)=0$ and
$\varphi_{I_{p}}(t)=\varphi_{I_{p}}(c_{J_{p}})$. Then
$$
\psi_I^{\rm out}(t)
=\mu(I)^{\frac{1}{2}}[-\varphi_{I}(c_{J_{p}})
\chi_{Q}(t)]
=-\mu(I)^{\frac{1}{2}}\frac{1}{\mu(I)}\chi_{Q\setminus I}(t)
$$
Finally for $t\in Q\setminus I_p$ we have $\varphi_{I}(t)=\varphi_{I_p}(t)=0$
and so
$$
|\psi_I^{\rm out}(t)|
=\mu(I)^{\frac{1}{2}}|-\varphi_{I}(c_{J_{p}})
\chi_{Q}(t)
+\varphi_{I_{p}}(c_{J_{p}})
\chi_{Q}(t)|
\leq \mu(I)^{\frac{1}{2}}\frac{2}{\mu(I)}\chi_{Q\setminus I}(t)
$$
since $\mu(I)\leq \mu(I_p)$. 

2) On the other hand, if $I'\neq I$ we have that 
$I'\cap I=\emptyset $ and so, since $J_p\subset I'$, for $t\in I$ we get $\varphi_{I}(c_{J_{p}})=0$ and
$\varphi_{I_{p}}(t)=\varphi_{I_{p}}(c_{J_{p}})$. With this
\begin{align*}
\psi_I^{\rm out}(t)
=\mu(I)^{\frac{1}{2}}\varphi_{I}(t)
=\mu(I)^{\frac{1}{2}}\frac{1}{\mu(I)}\chi_{I}(t).
\end{align*}
Meanwhile 
for 
$t\in I_p\setminus I$ we have $\varphi_I(t)=\varphi_{I}(c_{J_{p}})=0$ and 
$\varphi_{I_{p}}(t)=\varphi_{I_{p}}(c_{J_{p}})$ and so, we get 
$$
\psi_I^{\rm out}(t)
=\mu(I)^{\frac{1}{2}}[
-\varphi_{I_{p}}(t)+
\varphi_{I_{p}}(c_{J_{p}})
\chi_{Q}(t)]=0.
$$
Finally for $t\in Q\setminus I_p$ we have 
$\varphi_{I}(t)=\varphi_{I_{p}}(t)=\varphi_{I}(c_{J_{p}})=0$ and so,
$$
\psi_I^{\rm out}(t)
=\mu(I)^{\frac{1}{2}}
\varphi_{I_{p}}(c_{J_{p}})
\chi_{Q}(t)
\leq \mu(I)^{\frac{1}{2}}\frac{1}{\mu(I_p)}\chi_{Q\setminus I_p}(t).
$$
This finishes the proof of these two inequalities.

We also have that for $t\in (I_p\cup Q)\setminus 3J_p$ we have 
$|t-c(J_{p})|\geq 3\ell(J_{p})/2> \ell(J_{p})$. 
%
%
%
We then decompose the support of $\psi_I^{\rm out}$ as follows. Let  
  $\Delta_{k} =\{ t\in (I_p\cup Q)\setminus 3J_{p} : 2^{k-1}\ell(J_{p})< |t-c(J_{p})|\leq 2^{k}\ell(J_{p})\}\subset (2^{k+1}J_{p}\backslash 2^{k}J_{p})$. Then
  $$
(I_p\cup Q)\setminus (3J_p)\subset \bigcup_{k=m_{0}}^{m_{1}}\Delta_{k},
$$
with $m_0=\log \inrdist(I_p,J_p)$ and 
$m_1=\log\frac{\ell(I_p)+\ell(Q)}{\ell(J_p)}+1$. 
This way we can write
$$
\psi_I^{\rm out}=\sum_{k=  m_0}^{m_1}\Phi_k
$$
 where
  $\Phi_{k}=\psi_{I}^{{\rm out}} (\chi_{2^{k+1}J_{p}}-\chi_{2^{k}J_{p}})$. 
 We note that, since $J_{p}\subset 3J$,
 we have $\supp \Phi_{k} \subseteq \Delta_{k} \subseteq 2^{k+1}J_{p}\subset 2^{k+3}J$ and so,
  $\mu(\Delta_{k})\leq \mu(2^{k+3}J)$. 
Moreover,  
$\Delta_{k}$ is included in the difference of two concentric cubes with 
diameters $2^{k}\ell(J_p)$ and $2^{k+1}\ell(J_p)$. Then, despite $\Delta_{k}$ is not a cube, we denote
$\ell(\Delta_{k})=2^{k+1}\ell(J_p)$ and $c(\Delta_{k})=c(J_{p})$. 

The plan is now to estimate
  $|\langle T\Phi_{k}, \psi_{J}\rangle |$. 
 Since 
 $\Delta_k \cap J_p=\emptyset$, 
 we use the kernel representation
 and the zero mean of $\psi_{J}$ to write
 \begin{align*}
 |\langle T\Phi_k, \psi_{J}\rangle |
 &\leq \int_{J_{p}} \int_{\Delta_k}
 |\psi_{I}^{{\rm out}}(t)| |\psi_{J}(x)| |K(t,x)-K(t,c_{J_{p}})|\, d\mu(t)d\mu(x) .
 \end{align*}


As in previous proposition, 
 $
 \psi_{J}=\mu(J)^{\frac{1}{2}}(\frac{1}{\mu(J)}\chi_{J}-\frac{1}{\mu(J_p)}\chi_{J_{p}})
 $ 
Meanwhile, we can write previous two inequalities in a unified way as follows: 
 \begin{align*}
 |\psi_{I}^{{\rm out}}(t)| 
&\lesssim \mu(I)^{\frac{1}{2}}
\Big(\frac{1}{\mu(I)}\chi_{Q\setminus I}(t)\chi_{I}(c_{J_p})
\\
&
\hskip40pt +\frac{1}{\mu(I)}\chi_{I}(t)
\chi_{I_p\setminus I}(c_{J_p})
+\frac{1}{\mu(I_p)}\chi_{Q\setminus I}(t)
\chi_{I_p\setminus I}(c_{J_p})
\Big)
\\
&\lesssim \mu(I)^{\frac{1}{2}}
\Big(\frac{1}{\mu(I)}\chi_{I}(c_{J_p})
+\frac{1}{\mu(I_p)}
\chi_{I_p}(c_{J_p})
+\frac{1}{\mu(I)}\chi_{I}(t)
\chi_{I_p\setminus I}(c_{J_p})
\Big)
\end{align*}
 Then 
 \begin{align}\label{sum2}|\langle T\Phi_{k},\psi_{J}\rangle |
 &\lesssim \mu(I)^{\frac{1}{2}}\mu(J)^{\frac{1}{2}}
 \sum_{R\in \{ I,I_{p}\}}
 \sum_{S\in \{ J,J_{p}\}}\frac{\chi_{R}(c_{J_p})}{\mu(R)}
 \frac{1}{\mu(S)}
 \\
 &
  \nonumber
\hskip40pt  \int_{S} \int_{\Delta_k}  |K(t,x)-K(t,c_{J_{p}})|\, d\mu(t)d\mu(x)
 \\
 &
 \nonumber
 +\mu(I)^{\frac{1}{2}}\mu(J)^{\frac{1}{2}}
 \sum_{S\in \{ J,J_{p}\}}\frac{\chi_{I_p\setminus I}(c_{J_p})}{\mu(I)}
 \frac{1}{\mu(S)}
 \\
 &
  \nonumber
\hskip40pt  \int_{S} \int_{I\cap \Delta_k}  |K(t,x)-K(t,c_{J_{p}})|\, d\mu(t)d\mu(x).
 \end{align}



We now estimate the double integral in the right hand size of \eqref{sum2}, starting with the first one which we denote by $\Int$. 
We fix $R\in \{ I,I_{p}\}$ and $S\in \{ J,J_{p}\}$. For $t\in \Delta_k $ we have  $|t-c(J_{p})|> 2^{k-1}\ell(J_{p})\geq 
 \ell(J_p)$. For $x\in S\subset J_p$ we have $|x-c(J_{p})|\leq \ell(J_{p})/2$. 
With both things 
 \begin{equation*}
 |x-c(J_{p})|
 \leq \ell(J_{p})/2
 \leq 2^{k-1}\ell(J_{p})/2
 <|t-c(J_{p})|/2.
 \end{equation*}
and so, we can use the smoothness property \eqref{smoothcompactCZ}, to write
\begin{equation}\label{intin}
\Int \leq \int_{S}\int_{\Delta_{k}}\frac{|x-c(J_{p})|^{\delta }}{|t-c(J_{p})|^{\alpha +\delta}}
F(t,x)d\mu(t)
d\mu(x)
\end{equation}
with 
$F(t,x)=L(|t-c(J_p)|)
S(|x-c(J_p)|)D\Big(1+\frac{|t+c(J_p)|}{1+|t-c(J_p)|}\Big)$.
Since $L$ is non-increasing, $S$ is non-decreasing, 
$2^{k}\ell(J)\geq |t-c(J_p)|> 2^{k-1}\ell(J_p)\geq \ell(J_p)=2\ell(J)$
and $|x-c(J_p)|\leq \ell(J_p)/2=\ell(J)$, we have 
$$
F(t,x)\leq L(\ell(J))
S(\ell(J))D\Big(1+\frac{|t+c(J_p)|}{1+|t-c(J_p)|}\Big).
$$ 
On the other hand,
$|t+c(J_p)|\geq 2|c(J_p)|-|t-c(J_p)|$ which implies 
\begin{align*}
2(1+\frac{|t+c(J_p)|}{1+|t-c(J_p)|})&\geq 
2+\frac{2|c(J_p)|}{1+|t-c(J_p)|}-\frac{|t-c(J_p)|}{1+|t-c(J_p)|}
\\&
\geq 1+\frac{|c(J_p)|}{1+2^{k}\ell(J_p)}.
\end{align*}
Moreover, since $\Delta_{k}\subset 2^{k+3}J$ and 
$\ell(\Delta_k)=2^{k+2}\ell(J)$, we have
\begin{align*}
1+\frac{|c(J_p)|}{1+2^{k}\ell(J_p)}
&\gtrsim 1+\frac{|c(\Delta_{k})|}{1+\ell(\Delta_k)}
\gtrsim \rdist (\Delta_{k},\mathbb B)
\geq \rdist (2^{k+3}J,\mathbb B)
\end{align*}
with clear meaning of $\rdist (\Delta_{k},\mathbb B)$ despite $\Delta_{k}$ is not a cube. 
Then, since $D$ is non-increasing, we get
\begin{align*}
F(t,x)
&\leq L(\ell(J))
S(\ell(J))D(\rdist (2^kJ,\mathbb B))
=F(J,J,2^kJ).
\end{align*}
With this and $\Delta_{k}\subset 2^{k+3}J$, we continue the bound in \eqref{intin} as
\begin{align*}
\Int
&\lesssim \frac{\ell(J)^{\delta }}{(2^{k}\ell(J))^{\delta+\alpha }}\mu(2^{k+3}J)\mu(S)
F(J,J,2^kJ)
\\
&\lesssim  2^{-k\delta}
\frac{\mu(2^{k+3}J)}
{(2^{k+3}\ell(J))^\alpha }\mu(S)
F(J,J,2^kJ).
\end{align*}

For the second double integral in the right hand size of \eqref{sum2}, which we denote by $\Int'$, we can apply the same reasoning with the only difference of integrating over $I$ instead of $\Delta_k$. With this we obtain 
\begin{align*}
\Int'
&\lesssim \frac{\ell(J)^{\delta }}{(2^{k}\ell(J))^{\delta+\alpha }}\mu(I)\mu(S)
F(J,J,2^kJ)
\\
&\lesssim  2^{-k\delta}
\frac{1}
{(2^{k}\ell(J))^\alpha }\mu(I)\mu(S)
F(J,J,2^kJ).
\end{align*}

Then we continue the estimate in \eqref{sum2} as follows: since $\mu(I)\leq \mu(R)$,
 \begin{align*}
 |\langle T\Phi_{k},\psi_{J}\rangle |
 &\lesssim 2^{-k\delta }  F(J,J,2^{k}J)
 \mu(I)^{\frac{1}{2}}\mu(J)^{\frac{1}{2}}
\\
 &\Big(
 \sum_{\substack{R\in \{ I,I_{p}\}\\
 S\in \{ J,J_{p}\}}}
 \frac{\chi_{R}(c_{J_p})}{\mu(R)}
 \frac{\mu(2^{k+3}J )}{(2^{k+3}\ell(J))^\alpha }
+\sum_{S\in \{ J,J_{p}\}}
 \frac{\chi_{I_p\setminus I}(c_{J_p})}
{(2^{k}\ell(J))^\alpha }
\Big) 
 \\
 &\lesssim 2^{-k\delta }  F(J,J,2^{k}J)
\Big( \sum_{R\in \{ I,I_{p}\}}
 \Big(\frac{\mu(J)\chi_{R}(c_{J_p})}{\mu(R)}\Big)^{\frac{1}{2}}
\frac{\mu(2^{k+3}J )}{(2^{k+3}\ell(J) )^\alpha }
\\
&
\hskip100pt +\chi_{I_p\setminus I}(c_{J_p})\frac{\mu(I)^{\frac{1}{2}}\mu(J)^{\frac{1}{2}}}
{(2^{k}\ell(J))^\alpha }\Big)
 \\
 &\lesssim \sum_{R\in \{ I,I_{p}\}}
 \Big(\frac{\mu(J\cap R)}{\mu(R)}\Big)^{\frac{1}{2}}
 2^{-k\delta }F(J,J,2^{k}J)
\frac{\mu(2^{k+3}J )}{(2^{k+3}\ell(J) )^\alpha }
\\
&
+\chi_{I_p\setminus I}(c_{J_p})\frac{\mu(I)^{\frac{1}{2}}\mu(J)^{\frac{1}{2}}}
{\ell(J)^\alpha }
 2^{-k(\alpha+ \delta )}  F(J,J,2^{k}J)
 \end{align*}

Now, using that 
$F(J, J, 2^kJ)=L(\ell(J))L(\ell(J))
D(\rdist (2^kJ, \mathbb B))$ and summing in $k$, we have
that $|\langle T\psi_{I}^{\rm out},\psi_{J}\rangle |$ can be bounded by
 \begin{align*}
 &\sum_{R\in \{ I,I_{p}\}}\hspace{-.3cm}
 \Big(\frac{\mu(J\cap R)}{\mu(R)}\Big)^{\frac{1}{2}}
 L(\ell(J))S(\ell(J))\hspace{-.1cm}
\sum_{k\geq m_0}2^{-k\delta }\hspace{-.1cm}
 \frac{\mu(2^{k+3}J )}{(2^{k+3}\ell(J) )^\alpha }
 D(\rdist(2^{k}J,\mathbb B))
 \\
 &+\chi_{I_p\setminus I}(c_{J_p})
\frac{\mu(I)^{\frac{1}{2}}\mu(J)^{\frac{1}{2}}}
{\ell(J)^\alpha } L(\ell(J))S(\ell(J))
\sum_{k\geq m_0}2^{-k(\alpha +\delta )}
 D(\rdist(2^{k}J,\mathbb B))
 \end{align*}

If we denote $\lambda =\inrdist(I_p,J_p)$, 
since $m_0=\log \inrdist(I_p,J_p)$, the last factor in each term equals
 \begin{align*}
 &2^{-m_0\delta }\sum_{k\geq 0}2^{-k\delta }
 \frac{\mu(2^{k+3}2^{m_0}J )}{\ell(2^{k+3}2^{m_0}J)^\alpha }
 D(\rdist(2^{k}(2^{m_0}J), \mathbb B))
 \\
 &\lesssim
 \inrdist(I_p,J_p)^{-\delta}
 \sum_{k\geq 0}2^{-k\frac{\delta}{2}}
 D(\rdist(2^{k}\lambda J,\mathbb B))
\sum_{k\geq 0}2^{-k\frac{\delta}{2} }\frac{\mu(2^{k+3}\lambda J)}{(2^{k+3}\ell(\lambda J))^\alpha }
\\
 &
 \lesssim \inrdist(I_p,J_p)^{-\delta}
 \tilde D(\rdist(\lambda J,\mathbb B))
 \rho_{\rm out}(\lambda J),
 \end{align*}
 and 
\begin{align*}
 2^{-m_0(\alpha+\delta )}&\sum_{k\geq 0}2^{-k(\alpha+\delta )}
 D(\rdist(2^{k}(2^{m_0}J), \mathbb B))\\
 &\lesssim
 \inrdist(I_p,J_p)^{-(\alpha +\delta)}
 \sum_{k\geq 0}2^{-k(\alpha +\delta) }
 D(\rdist(2^{k}\lambda J,\mathbb B))
 \\
 &
 \lesssim \inrdist(I_p,J_p)^{-(\alpha +\delta)}
 \tilde D(\rdist(\lambda J,\mathbb B)),
 \end{align*}
with $\tilde D$ defined in \eqref{Dtilde}.
This ends the proof. 

\end{proof}

\section{Paraproducts}\label{Para}
The proof of Theorem \ref{Mainresult2} is divided in two parts: we first deal with the part associated to $\psi_{I,J}^{{\rm full}}$, which corresponds to the paraproduct case in the classical proof. 
Then we use the estimates of the bump lemma for the remaining part. 

For the first part of the proof we will use 
the classical Carleson's Embedding Theorem

\begin{lemma}[Carleson Embedding Theorem]\label{Carleson}
Let $(a_{I})_{I\in \mathcal D}$ be a collection of non-negative numbers. 
Then 
\begin{align}\label{Carleson2} 
\sum_{I\in \mathcal D}a_{I}
|\langle f\rangle_{I}|^{2}
\lesssim \sup_{I\subset \mathcal D}\Big(\frac{1}{\mu(I)}\sum_{\substack{J\in \mathcal D(I)
}}a_{J}\Big)\, \| f\|_{L^{2}(\mu)}^{2}
\end{align}
for all $f\in L^{2}(\mu)$.
\end{lemma}

The following proposition deals with the paraproduct part of the operator. The proof provided follows the work in \cite{NTV4}.
\begin{proposition}[Paraproducts]\label{paraproducts}
Let $Q\in \mathcal C$ and $\theta \in (0,1)$ be fixed. 
We define the following bilinear forms: 
for $f,g$ bounded functions with $\supp f\cup \supp g\subset Q$ and mean zero,
\begin{align*}
\Pi (f,g)
&=\sum_{I\in \mathcal D(Q)}
\sum_{\substack{J\in \mathcal D(I_p)\\
\inrdist(J_p,I_p)>\lambda_\theta}
}
\langle f, \psi_{I}\rangle \langle g, \psi_{J}\rangle \langle T\psi_{I,J}^{{\rm full}},\psi_{J}\rangle
\end{align*}
\begin{align*}
\Pi' (f,g)
&=\sum_{J\in \mathcal D(Q)}
\sum_{\substack{I\in \mathcal D(J_p)\\
\inrdist(I_p,J_p)>\lambda_{\theta}}
}
\langle f, \psi_{I}\rangle \langle g, \psi_{J}\rangle \langle T\psi_{I},\psi_{J,I}^{\rm full}\rangle
\end{align*}
with
$\lambda_\theta =1+\ec(I,J)^{-\theta}$.

Then 
given $\epsilon >0$ there exist 
$M_0\in \mathbb N $ independent of $Q$ and the functions $f, g$ such that for all $M>M_0$
$$
|\Pi (P_M^\perp f,P_M^\perp g) |+|\Pi' (P_M^\perp f,P_M^\perp g) |\lesssim 
\epsilon \| f\|_{L^2(\mu)}\| g\|_{L^2(\mu)}.
$$
\end{proposition}
\begin{proof}
By symmetry, we only need to work with $\Pi$. 
By writing $f$ as the sum of $2^n$ 
functions (the restriction of $f$ to each quadrant), we can assume that $\supp f\subset Q'\in \mathcal D$. 

By Lemma \ref{densityinL2}, we can 
by rewriting $\Pi (P_M^\perp f,P_M^\perp g)$ in the following way
\begin{align*}
\Pi (P_M^{\perp} f,P_M^{\perp}g)
&
=\sum_{J\in \mathcal D_M^c(Q)} \langle g,\psi_{J}\rangle
\langle T( \hspace{-.5cm}
\sum_{\substack{I\in \mathcal D_M^c(Q)\\ 
J_p\subset I_p, 
\inrdist(J_p,I_p)>\lambda_\theta}}
\langle f, \psi_{I}\rangle\psi_{I,J}^{{\rm full}}
),\psi_{J}\rangle .
\end{align*}
We can assume that the first sum in previous expression only contains terms for which $\mu(J_p)\neq 0$ since otherwise $\psi_J\equiv 0$. 
Then, since in the second sum we have 
$\ell(J_p)<\ell(I_p)$, 
by Lemma \ref{decompfull}, 
we have 
$$
\sum_{I\in \child(I_p)}
\langle f, \psi_{I}\rangle 
\psi_{I,J}^{{\rm full}}
=\hat{\Delta}_{I_p}(f)
=\Big(\sum_{I\in \child(I_p)}\hat E_{I}f\Big)-\hat E_{I_p}f,
$$
where 
$\hat E_{I}f=\langle f\rangle _{I}\chi_{I}(c_{J_{p}})\chi_{Q}$.


The inner sum takes place under the condition 
$\inrdist(J_p,I_p)>\lambda_{\theta}\geq 2$. 
Then, 
for $J_p\in \mathcal{D}_M^c(Q)$, 
let 
$\lambda $ be the smallest integer such that 
$\inrdist(J_p,I_p)>\lambda$. Let also 
$J_{\lambda}\in \mathcal{D}(Q)$ 
be the smallest cube such that 
$J_p\subset J_{\lambda}$ and $\inrdist(J_p,J_{\lambda})>\lambda$.  
If such cube does not exist, we then define 
$J_{\lambda}=\emptyset $. 
Now, 
by summing a telescopic sum, we have for $J\in {\mathcal D}_M^c(Q)$ fixed and $M\in \mathbb N$ such that $2^M>\ell(Q)$, 
\begin{align*}
\sum_{\substack{I\in \mathcal D_M^c(Q)\\ 
J_p\subset I_p, 
\inrdist(J_p,I_p)>\lambda}}
\langle f, \psi_{I}\rangle 
\psi_{I,J_p}^{{\rm full}}
&=\sum_{\substack{I_p\in \mathcal D_M^c(Q)\\ 
J_p\subset I_p, 
\inrdist(J_p,I_p)>\lambda}}
\sum_{I'\in \child(I_p)}
\langle f, \psi_{I'}\rangle 
\psi_{I',J_p}^{{\rm full}}
\\
&=\sum_{\substack{I_p\in \mathcal D_M^c(Q)\\ 
J_p\subset I_p, 
\inrdist(J_p,I_p)>\lambda}}\hat\Delta_{I_p}(f)
\\
&=\sum_{R\in \child(J_{\lambda})}
\hat E_{R}f
-\hat E_{Q'}f
\\
&
=\sum_{R\in \child(J_{\lambda})}
\langle f\rangle _{R_J}
\chi_{Q},
\end{align*}
where $Q'\in \mathcal D$ such that 
$\sup f\subset Q'$ and so, 
$\langle f\rangle_{Q'}=0$. 
The cardinality of $\child(J_{\lambda})$ is $2^n$ and so, 
we can enumerate the family in a uniform way accordingly with their position inside $J_{\lambda}$: 
$\child(J_{\lambda})=\{ J_j\}_{j=1}^{2^n}$. 
With this
\begin{align*}
\Pi (P_M^{\perp}f,P_M^{\perp}g)
&=\sum_{J\in \mathcal D_M^c(Q)} \langle g,\psi_{J}\rangle
\langle T(\sum_{R\in \child(J_{\lambda})}\hat E_{R}f)
,\psi_{J}\rangle 
\\
&=\sum_{j=1}^{2^n}
\sum_{J\in \mathcal D_M^c(Q)} 
\langle f\rangle _{J_j}
\langle g,\psi_{J}\rangle 
\langle T\chi_{Q},\psi_{J}\rangle .
\end{align*}

With this, boundedness of $\Pi (P_M^{\perp}f,P_M^{\perp}g)$ follows once we obtain for each $j$ fixed a uniform estimate for  
\begin{align*}
\Pi_{j} (P_M^{\perp}f,P_M^{\perp}g)
&=\sum_{J\in \mathcal D_M^c(Q)} 
\langle g,\psi_{J}\rangle \langle f\rangle _{J_j}
\langle T\chi_{Q}, \psi_{J}\rangle .
\end{align*}
By Cauchy's inequality and Lemma \ref{Planche}, 
\begin{align*}
|\Pi_{j} (P_M^{\perp}f,P_M^{\perp}g)|
&\leq \Big(\sum_{J\in {\mathcal D}_M^c(Q)} |\langle f\rangle _{J_j}|^{2}
|\langle T\chi_{Q}, \psi_{J}\rangle |^{2} \Big)^{\frac{1}{2}}
\Big(\sum_{J\in {\mathcal D}_M^c(Q)} 
|\langle g,\psi_{J}\rangle |^{2}\Big)^{\frac{1}{2}}
\\
&
\lesssim \Big(\sum_{J\in {\mathcal D}_M^c(Q)} |\langle f\rangle _{J_j}|^{2}
|\langle T\chi_{Q}, \psi_{J}\rangle |^{2} \Big)^{\frac{1}{2}}\| g\|_{L^2(\mu)}
\end{align*}
For each $J\in {\mathcal D}_M^c(Q)$ there is a unique $J_j\in \mathcal D_M^c(Q)$
such that $J_j\in \child{(J^{\lambda})}$. But now we consider a cube $J_j\in \mathcal D_M^c(Q)$ in fixed position depending on $j$ and define 
$\mathcal J(J_j)$ as the family of
all cubes $J\in {\mathcal D}_M^c(Q)$ such that for the corresponding cube $J_{\lambda}$ we have 
$J_j\in \child{(J_{\lambda})}$.

We now show that  
$\mathcal J(J_j)\cap \mathcal J(J_j')=\emptyset $ for $J_j\neq J_j'$. 
If there is $J\in \mathcal J(J_j)\cap \mathcal J(J_j')$ then the corresponding cube 
$J_{\lambda}$ satisfies 
$J\subset J_{\lambda}$ with 
$J_j\in \child{J_{\lambda}}$
and $J_j'\in \child{J_{\lambda}}$. But then, since the position indicated by $j$ is fixed, we have
$J_j=J_j'$, which is contradictory.

Then by Lemma \ref{Carleson}, 
\begin{align*}
|&\Pi_{j} (P_M^{\perp}f,P_M^{\perp}g)|
\lesssim
\Big(\sum_{J_j\in {\mathcal D}_M^c(Q)} |\langle f\rangle _{J_j}|^{2}
\sum_{J\in \mathcal J(J_j)}
|\langle T\chi_{Q}, \psi_{J}\rangle |^{2} \Big)^{\frac{1}{2}}\| g\|_{L^2(\mu)}
\\
&
\lesssim
\sup_{R\in \mathcal D_M^c(Q)}\Big(\mu(R)^{-1}
\hspace{-.5cm}
\sum_{\substack{J_j\in \mathcal D_M^c(R)\\ \inrdist(J,R)>1}}
\hspace{-.2cm}\sum_{J'\in \mathcal J(J_j)}
|\langle T\chi_{Q}, \psi_{J'}\rangle |^{2}\Big)^{\frac{1}{2}}
\| f\|_{L^2(\mu)}\| g\|_{L^2(\mu)},
\end{align*}
were more terms were added to the first factor to get the condition $\inrdist(J,R)>1$.
We now prove that for all $R\in \mathcal D_M^c(Q)$,
$$
\sum_{\substack{J_j\in \mathcal D_M^c(R)\\ \inrdist(J,R)>1}}
\sum_{J'\in \mathcal J(J_j)}
|\langle T\chi_{Q}, \psi_{J'}\rangle |^{2}
\lesssim \epsilon \mu(R).$$

For $R\in \mathcal D_M^c(Q)$ with 
$\ell(R)<2^{-M}$ or $\rdist (R,\mathbb B_{2^M})>M$, 
we construct  
$\mathcal{W}(R)$ a Whitney decomposition of $R$ defined by the maximal (with respect to the inclusion) dyadic cubes $S$ such that $\inrdist(S,R)>1$ and 
$3S\subset R$. The cubes $S$ in 
$\mathcal{W}(R)$ form a partition of $R$ and for any cube $J\in \mathcal D_M^c(R)$ such that  $\inrdist(J,R)>1$ there exists
$S\in \mathcal{W}(R)$ such that 
$J_p\subset S$. 
Then  
we can write
\begin{align*}
\sum_{J\in \mathcal D_M^c(R)} \sum_{J'\in \mathcal J(J)}&
|\langle T\chi_{Q}, \psi_{J'}\rangle |^{2}
\leq \sum_{S\in \mathcal{W}(R)}\sum_{J\in \mathcal D_M^c(S)} \sum_{J'\in \mathcal J(J)}
|\langle T\chi_{Q}, \psi_{J'}\rangle |^{2}
\\
&\leq \sum_{S\in \mathcal{W}(R)}
\Big(\sum_{J\in \mathcal D_M^c(S)}
\sum_{J'\in \mathcal J(J)}
|\langle T\chi_{2S}, \psi_{J'}\rangle |^{2}
\\
&
+\sum_{J\in \mathcal D_M^c(S)} 
\sum_{J'\in \mathcal J(J)}
|\langle T(\chi_{Q\setminus 2S}), \psi_{J'}\rangle |^{2}
\Big).
\end{align*}
We will later deal with each term in different ways. 

On the other hand, for $R\in \mathcal D_M^c(Q)$ with 
$\ell(R)>2^{M}$, we 
start by defining the same Whitney decomposition of $R$ as before $\mathcal{W}(R)$. But then
we decompose each $S\in \mathcal{W}(R)$ as follows: 
$$
S=\bigcup_{\substack{\bar S\in \mathcal D(S)\\ \ell( \bar S)=2^{-(M+2)}}}\bar S .
$$
This ensures that $\bar S\in \mathcal D_{M}^{c}(Q)$. 
Then, similarly as before, we write
\begin{align*}
\sum_{J\in \mathcal D_M^c(R)} \sum_{J'\in \mathcal J(J)}&
|\langle T\chi_{Q}, \psi_{J'}\rangle |^{2}
\leq \sum_{S\in \mathcal{W}(R)}
\sum_{\substack{\bar S\in \mathcal D(S)\\ \ell(\bar S)=2^{-(M+1)}}}
\sum_{J\in \mathcal D_M^c(\bar S)} \sum_{J'\in \mathcal J(J)}
|\langle T\chi_{Q}, \psi_{J'}\rangle |^{2}
\\
&\leq \sum_{S\in \mathcal{W}(R)}
\Big( \sum_{\substack{\bar S\in \mathcal D(S)\\ \ell(\bar S)=2^{-(M+1)}}}
\sum_{J\in \mathcal D_M^c(\bar S)}
\sum_{J'\in \mathcal J(J)}
|\langle T\chi_{2\bar S}, \psi_{J'}\rangle |^{2}
\\
&
+\sum_{\substack{\bar S\in \mathcal D(S)\\ \ell(\bar S)=2^{-(M+1)}}}
\sum_{J\in \mathcal D_M^c(\tilde S)} 
\sum_{J'\in \mathcal J(J)}
|\langle T(\chi_{Q\setminus 2\bar S}), \psi_{J'}\rangle |^{2}
\Big).
\end{align*}
As before, we also deal with each term differently. In fact, we will show in detail the first case and only the small differences of the second case. 

In the first case, 
for each $S\in \mathcal{W}(R)$, we use the fact the 
the families $\mathcal J(J)$ are pairwise disjoint, Lemma \ref{Planche}, 
and the testing condition, to estimate 
the inner double sum in the first term as follows:
\begin{align*}
\sum_{J\in \mathcal D_M^c(S)}
\sum_{J'\in \mathcal J(J)}
|\langle T\chi_{2S}, \psi_{J'}\rangle |^{2}
&\leq 
\sum_{J'\in \mathcal D_M^c(S)}
|\langle \chi_{2S}T\chi_{2S}, \psi_{J'}\rangle |^{2}
\\
&\lesssim 
\| \chi_{2S}T\chi_{2S}\|_{L^2(\mu)}^2
\\
&\lesssim \mu(2S)F_\mu(2S)^2. 
\end{align*}

Since $2S\subset R\in \mathcal D_M^c(Q)$ and 
$\ell(R)<2^{-M}$ or $\rdist (R,\mathbb B_{2^M})>M$, 
we have $\ell(2S)<2^{-M}$ or $\rdist (2S,\mathbb B_{2^M})>M$. 
Therefore $2S\in \mathcal D_M^c(Q)$ and so,
$F_\mu(2S)\leq \sup \{F_\mu(K): K\in \mathcal D_M^c\}<\epsilon $. Then, by summing in $S\in \mathcal{W}(R)$, we have
\begin{align*}
\sum_{S\in \mathcal{W}(R)}\sum_{J\in \mathcal D_M^c(S)}
\sum_{J'\in \mathcal J(J)}
|\langle T\chi_{2S}, \psi_{J'}\rangle |^{2}
&
\lesssim \epsilon^2 \sum_{S\in \mathcal{W}(R)}\mu(2S)
\lesssim \epsilon^2 \mu(R).
\end{align*}
In the last inequality we used that, since 
$3S\subset R$ and the cubes $S$
are disjoint by maximality,  
the cubes $2S$ can only overlap a uniform amount of times, as we show. 

Since $3S\subset R$ and $6S\not\subset R$, we have $\dist(2S,\partial R)\geq \ell(S)/2$ and 
$\dist(S,\partial R)\leq 5\ell(S)$ respectively. 
Then, for all $S$ such that $x\in 2S$ we have 
$\dist(x,\partial R)\geq \dist(2S,\partial R)\geq \ell(S)/2$ and 
$\dist(x,\partial R)\leq 5\ell(S)+\dist(S,\partial R)\lesssim 16\ell(S)$. 

Now we reason as follows. For each $x\in R$ by disjointness there is at most one cube $S_0$ such that $x\in S_0\subset 2S_0$. With this,
\begin{itemize}
\item if $\ell(S)\leq \ell(S_0)/64$, then 
$\ell(S)\leq 2\dist(x,\partial R)/64$ and so \\
$\dist(x,\partial R)\geq 32\ell(S)$, 
which implies that $x\notin 2S$. 

\item If $\ell(S)\geq 64\ell(S_0)$, then 
$\ell(S)\geq 64\dist(x,\partial R)/16$ and so \\
$\dist(x,\partial R)\leq \ell(S)/4$, 
which implies that $x\notin 2S$. 
\end{itemize}

Therefore, there are up to 12 size lengths of $S$ for which $x\in 2S$. In addition, 
since the cubes $S$ are disjoint, for each $x\in R$ there are up to $3^n$ cubes $S$ of a fixed size length
such that $x\in 2S$. With this, we get that there are in total up to $12\cdot 3^n$ different cubes $S$ such that 
$x\in 2S$.

For the second term, we reason as follows. 
Let $S\in \mathcal{W}(R)$, $J\in \mathcal D_M^c(S)$ and 
$J'\in \mathcal J(J)$ be fixed. 
Since $J_p'\subset S$ and $\psi_{J'}$ has mean zero, we can write
\begin{align*}
|\langle T(\chi_{Q\setminus 2S}), \psi_{J'}\rangle |
&=|\langle T(\chi_{Q\setminus 2S})-T(\chi_{Q\setminus 2S})(c_{J'_p}), \psi_{J'}\rangle |
\\
&\leq \int_{Q\setminus 2S}\int_{J'_p}
|K(t,x)-K(t,c_{J'_p})||\psi_{J'}(x)|d\mu(t)d\mu(x)
\\
&\leq \int_{Q\setminus 2S}
\int_{J_p}\frac{|x-c_{J'_p}|^\delta }{|t-x|^{\alpha +\delta }}F(t,x)|\psi_{J'}(x)|d\mu(t) d\mu(x),
\end{align*}
where $F(t,x)=L(|t-x|)S(|x-c_{J'_p}|)D(1+\frac{|t+c_{J'_p}|}{1+|t-c_{J'_p}|})$.
Now $|t-x|\gtrsim \ell(S)$ and 
$|x-c_{J'_p}|\leq \ell(J'_p)/2$. 
Then
$2|x-c_{J'_p}|\leq \ell(J'_p)\leq \ell(S)\leq |t-x|$. Moreover 
$|t-x|\gtrsim |t-c_{J'_p}|-|x-c_{J'_p}|\geq |t-c_{J'_p}|-|t-x|/2$, that is,  $|t-x|\gtrsim 2|t-c_{J'_p}|/3$. 
With this, we have
\begin{align*}
|\langle T(\chi_{Q\setminus 2S})&, \psi_{J}\rangle |
\leq L(\ell(S))S(\ell(S))\int_{J'_p}|\psi_{J'}(x)|
d\mu(x)
\\&
\hspace{1cm}\int_{Q\setminus 2S}\frac{\ell(J')^{\delta }}{|t-c_{J'_p}|^{\alpha +\delta }}
D(1+\frac{|t+c_{J'_p}|}{1+|t-c_{J'_p}|})
d\mu(t)
\\
&\leq 
\ell(J')^{\delta }\mu(J'_p)^{\frac{1}{2}}
L(\ell(S))S(\ell(S))
\int_{Q\setminus 2S}\frac{D(1+\frac{|t+c_{J'_p}|}{1+|t-c_{J'_p}|})}
{|t-c_{J'_p}|^{\alpha +\delta }}
d\mu(t).
\end{align*}
Since $J'\subset S$,  
for $t\in Q\setminus 2S$ we have 
$\dist(t,J')\geq \ell(S)$. Then 
we decompose 
$$
Q\setminus 2S=\bigcup_{i=1}^{\log \frac{\ell(Q)}{\ell(S)}}
S_i
$$
where 
$S_i=\{ t\in Q\setminus 2S : 2^{i-1}\ell(S)<
|t-c_{J_p'}|\leq 2^{i}\ell(S)
\}\subset 2^{i+1}S$. Note that 
$S_i\subset B(c_i,2^{i}\ell(S))$ with 
$c_i\in S_i$, and $c(S_i)=c_{J'_p}$. 
Moreover, 
since 
$|t-c_{J'_p}|+|t+c_{J'_p}|\geq 2|c_{J'_p}|$, 
we have
$$
1+\frac{|t+c_{J'_p}|}{1+|t-c_{J'_p}|}\geq 
1+\frac{2|c_{J'_p}|}{1+|t-c_{J'_p}|}\geq 
1+\frac{|c(2^iS)|}{1+2^i\ell(S)}. 
$$
Then
$D(1+\frac{|t+c_{J'_p}|}{1+|t-c_{J'_p}|})
\lesssim D(1+\frac{|c(2^iS)|}{1+2^i\ell(S)})\lesssim 
D(\rdist (2^iS, \mathbb B))$. With this
\begin{align*}
\int_{Q\setminus 2S}\frac{D(1+\frac{|t+c_{J'_p}|}{1+|t-c_{J'_p}|})}
{|t-c_{J'_p}|^{\alpha +\delta }}
d\mu(t)
&\lesssim \sum_{i=1}
\frac{D(\rdist (2^iS, \mathbb B))}{(2^{i}\ell(S))^{\alpha +\delta }}\mu(S_i)
\\
&\lesssim \sum_{i=1}
\frac{1}{(2^{i}\ell(S))^{\delta }}
\frac{\mu(2^{i+1}S)}{(2^{i+1}\ell(S))^{\alpha }}
D(\rdist (2^iS, \mathbb B))
\\&
\lesssim \frac{1}{\ell(S)^{\delta}}\tilde D(S)\rho_{\mu}(S).
\end{align*}
Then
\begin{align*}
|\langle T(\chi_{Q\setminus 2S}), \psi_{J'}\rangle |
&\lesssim \Big(\frac{\ell(J')}{\ell(S)}\Big)^{\delta }\mu(J'_p)^{\frac{1}{2}}
L(\ell(S))S(\ell(S))
\tilde D(S)\rho_{\mu}(S)
\\
&\leq  \Big(\frac{\ell(J')}{\ell(S)}\Big)^{\delta }\mu(J'_p)^{\frac{1}{2}}F_{\mu}(S)
\leq 
\epsilon\Big(\frac{\ell(J')}{\ell(S)}\Big)^{\delta }\mu(J'_p)^{\frac{1}{2}}.
\end{align*}
The last inequality is due to the fact that, 
as we saw before, 
$S\in \mathcal D_M^c(Q)$, and then 
$F_\mu(S)\leq \sup \{F_\mu(K): K\in \mathcal D_M^c(Q)\}<\epsilon $. 
Now, we parametrize the cubes $J', J$ according with their relative size with respect to $S$: $\ell(J)=2^{-k}\ell(S)$
and $\ell(J')=2^{-k'}\ell(J)$, which imply
$\ell(J')=2^{-(k+k')}\ell(S)$. 
We sum in $J$ and $J'$ and use that the cubes with fixed side lenght are disjoint, to get
\begin{align*}
\sum_{J\in \mathcal D_M^c(S)} 
&\sum_{J'\in \mathcal J(J)}
|\langle T(\chi_{Q\setminus 2S}), \psi_{J}\rangle |^{2}
\lesssim \epsilon^2
\sum_{J\in \mathcal D(S)} 
\sum_{J'\in \mathcal J(J)}
\Big(\frac{\ell(J')}{\ell(S)}\Big)^{2\delta }\mu(J')
\\
&\lesssim \epsilon^2
\sum_{k\geq 1}2^{-k2\delta}
\sum_{\substack{J\in {\mathcal D}(S)\\ \ell(J)=2^{-k}\ell(S)}}
\sum_{k'\geq 1}2^{-k'2\delta}
\sum_{\substack{J\in \mathcal J(J)\\ 
\ell(J)=2^{-k'}\ell(S')}}
\mu(J')
\\
&\leq \epsilon^2
\sum_{k\geq 1}2^{-k2\delta}
\sum_{\substack{J\in {\mathcal D}(S)\\ \ell(J)=2^{-k}\ell(S)}}
\sum_{k'\geq 1}2^{-k'2\delta}
\mu(J)
\\
&\lesssim \epsilon^2
\sum_{k\geq 1}2^{-k2\delta}
\sum_{\substack{J\in {\mathcal D}(S)\\ \ell(J)=2^{-k}\ell(S)}}
\mu(J)
\\
&
\leq \epsilon^2
\sum_{k\geq 1}2^{-k2\delta}\mu(S)
\lesssim \mu(S)
\epsilon^2.
\end{align*}
Summing now over the cubes $S$ in $\mathcal{W}(R)$, we finally get
\begin{align*}
\sum_{S\in \mathcal{W}(R)}
\sum_{J\in \mathcal D_M^c(S)} 
\sum_{J'\in \mathcal J(J)}
|\langle T(\chi_{Q\setminus 2S}), \psi_{J'}\rangle |^{2}
&\lesssim \epsilon^2\sum_{S\in \mathcal{W}(R)}
 \mu(S)
\\
&\lesssim \epsilon^2\mu(R).
\end{align*}

When $\ell(R)>2^{M}$, the reasoning is similar with very few modifications. In this case we have for the first term
\begin{align*}
\sum_{\substack{\bar S\in \mathcal D(S)\\ \ell(\tilde S)=2^{-(M+2)}}}
&\sum_{J\in \mathcal D_M^c(\bar S)}
\sum_{J'\in \mathcal J(J)}
|\langle T\chi_{2\bar S}, \psi_{J'}\rangle |^{2}    
\\
&\leq 
\sum_{\substack{\tilde S\in \mathcal D(S)\\ \ell(\bar S)=2^{-(M+2)}}}
\sum_{J'\in \mathcal D_M^c(\bar S)}
|\langle \chi_{2\bar S}T\chi_{2S}, \psi_{J'}\rangle |^{2}
\\
&\lesssim 
\sum_{\substack{\bar S\in \mathcal D(S)\\ \ell(\bar S)=2^{-(M+2)}}}
\| \chi_{2\bar S}T\chi_{2\bar S}\|_{L^2(\mu)}^2
\\
&\lesssim 
\sum_{S\in \mathcal{W}(R)}
\sum_{\substack{\bar S\in \mathcal D(S)\\ \ell(\bar S)=2^{-(M+2)}}}
\mu(2\bar S)F_\mu(2\bar S)^2. 
\end{align*}
Since $\ell(2\bar R)=2^{-(M+1)}$ we have $2\bar S\in \mathcal D_{M}^{c}$ and so $F_\mu(2\bar S)<\epsilon$. Then 
\begin{align*}
\sum_{S\in \mathcal{W}(R)}
\sum_{\substack{\bar S\in \mathcal D(S)\\ \ell(\bar S)=2^{-(M+2)}}}
&\sum_{J\in \mathcal D_M^c(\bar S)}
\sum_{J'\in \mathcal J(J)}
|\langle T\chi_{2\bar S}, \psi_{J'}\rangle |^{2} 
\\
&
\lesssim \epsilon^2
\sum_{S\in \mathcal{W}(R)}
\sum_{\substack{\bar S\in \mathcal D(S)\\ \ell(\bar S)=2^{-(M+2)}}}\mu(2\bar S)
\\
&\lesssim \epsilon^2
\sum_{S\in \mathcal{W}(R)}\mu(2S)
\lesssim \epsilon^2\mu(R), 
\end{align*}
and we continue as before. In a similar way we can estimate the second term.

To finish the proof, we still need to prove that $\Pi  (P_M^{\perp} f,P_M^{\perp}g)$ belongs to the class of operators for which the theory applies.
In particular, we must show that
the integral representation of Definition \ref{intrep} holds with
a kernel satisfying the Definition \ref{prodCZoriginal} of a compact Calder\'on-Zygmund kernel. This work is independent of the measure $\mu$ and it can be done in exactly the same way it was performed in \cite{V2}.

\end{proof}

\section{$L^{p}$ compactness}\label{L2}
In this section we develop the proof of the main result, Theorem \ref{Mainresult2}. 
We start with a technical lemma needed in the proof of Theorem \ref{Mainresult2}. 
The lemma shows that the regions that are sufficiently close to the border of an open dyadic cube have arbitrarily small measure. 
\begin{notation}
For $N\in \mathbb N$, we define the following two collections of dyadic cubes:  
$\mathcal{D}(Q)_{\geq N}=\{ I\in \mathcal{D}_M^c(Q)/ \ell(I)\geq 2^{-N}\ell(Q)\}$  
and $\mathcal{D}(Q)_{N}=\{ I\in \mathcal{D}_M^c(Q)/ \ell(I)=2^{-N}\ell(Q)\}$.
\end{notation}

\begin{lemma}\label{zeroborder}
Let 
$\mu$ be a positive Radon measure in $\mathbb R^n$ with power growth $0<\alpha \leq n$. 
Let $Q\in \mathcal D$,  
$N_0, M \in \mathbb N$ and $\theta \in (0,1)$ be fixed. 

Let $I\in \mathcal{D}(Q)_{\geq N_0}$ with 
$\ell(R)=2^{-k_I}\ell(Q)$, $0\leq k_I\leq N_0$. For 
$k\geq k_I$  
let $C_k(I)$ be the union of the interior of all cubes 
$R\in \mathcal{D}(3I)$ such that $\ell(R)=2^{-k}\ell(Q)\leq \ell(I)$ 
and $\inrdist(R,I)<1+\ec(I,R)^{-\theta} $. 
Let finally $C_k=\bigcup_{I\in \mathcal{D}(Q)_{\geq N_0}}  C_k(I)$.
   
Then 
for each $\epsilon>0$ there exist $k_0\in \mathbb N$ such that $\mu(C_k)<\epsilon$ for all $k>k_0$. 

\end{lemma}
\begin{proof}

We start by noting that the family of cubes $\mathcal{D}(Q)_{\geq N_0}$ has cardinality 
less than $2^{N_0+1}$. 
Let $I\in \mathcal{D}(Q)_{\geq N_0}$ be fixed. 

We remind that $\mathfrak{D}_{I}=
\cup_{I'\in \child(I)}\partial I'$. Then, 
for each cube $R$ in the definition of 
$C_k(I)$, 
the condition $\inrdist(R,I)-1\leq \ec(I,R)^{-\theta} $ implies 
$$
\frac{\dist(R,\mathfrak{D}_{I})}{\ell(R)}
=\inrdist(R,I)-1
\leq \Big(\frac{\ell(R)}{\ell(I)}\Big)^{-\theta},
$$
that is
$$
\dist(R,\mathfrak{D}_{I})
\leq \Big(\frac{\ell(R)}{\ell(I)}\Big)^{1-\theta}\ell(I).
$$
Since $\ell(I)=2^{-k_I}\ell(Q)$ and $\ell(R)=2^{-k}\ell(Q)$
then 
$$
\dist(R,\mathfrak{D}_{I})
\leq 2^{-(k-k_I)(1-\theta)}\ell(I). 
$$

Now, for $j\geq k$ we define the set 
$$
D_{j}(I)=\{x\in 3I/ 
2^{-(j-k_I+1)(1-\theta)}\ell(I)<\dist(x,\mathfrak{D}_{I})
\leq 2^{-(j-k_I)(1-\theta)}\ell(I)\}. 
$$


Then the sets $(D_j(I))_{j\geq k_I}$ are pairwise disjoint and for each $k>k_I$, 
$\bigcup_{j\geq k}D_j(I)\subset C_k(I)\subset \bigcup_{j\geq k-1}D_j(I)$. 
Then
$$
\sum_{j\geq k}\mu(D_j(I))\leq \mu(C_k(I))\leq \mu(C_k)\leq \mu(Q)<\infty .
$$
Then, for any $\epsilon>0$ there 
exists  
$k_{0,I}\geq k_I$ dependent on $I$ such that 
$$
\sum_{j\geq k}\mu(D_j(I))<2^{-(N_0+1)}\epsilon.
$$
for all $k\geq k_{0,I}$. 

Let now $k_0=\max\{ k_{0,I}: I\in \mathcal{D}(Q)_{\geq N_0}\}$. Since $C_{k+1}\subset C_k$ for each $k\in \mathbb N$, we have for all $k>k_0$
\begin{align*}
\mu(C_k)&\leq 
\mu(C_{k_{0}})\leq 
\sum_{I\in \mathcal{D}(Q)_{\geq N_0}}\mu(C_{k_{0}}(I))
\\&
\leq \sum_{I\in \mathcal{D}(Q)_{\geq N_0}}\mu(\bigcup_{j\geq k_0-1\geq k_{0,I}-1}D_j(I))
\leq \sum_{I\in \mathcal{D}(Q)_{\geq N_0}}\sum_{j\geq k_{0,I}}\mu(D_j(I))
\\&
<\sum_{I\in \mathcal{D}(Q)_{\geq N_0}}2^{-(N_0+1)}\epsilon <\epsilon.
\end{align*}


\end{proof}


We finally proceed  with the proof of the main result in the paper, Theorem \ref{Mainresult2}. 
\begin{proof}[Proof of Theorem \ref{Mainresult2}]
The necessity of the hypotheses can be shown in a similar way as it was done in \cite{V2}. Then we focus on their sufficiency. 

Once boundedness is proved on $L^2(\mu)$, a classical argument that applies to Calder\'on-Zygmund operators, allows to extend the result to weak estimates from $L^1(\mu)$ to  $L^{1,\infty}(\mu)$. Then by a standard interpolation argument one can prove  
boundedness on $L^p(\mu)$ for all $1<p<\infty $.  
Moreover, as shown \cite{V}, we can deduce compactness on $L^{p}(\mu)$ for all $1<p<\infty $ by interpolation between compactness on $L^{2}(\mu)$ and boundedness on  $L^{p}(\mu)$. 
For all this we only focus on the case $p=2$. 


Let $Q\in \mathcal C$ with $c(Q)=0$, $\ell(Q)>2$. Let  
$0<\gamma <\ell(Q)$ 
and $N_0=\log \frac{6\ell(Q)^{\alpha+2}}{\gamma^{\alpha+1}}$. Let 
$T_{\gamma, Q}$ be the truncated operator of Definition \ref{TgammaQ}.

We start by considering the dyadic grid $\mathcal D=\mathcal D^1$ as denoted in Notation \ref{grids}.
Let $(\psi_{I})_{I\in {\mathcal D}}$ 
be the Haar wavelets frame of Definition \ref{defpsi} 
and $P_{M}$ 
be the lagom projection operators related to that system.  
We also fix 
the parameter $\theta=\frac{\alpha }{\alpha +\delta/2} \in (0,1)$.

We aim to prove that  $T_{\gamma, Q}$ are uniformly compact on $L^2(\mu)$ with bounds independent of $Q\in \mathcal C$ and $0<\gamma<\ell(Q)$. By the comments  
at the end of subsection \ref{charaofcomp}, we need to show that
for any $\epsilon >0$ there exists $M_0\in \mathbb N$ (independent of $Q$ and $\gamma$) such that  $\|P_{M}^{\perp}T_{\gamma, Q} P_{M}^{\perp}\|_2\lesssim \epsilon $
for all $M>M_0$, with implicit constant   
independent of $Q$ and $\gamma$ (it may depend on $\delta $ and the constants appearing in the kernel smooth condition and the testing conditions).
This is equivalent to show that 
\begin{align}\label{tmu}
|\langle T_{\gamma, Q} P_{M}^{\perp}f,P_{M}^{\perp}g\rangle| \lesssim \epsilon
\end{align}
for all $M>M_0$, all $f,g$ functions in the unit ball of $L^{2}(\mu)$, bounded, compactly supported on $Q$,  
and with mean zero with respect to $\mu$.

Then let $f$ and $g$ be fixed functions
in the unit ball of $L^{2}(\mu)$ bounded, supported on $Q$, with mean zero with respect to $\mu$, 
and satisfying 
$\|P_{M}^{\perp}T_{\gamma, Q} P_{2M}^{\perp}\|_2\leq 2
|\langle P_{M}^{\perp}T_{\gamma, Q} P_{2M}^{\perp}f,g\rangle |$.

By writing $g$ as the sum of $2^n$ functions each, its restriction each quadrant, we can assume that there exist $Q'\in \mathcal D$ such that 
$\supp g\subset Q'$. 

Let $0<\epsilon <((\|f\|_{L^\infty (\mu)}+\|g\|_{L^\infty (\mu)})\mu(2Q)^{\frac{1}{2}})^{-4}$ be fixed. Let $M_0$ such that for all $M>M_0$ we have 
$M^{-\frac{\delta }{8}}
+M^{-\alpha (\frac{\alpha +\delta}{\alpha+\delta/2}-1)}
+M^{-\frac{\alpha \delta}{\alpha +\delta/2}}
<\epsilon $ and 
$$
\sup_{\substack{J\in  \mathcal D_{M}^{c}\\
(I,J) \in \mathcal F_{M}}}F_\mu(I,J)
+F_\mu(J,I)
<\epsilon , 
$$
where $\mathcal D_M$ is given in Definition \ref{FM}.

Then, for $\epsilon >0$ fixed and chosen $M_{0}\in \mathbb N$, we are going to prove that 
\begin{equation}\label{M2M}
|\langle T_{\gamma, Q}P_{2M}^{\perp }f,P_{M}^{\perp}g\rangle |\lesssim \epsilon^{1/4} 
\end{equation}
for all 
$M>M_{0}$, 
which is also enough for our purposes.
%
%
%
%
To simplify notation, from now we denote the operator 
$T_{\gamma, Q}$ simply by $T$.


By Lemma  \ref{orthorep}, we have that 
\begin{align}\label{orthopro}
\langle TP_{2M}^{\perp }f,P_{M}^{\perp}g\rangle
=\hspace{-.3cm}
\sum_{I\in {\cal D}_{2M}^{c}(Q)}\sum_{J\in {\cal D}_{M}^{c}(Q)}
\langle f,\psi_{I}\rangle 
\langle g, \psi_{J}\rangle
\langle T\psi_I,\psi_J\rangle .
\end{align}

As in Lemma \ref{zeroborder}, we set up the following notation: 
for $N\in \mathbb N$, let
$\mathcal{D}(Q)_{\geq N}=\{ I\in \mathcal{D}_M^c(Q)/ \ell(I)\geq 2^{-N}\ell(Q)\}$; for 
$I\in \mathcal{D}(Q)_{\geq N_0}$ and 
$k>N_0$, $C_k(I)$ denotes the union of the interior of all cubes 
$R\in \mathcal{D}(Q)$ such that $\ell(R)=2^{-k}\ell(Q)\leq \ell(I)$ and 
$\inrdist(R,I)-1\leq \ec(I,R)^{-\theta}$; finally 
$C_k=\bigcup_{I\in \mathcal{D}(Q)_{\geq N_0}}
C_k(I)$. 

Now, by Lemma \ref{zeroborder} and the implicit limit in the equality at \eqref{orthopro}, we can choose $N_1>N_0+M$ so that for all $N>N_1$,
\begin{align}\label{Small}
\mu(C_N)^{\frac{1}{2}}\| f\|_{L^{\infty }(\mu)}\| g\|_{L^{\infty }(\mu)}
2^{N_0(n+3)}<\epsilon ,
\end{align}

and 
\begin{align*}
|\langle TP_{2M}^{\perp }f,P_{M}^{\perp}g\rangle |
&\leq 2\Big|\sum_{I\in {\cal D}_{2M}^{c}(Q)_{\geq N}}\sum_{J\in {\cal D}_{M}^{c}(Q)_{\geq N}}
 \langle f,\psi_{I}\rangle 
 \langle g, \psi_{J}\rangle
 \langle T\psi_I,\psi_J\rangle \Big|
\\&
=2|\langle TP_{2M}^{\perp }P_{M_N}f,P_{M}^{\perp}P_{M_N}g\rangle |,
 \end{align*}
 with $M_N=N-\log \ell(Q)$. 
We note that $P_{M}^{\perp}P_{M_N}$ is not the zero operator since $N>N_0+M>\log \ell(Q)+M$ implies that 
$M_N>M$. 
Again to simplify notation, we stop writing the conditions $I\in {\cal D}_{2M}^{c}(Q)_{\geq N}$, $J\in {\cal D}_{M}^{c}(Q)_{\geq N}$. We will recover this notation whenever needed.


We fix such $N>N_1$. 
We denote by $\partial \mathcal{D}(Q)$ the union of 
$\partial I$ for all $I\in \mathcal{D}(Q)_{\geq N}$. Then we
decompose the argument functions as 
$P_{2M}^{\perp}P_{M_N}f=f_1+f_{1,\partial}$, $P_M^{\perp}P_{M_N}g=g_1+g_{1,\partial}$ where 
$f_{1,\partial}=(P_{2M}^{\perp}P_{M_N}f)\chi_{\partial D(Q)}$ and $g_{1,\partial}=(P_M^{\perp}P_{M_N}g)\chi_{\partial D(Q)}$. 
With this, 
\begin{align}\label{fdecom}
\nonumber
\langle TP_{2M}^{\perp }P_{M_N}f,P_{M}^{\perp}P_{M_N}g\rangle 
&=\langle TP_{2M}^{\perp }P_{M_N}f,g_1\rangle 
+\langle Tf_1,g_{1,\partial}\rangle 
\\&
+\langle Tf_{1,\partial},g_{1,\partial}\rangle 
\end{align}

We note that 
$$
g_1=\sum_{J\in {\cal D}_{M}^{c}(Q)_{\geq N}}
 \langle g,\psi_{J}\rangle \psi_{\tilde J}
$$
where $\tilde J$ is the interior of $J$ and so, it is an open cube. Similar for $f_1$. Therefore, when we deal with 
$$
\langle TP_{2M}^{\perp }P_{M_N}f,g_1\rangle
=\sum_{I\in {\cal D}_{2M}^{c}(Q)_{\geq N}}\sum_{J\in {\cal D}_{M}^{c}(Q)_{\geq N}}
 \langle f,\psi_{I}\rangle 
 \langle g, \psi_{J}\rangle
 \langle T\psi_I,\psi_{\tilde J}\rangle 
$$
we have that the cubes $\tilde J\in \tilde{\mathcal D}$ are all open. Similarly, when we later deal with 
$$
\langle Tf_1,g_{1,\partial}\rangle 
=\sum_{I\in {\cal D}_{2M}^{c}(Q)_{\geq N}}\sum_{J\in {\cal D}_{M}^{c}(Q)_{\geq N}}
 \langle f,\psi_{I}\rangle 
 \langle g, \psi_{J}\rangle
 \langle T\psi_{\tilde I},\psi_{J}\rangle 
$$
we will have that $\tilde I\in \tilde{\mathcal D}$ are all open cubes. 

However we will start our work without reflecting this distinction in the notation since it is only useful at the end of the argument. 
That is, although 
the work to prove \eqref{M2M} starts with the first term $\langle TP_{2M}^{\perp }P_{M_N}f,g_1\rangle$, since the same argument will also work for the second term 
$\langle Tf_1,g_{1,\partial}\rangle $,
we write each term simply by 
$\langle P_{M}^{\perp}TP_{2M}^{\perp }f,g\rangle$ 
and we will make distinctions only at the end of the proof. We hope this license will not cause any confusion.  

In view of the rates of decay stated in propositions \ref{twobumplemma1} and \ref{twobumplemma2},  
we parametrize the sums according to eccentricity, relative distance and inner relative distance of the cubes as follows.
For fixed $e\in \mathbb Z$, $m\in \mathbb N$ and every given dyadic cube $J$,
we define the family
$$
J_{e,m}=\{ I\in {\mathcal D}_{2M}^c(Q):\ell(I)=2^{e}\ell(J), m\leq \rdist(I_p,J_p)< m+1 \} .
$$
For $m=1$ and $1\leq k\leq 2^{-\min(e,0)-2}-1$, we also define 
$$
J_{e,1, k}=J_{e,1}\cap \{ I\in {\mathcal D}_M^c(Q):k\leq \inrdist(I_p,J_p)< k+1 \} .
$$

The cardinality of 
$J_{e,m}$ is comparable to $2^{-\min(e,0)n}nm^{n-1}$, while the cardinality of 
$J_{e,1,k}$ is comparable to $n(2^{-\min(e,0)}-k)^{n-1}$.
By symmetry, 
we have $I\in J_{e,m}$ if and only if $J\in I_{-e,m}$
and, similarly,  $I\in J_{e,1,k}$ if and only if $J\in I_{-e,1,k}$.

Accordingly with previous parametrization, 
we divide the double sum in \eqref{orthopro} into three parts $D_i$, $N_i$ and $B_6$ (distant cubes, nested cubes and borderline cubes)  and add and subtract the paraproducts $P_i$ into the second part. Namely, we write
\begin{align*}
\langle P_{M}^{\perp}TP_{2M}^{\perp}&f,g\rangle 
=\sum_{e\in \mathbb Z}\sum_{m\geq 2}
\sum_{J}
\sum_{I\in J_{e,m}}
\langle f,\psi_{I}\rangle 
\langle g, \psi_{J}\rangle
\langle T\psi_I,\psi_J\rangle 
\\
&
+\sum_{e\in \mathbb Z}
\sum_{k=2^{\theta |e|}+1}^{2^{|e|}}
\sum_{J}
\sum_{\substack{I\in J_{e,1,k}\\ \dist(I_p,J_p)>0}}
\langle f,\psi_{I}\rangle 
\langle g, \psi_{J}\rangle
\langle T\psi_I,\psi_J\rangle 
\\
&+\sum_{e\geq 0}\sum_{k=2^{\theta |e|}+1}^{2^{|e|-2}}
\sum_{I}
\sum_{\substack{J_p\subset I_p\\ J\in I_{e,1,k}}}
\langle f,\psi_{I}\rangle 
\langle g, \psi_{J}\rangle
\langle T(\psi_I-\psi_{I,J_p}^{\rm full}),\psi_J\rangle 
\\
&+\sum_{e<0}\sum_{k=2^{\theta |e|}+1}^{2^{|e|-2}}
\sum_{J}
\sum_{\substack{I_p\subset J_p\\ I\in J_{-e,1,k}}}
\langle f,\psi_{I}\rangle 
\langle g, \psi_{J}\rangle
\langle T\psi_I,\psi_J-\psi_{J,I_p}^{\rm full}\rangle 
\\
&
+\Pi(P_{2M}^{\perp}f,P_M^{\perp}g)
+\Pi'(P_{2M}^{\perp}f,P_M^{\perp}g)
\\
&+\sum_{e\in \mathbb Z}\sum_{k=1}^{2^{\theta |e|}}
\Big(\sum_{I}
\sum_{J\in I_{e,1,k}}
+\sum_{J}
\sum_{I\in J_{-e,1,k}}
\Big)
\langle f,\psi_{I}\rangle 
\langle g, \psi_{J}\rangle
\langle T\psi_I,\psi_J\rangle
\\
&=D_1+D_2+N_2+N_3+P_4+P_5+B_6.
\end{align*}

The terms $P_4$ and $P_5$ are the paraproduct bilinear forms, which are bounded by Proposition \ref{paraproducts}. 
The terms $D_1$, $D_2$ correspond to the distant cubes and so, they can be dealt by using the estimates of Remark \ref{alterbump} and Proposition \ref{twobumplemma1} respectively.
The terms $N_2$, $N_3$ correspond to the nested cubes, for which we use the estimate of 
Proposition \ref{twobumplemma2}. By symmetry we only need to work with $N_2$.
Finally, the term 
$B_6$ corresponds to borderline cubes.

1) We start with $D_1$. 
Since $m\geq 2$, we have 
by \eqref{bumpwithec} in Remark \ref{alterbump} 
\begin{align*}
|\langle T\psi_I,\psi_J\rangle |
&\lesssim \frac{2^{-|e|\delta}}{m^{\alpha +\delta}}
\frac{\mu(I)^{\frac{1}{2}}
\mu(J)^{\frac{1}{2}}}{\ell(I\smlor J)^\alpha }
F_1(I,J)
\end{align*}
where $F_{1}(I, J)
$ is given in Proposition \ref{twobumplemma1}.
Then 
\begin{align}\label{moduloin2}
|D_1|
\lesssim
\sum_{\substack{e\in \mathbb Z\\ m\geq 2}}
\frac{2^{-|e|\delta}}{m^{\alpha +\delta}}
\sum_{J}\sum_{I\in J_{e,m}}
\frac{\mu(I)^{\frac{1}{2}}
\mu(J)^{\frac{1}{2}}}{\ell(I\smlor J)^\alpha }
|\langle f,\psi_{I}\rangle | |\langle g,\psi_{J}\rangle |
F_1(I,J).
\end{align}



To estimate this last quantity, we divide the study into two cases:
$(I,J)\in \mathcal F_{M}$ 
and
$(I,J)\notin \mathcal F_{M}$. 

1.a) In the first case, to simplify the argument, we assume $\ell(J)\leq \ell(I)$, that is, $e\geq0$. The other case follows by symmetry. Then $I\smlor J=I$ and, by Cauchy's inequality, we can bound the terms in (\ref{moduloin2}) 
corresponding to this case by
\begin{align}\label{aftercauchy}
&\sum_{e\geq 0}
2^{-e\delta }
\Big( \sum_{I}
\sup_{\substack{J\in  \mathcal D_{M}^{c}\\
(I,J)\in \mathcal F_{M}}}F_1(I,J)
|\langle f,\psi_{I}\rangle |^2
\sum_{m\geq 2}\frac{1}{m^{\alpha +\delta }}
\frac{1}{\ell(I)^{\alpha}}
\sum_{J \in I_{-e,m}}
\mu(J)\Big)^{\frac{1}{2}}
\\
\nonumber
&\hskip40pt 
\Big( \sum_{J}
\sup_{\substack{I\in  \mathcal D_{M}^{c}\\
(I,J)\in \mathcal F_{M}}}F_1(I,J)
|\langle g,\psi_{J}\rangle |^2
\sum_{m\geq 2}\frac{1}{m^{\alpha +\delta }}\frac{1}{2^{e\alpha }\ell(J)^{\alpha}}
\sum_{I \in J_{e,m}}
\mu(I)
\Big)^{\frac{1}{2}}.
\end{align}
We note that the cubes $J\in I_{-e,m}$ are pairwise disjoint, and that 
also the cubes $I\in J_{e,m}$ are pairwise disjoint. 
Then we have
$$
\sum_{J\in I_{-e,m}}
\mu(J)\lesssim \mu(mI\setminus (m-1)I), 
$$
$$
\sum_{I\in  J_{e,m}}
\mu(I)\lesssim \mu(m2^{e}J\setminus (m-1)2^{e}J) .
$$
We start with the first factor of 
\eqref{aftercauchy}, whose 
inner sum can be written as
\begin{align*}
\frac{1}{\ell(I)^{\alpha}}
\sum_{m\geq 2}\frac{1}{m^{\alpha +\delta }}
\sum_{J \in I_{-e,m}}
\mu(J)
\lesssim \lim_{R\rightarrow \infty }
\frac{1}{\ell(I)^{\alpha}}
\sum_{m=2}^R
\frac{\mu(mI)-\mu((m-1)I))}{m^{\alpha+\delta }}.
\end{align*}
Now, we write $a_{m}=\mu(mI)$ and use Abel's formula to get 
\begin{align*}
\frac{1}{\ell(I)^{\alpha }}\sum_{m= 2}^{R}\frac{a_{m}-a_{m-1}}{m^{\alpha +\delta }}
&=\frac{a_{R}}{R^{\alpha +\delta }\ell(I)^{\alpha }}-
\frac{a_1}{2^{\alpha +\delta}\ell(I)^{\alpha }}
\\
&
+\frac{1}{\ell(I)^{\alpha }}\sum_{m= 2}^{R-1}a_{m}\Big(\frac{1}{m^{\alpha +\delta }}-\frac{1}{(m+1)^{\alpha +\delta }}\Big).
\end{align*}
For the first term we have 
$$
\frac{a_{R}}{\ell(I)^{\alpha }R^{\alpha +\delta }}= \frac{\mu(RI)}{\ell(RI)^{\alpha }}
\frac{1}{R^\delta}
=\rho(RI)\frac{1}{R^{\delta }}\lesssim \frac{1}{R^{\delta }}\leq \rho_{\rm out}(I)
$$
for $R$ sufficiently large, 
where we remind that  
$
\rho_{\rm out}(I)=\sum_{m\geq 1}\frac{\mu(mI)}{\ell(mI)^{\alpha }}
\frac{1}{m^{\delta+1}}. 
$
The second term is bounded in a similar way:
$$
\frac{a_{1}}{2^{\alpha +\delta}\ell(I)^{\alpha }}\lesssim \frac{\mu(I)}{\ell(I)^{\alpha }}
=\rho(I)\leq \rho_{\rm out}(I).
$$
The last term is bounded by 
\begin{align*}
\frac{1}{\ell(I)^{\alpha }}
&
\sum_{m=2}^{R-1}a_{m}\frac{(m+1)^{\alpha +\delta }-m^{\alpha +\delta }}{(m+1)^{\alpha +\delta }m^{\alpha +\delta }}
\\
&
\lesssim \sum_{m=2}^{R-1}
\frac{\mu(mI)}{m^{\alpha}\ell(I)^{\alpha }}
\frac{(m+1)^{\alpha +\delta -1}}{(m+1)^{\alpha +\delta }m^{\delta }}
\\
&
\lesssim \sum_{m=2}^{R-1}
\frac{\mu(mI)}{\ell(mI)^{\alpha }}
\frac{1}{m^{\delta+1}}\leq \rho_{\rm out}(I).
\end{align*}
We finish the work with the first factor by noting that  
$F_1(I,J)\rho_{\rm out}(I)\leq F_\mu(I,J)<\epsilon $, since 
$(I,J)\in \mathcal F_{M}$. 
For the second factor we can use similar calculations to obtain 
\begin{align*}
\sum_{m\geq 2}\frac{1}{m^{\alpha +\delta }}\frac{1}{2^{e\alpha }\ell(J)^{\alpha}}
\sum_{I \in J_{e,m}}
\mu(I)
\lesssim \rho_{\rm out}(2^eJ).  
\end{align*}
However, now $2^eJ$ does not belong to 
$\mathcal D_M^c(Q)$ in general and so, the only inequality we can use is $F_1(I,J)\rho_{\rm out}(2^eJ)\lesssim 1$.

With both things and Lemma \ref{Planche}, we conclude that the terms in 
$D_1$ corresponding to both cases ($e\geq0$ and $e\leq 0$) can be bounded by a constant times
\begin{align*}
\sum_{e\geq 0}
2^{-e\delta }
&\Big( \sum_{I\in {\mathcal D}_{2M}^{c}}
\sup_{\substack{J\in  \mathcal D_{M}^{c}\\
(I,J)\in \mathcal F_{M}}}F_\mu(I,J)
|\langle f,\psi_{I}\rangle |^2
\Big)^{\frac{1}{2}}
\Big( \sum_{J\in {\mathcal D}_{2M}^{c}}
|\langle g,\psi_{J}\rangle |^2
\Big)^{\frac{1}{2}}
\\
&\lesssim \epsilon^{\frac{1}{2}} \sum_{e\geq 0}
2^{-e\delta } \| f\|_{L^{2}(\mu)}
\| g\|_{L^{2}(\mu)}
\lesssim \epsilon^{\frac{1}{2}}.
\end{align*}

%





1.b) We now study the case when 
$(I,J)\notin \mathcal F_{M}$, that is, when 
$I\in \mathcal D_{2M}^{c}(Q)$, $J\in \mathcal D_{M}^{c}(Q)$ are such that $F_\mu(I,J)\geq \epsilon$. By Lemma \ref{smallF}, we have that 
 $|\log(\ec(I,J))|\gtrsim \log M$, or $\rdist(I, J)\gtrsim M^{\frac{1}{8}}$. 
Then, instead the smallness of $F_\mu$,
in this case we use that 
the size and location of the cubes $I$ and $J$ are such that either their eccentricity or their relative distance are extreme. 

We fix $e_{M}\in \{0, \log M\}$, $m_{M}\in \{M^{\frac{1}{8}},1\}$ such that $e_{M}=0$ implies $m_{M}=M^{\frac{1}{8}}$.
Then, 
by 
the calculations developed in the sub-case 1.a) and 
$F_\mu(I,J)\lesssim 1$, we can bound the relevant part of  (\ref{moduloin2}) by a constant times
\begin{align*}
\sum_{|e|\geq e_{M}}&
\sum_{m\geq m_{M}}
\frac{2^{-|e|\delta }}{m^{\alpha +\delta }}
\sum_{J}\sum_{I\in J_{e,m}}
\frac{\mu(I)^{\frac{1}{2}}
\mu(J)^{\frac{1}{2}}}{\ell(I\smlor J)^\alpha }
|\langle f,\psi_{I}\rangle | |\langle g,\psi_{J}\rangle |
F_1(I,J)
\\
&\lesssim \sum_{|e|\geq e_M}
2^{-|e|\delta }
\Big( \sum_{I}
|\langle f,\psi_{I}\rangle |^2
\sum_{m\geq m_{M}}\frac{1}{m^{\alpha +\delta }}
\frac{\mu(mI\setminus (m-1)I) }{\ell(I)^{\alpha}}\Big)^{\frac{1}{2}}
\\
\nonumber
&\hskip50pt 
\Big( \sum_{J}
|\langle g,\psi_{J}\rangle |^2
\sum_{m\geq m_{M}}\frac{1}{m^{\alpha +\delta }}
\frac{\mu(m2^{e}J\setminus (m-1)2^{e}J) }{\ell(2^{e}J)^{\alpha}}
\Big)^{\frac{1}{2}}.
\end{align*}
Now, by Abel's inequality as in case a) and 
$\rho(I)=\frac{\mu(I)}{\ell(I)^\alpha}\lesssim 1$, 
we have 
\begin{align*}
\sum_{m\geq m_M}
\frac{\mu(mI\setminus (m-1)I) }{m^{\alpha+\delta }\ell(I)^{\alpha}}
&\lesssim 
\lim_{R\rightarrow \infty }\frac{1}{R^{\delta }} +
\frac{1}{m_M^{\delta}}
+\hspace{-.2cm}\sum_{m=m_M+1}^{R}\frac{1}{m^{\delta+1}}
\lesssim m_M^{-\delta},
\end{align*}
and similar for the second factor. 
Then previous expression can be bounded by
\begin{align*}
\sum_{|e|\geq e_{M}}2^{-|e|\delta }\, m_{M}^{-\delta }\| f\|_{L^{2}(\mu)}\| g\|_{L^{2}(\mu)}
&\lesssim 2^{-e_{M}\delta }m_{M}^{-\delta }
\lesssim M^{-\frac{\delta }{8}}
<\epsilon , 
\end{align*}
by the choice of $M$.

2) We now work with $D_2$. 
Since $m=1$ and $k\geq 1+2^{|e|\theta}$ we now have by Proposition
\ref{twobumplemma1} 
\begin{align*}
|\langle T\psi_I,\psi_J\rangle |
&\lesssim \frac{1}{k^{\alpha +\delta}}
\frac{\mu(I)^{\frac{1}{2}}
\mu(J)^{\frac{1}{2}}}{\ell(I\smland J)^\alpha }
F_1(I,J), 
\end{align*}
with $F_1(I,J)$ as before.  
Therefore, 
\begin{align}\label{moduloin2-2}
|D_2|
\lesssim
\sum_{e\in \mathbb Z}\sum_{k=2^{|e|\theta}}^{2^{|e|}}
\sum_{J}\sum_{I\in J_{e,1,k}}\frac{1}{k^{\alpha +\delta}}
\frac{\mu(I)^{\frac{1}{2}}
\mu(J)^{\frac{1}{2}}}{\ell(I\smland J)^\alpha }
|\langle f,\psi_{I}\rangle | |\langle g,\psi_{J}\rangle |
F_1(I,J).
\end{align}



Again to estimate this last quantity, we divide the study into the same two cases as before:
$(I,J)\in \mathcal F_{M}$ 
and
$(I,J)\notin \mathcal F_{M}$. 

2.a) For the first case, 
we assume again $\ell(J)\leq \ell(I)$.
In this case, $I\smland J=J$ and  $\ell(I)=2^e\ell(J)$. Moreover, 
since $e\geq 0$ we have that for each 
$I$ and each $k\in \{ 2^{\theta e}, \ldots, 2^{e}\}$
the cardinality of $I_{-e,1,k}$ is at most  $n(2^{e-1}-2^{e\theta})^{n-1}$. On the other hand, for each $J$ there are only a quantity comparable to $n$ cubes $I$ such that $m=1$ and there 
is only one $k\geq 2^{e\theta}$
such that $J_{e,1,k}$ is not empty. Then we can consider that this $k_J$ is completely determined by $J$. 

With this and 
Cauchy's inequality, we can bound the terms in (\ref{moduloin2-2}) 
corresponding to this case by
\begin{align*}
&\sum_{e\geq 0}
\Big( \sum_{k=2^{e\theta}}^{2^{e}}\sum_{I}
\sum_{J \in I_{-e,1,k}}
F_1(I,J)
|\langle f,\psi_{I}\rangle |^2
\frac{1}{k^{\alpha +\delta }}
\frac{1}{\ell(J)^{\alpha}}
\mu(J)\Big)^{\frac{1}{2}}
\\
&\hskip47pt 
\Big(
\sum_{k=2^{e\theta}}^{2^{e}}
\sum_{I}\sum_{J \in I_{-e,1,k}}
F_1(I,J)
|\langle g,\psi_{J}\rangle |^2
\frac{1}{k^{\alpha +\delta }}\frac{1}{\ell(J)^{\alpha}}
\mu(I)
\Big)^{\frac{1}{2}}
\\
&\leq \sum_{e\geq 0}
\Big( \sum_{I}
\sup_{\substack{J\in  \mathcal D_{M}^{c}\\
(I,J)\in \mathcal F_{M}}}F_1(I,J)
|\langle f,\psi_{I}\rangle |^2
\sum_{k=2^{e\theta}}^{2^{e}}\frac{1}{k^{\alpha +\delta }}
\frac{2^{e\alpha }}{\ell(I)^{\alpha}}
\sum_{J \in I_{-e,1,k}}
\mu(J)\Big)^{\frac{1}{2}}
\\
&\hskip47pt 
\Big( \sum_{J}
\sup_{\substack{I\in  \mathcal D_{M}^{c}\\
(I,J)\in \mathcal F_{M}}}F_1(I,J)
|\langle g,\psi_{J}\rangle |^2
\frac{1}{k_J^{\alpha +\delta }}\frac{1}{\ell(J)^{\alpha}}
\sum_{I \in J_{e,1,k_J}}
\mu(I)
\Big)^{\frac{1}{2}}.
\end{align*}

Now, for each cube $I\in \mathcal{D}(Q)_{N}$ we define
$I^{\rm fr}$ as the family of dyadic cubes $I'\in \mathcal{D}(Q)_{N}$ such that $\ell(I')=\ell(I)$ and $\dist(I',I)=0$. Then 
for each $I'\in I^{\rm fr}$, 
all $k\in \{ 2^{\theta e}, \ldots, 2^{e}\}$ and all $J\in I_{-e,1,k}$, since 
$\ell(J)=2^{-e}\ell(I)$ is fixed, 
we denote by $I_k\in \mathcal C$ the cube 
such that $c(I_k)=c(I')$ and $\ell(I_k)=(2^{e}-k)\ell(J)\leq 2^{e}\ell(J)=\ell(I)$. With this, 
$$
I_{-e,1,k}\subset \{ t\in 3I : k\ell(J)<\dist(t,I)\leq (k+1)\ell(J)\}
\subset I_k\setminus I_{k+1}.
$$
Now, since the cubes $J$ in $I_{-e,1,k}$ are pairwise disjoint 
$$
\sum_{J \in I_{-e,1,k}}\hspace{-.2cm}
\mu(J)\lesssim \sum_{I'\in I^{\rm fr}}\mu(I_k\setminus I_{k+1}).
$$
On the other hand, since 
the cardinality of $J_{e,1,k}$ is comparable to $n$, we have
$$
\sum_{I\in  J_{e,1,k}}
\mu(I)\lesssim \sum_{I'\in I^{\rm fr}}\mu(I')\leq \mu(3I) 
$$
since the cubes in $I^{\rm fr}$ are disjoint.
Then previous expression is bounded by
a constant times  
\begin{align}\label{aftercauchy-2}
&\sum_{e\geq 0}
\Big(\hspace{-.1cm} 
\sum_{I\in {\mathcal D}_{2M}^{c}}
\sup_{\substack{J\in  \mathcal D_{M}^{c}\\
\langle I,J\rangle \in \mathcal I_{M}^{c}}}
\hspace{-.3cm}F_1(I,J)
|\langle f,\psi_{I}\rangle |^2
\sum_{I'\in I^{\rm fr}}
\sum_{k=2^{\theta e}}^{2^{e}}
\frac{1}{k^{\alpha +\delta }}\frac{2^{e\alpha}}{\ell(I)^{\alpha}}
\mu(I_k\setminus I_{k+1})
\Big)^{\frac{1}{2}}
\\
\nonumber
&
\hskip30pt
\Big( \sum_{J\in {\mathcal D}_{M}^{c}}
\sup_{\substack{I\in  \mathcal D_{M}^{c}\\
\langle I,J\rangle \in \mathcal I_{M}^{c}}}\hspace{-.3cm}F_1(I,J)
|\langle g,\psi_{J}\rangle |^2
2^{-e\theta(\alpha +\delta)}
\frac{\mu(3I)}{2^{-e\alpha }\ell(3I)^\alpha }
\Big)^{\frac{1}{2}}.
\end{align}


We start working first factor of 
\eqref{aftercauchy-2}.
As before, we write $a_{k}=\mu(I_k)$ and 
evaluate the inner sum of the first factor by using 
Abel's formula:
\begin{align*}
\sum_{k=2^{\theta e}}^{2^{e}} \frac{a_{k}-a_{k+1}}{k^{\alpha +\delta }}
&=\frac{a_{2^{\theta e}}}{2^{(\alpha +\delta )\theta e}}
-\frac{a_{2^{e}+1}}{2^{(\alpha +\delta)e}}
+
\sum_{k=2^{\theta e}+1}^{2^{e}}
a_{k}\Big(\frac{1}{k^{\alpha +\delta }}-\frac{1}{(k-1)^{\alpha +\delta }}\Big).
\end{align*}
For the first term we have
\begin{align*}
\frac{a_{2^{\theta e}}}{2^{(\alpha +\delta)\theta e}}
&\leq \frac{\mu(I_{2^{\theta e}})}{2^{(\alpha +\delta )\theta e}}
\leq \rho(I_{2^{\theta e}})\ell(I_{2^{\theta e}})^\alpha 2^{-(\alpha +\delta )\theta e}
\\
&
\leq \rho_{\rm in }(3I)\ell(I)^\alpha 2^{-(\alpha +\delta )\theta e}.
\end{align*}
Similarly, 
the absolute value of the second term can be bounded by
$$
\frac{a_{2^{e-2}}}{2^{(\alpha +\delta)e}}
 \leq \frac{\mu(I)}{2^{(\alpha +\delta) e}}
 =\rho(I)\ell(I)2^{-(\alpha +\delta) e}
 \leq \rho_{\rm }(3I)\ell(I)2^{-(\alpha +\delta) e}.
 $$
The absolute value of the last term is bounded by 
\begin{align*}
\sum_{k=2^{\theta e}+1}^{2^{e}}
a_{k}\frac{k^{\alpha +\delta }-(k-1)^{\alpha +\delta }}{(k-1)^{\alpha +\delta }k^{\alpha +\delta }}
&
\lesssim \sum_{k=2^{\theta e}+1}^{2^{e}}
\mu(I_{k})
\frac{k^{\alpha +\delta -1}}{(k-1)^{\alpha +\delta }k^{\alpha +\delta }}
\\
&
\lesssim \sum_{k=2^{\theta e}+1}^{2^{e}}
\rho(I_k)\ell(I_{k})^\alpha 
\frac{1}{k^{\alpha + \delta+1}}
\\
&
\lesssim \rho_{\rm in}(3I)
\ell(I)^\alpha
\sum_{k=2^{\theta e}}^{2^{e}} 
\frac{1}{k^{\alpha + \delta+1}}
\\
&\lesssim \rho_{\rm in}(3I)\ell(I)^\alpha
2^{-(\alpha +\delta )\theta e}
\end{align*}

From the three inequalities, $\ell(I)=2^e\ell(J)$ and the fact that the cardinality of $I^{\rm fr}$ is 
$3^n$, we get 
\begin{align*}
\sum_{I\in I^{\rm fr}}
\frac{2^{e\alpha}}{\ell(I)^\alpha }\sum_{k=2^{\theta e}}^{2^{e}} \frac{a_{k}-a_{k+1}}{k^{\delta }}
&\lesssim 
\rho_{\rm in }(3I)2^{-e(\theta (\alpha +\delta )-\alpha )}.
\end{align*}

On the other hand, for the second factor in \eqref{aftercauchy-2}, we have
$$
2^{-e\theta(\alpha +\delta)}
\frac{\mu(3I)}{2^{-e\alpha }\ell(3I)^\alpha }
\leq 2^{-e(\theta (\alpha +\delta )-\alpha )}\rho_{\rm in}(3I).
$$

With all this, 
the inequality 
$F_1(I,J)\rho_{\rm in}(3I)\lesssim F_{\mu}(I,J)$ 
and Lemma \ref{Planche},
the terms in $N_2$
corresponding to this case can be bounded by
\begin{align*}
&\sum_{e\geq 0}2^{-e(\theta (\alpha +\delta )-\alpha )}
\Big( \sum_{I\in {\mathcal D}_{2M}^{c}}
\sup_{\substack{J\in  \mathcal D_{M}^{c}\\
\langle I,J\rangle \in \mathcal I_{M}^{c}}}\hspace{-.3cm}F_\mu(I,J)
|\langle f,\psi_{I}\rangle |^2\Big)^{\frac{1}{2}}
\\
&
\hskip85pt
\Big(  \sum_{J\in {\mathcal D}_{M}^{c}}
\sup_{\substack{J\in \mathcal D_{M}^{c}\\
\langle I,J\rangle \in \mathcal I_{M}^{c}}}\hspace{-.3cm}F_\mu(I,J)
|\langle g,\psi_{J}\rangle |^2\Big)^{\frac{1}{2}}
\\
&\lesssim \epsilon \sum_{e\geq 0}
2^{-e(\theta (\alpha +\delta )-\alpha )}
\| f\|_{L^{2}(\mu)}\| g\|_{L^{2}(\mu)}
\lesssim \epsilon ,
\end{align*}
by the choice of $0<\frac{\alpha }{\alpha +\delta}<\theta =\frac{\alpha }{\alpha +\frac{\delta}{2}}\leq 1$.

2.b) We now study the case when 
$(I,J)\notin \mathcal F_{M}$
and so, as before instead the smallness of $F_\mu$,
we use that either 
the eccentricity or the relative distance between $I$ and $J$ are extreme. 

As in case 1.b), we fix $e_{M}\in \{0, \log M\}$, $m_{M}\in \{M^{\frac{1}{8}},1\}$ such that $e_{M}=0$ implies $m_{M}=M^{\frac{1}{8}}$. 
But, since $m=1$, we have that $m_{M}\leq m=1<M^{\frac{1}{8}}$, which implies $m_{M}=1$ and so, $e_{M}=\log M$.

Then, 
by Lemma \ref{Planche}, the calculations developed in the sub-case 2.a) and 
$F_\mu(I,J)\lesssim 1$, we can bound the relevant part of  (\ref{moduloin2}) by a constant times
\begin{align*}
&\sum_{|e|\geq e_{M}}\sum_{k=2^{|e|\theta}}^{2^{|e|}}
\sum_{J}\sum_{I\in J_{e,1,k}}\frac{1}{k^{\alpha +\delta}}
\frac{\mu(I)^{\frac{1}{2}}
\mu(J)^{\frac{1}{2}}}{\ell(I\smland J)^\alpha }
|\langle f,\psi_{I}\rangle | |\langle g,\psi_{J}\rangle |
F_1(I,J)
\\
&\lesssim \sum_{|e|\geq e_{M}}
\Big( \sum_{I\in {\mathcal D}_{2M}^{c}}
\sup_{\substack{J\in  \mathcal D_{M}^{c}\\
\langle I,J\rangle \in \mathcal I_{M}^{c}}}
\hspace{-.3cm}F_1(I,J)
|\langle f,\psi_{I}\rangle |^2
\sum_{k=2^{\theta e}}^{2^{e}}
\frac{1}{k^{\alpha +\delta }}\frac{2^{e\alpha}}{\ell(I)^{\alpha}}
\mu(I_k\setminus I_{k+1})
\Big)^{\frac{1}{2}}
\\
&
\hskip50pt
\Big( \sum_{J\in {\mathcal D}_{M}^{c}}
\sup_{\substack{I\in  \mathcal D_{M}^{c}
\\ \langle I,J\rangle \in \mathcal I_{M}^{c}}}\hspace{-.3cm}F_1(I,J)
|\langle g,\psi_{J}\rangle |^2
2^{-e\theta(\alpha +\delta)}
\frac{\mu(3I)}{2^{-e\alpha }\ell(3I)^\alpha }
\Big)^{\frac{1}{2}}
\\
&
\lesssim \sum_{|e|\geq e_M}2^{-|e|(\theta (\alpha +\delta )-\alpha )}
\Big( \sum_{I\in {\mathcal D}_{2M}^{c}}
\sup_{\substack{J\in  \mathcal D_{M}^{c}\\
\langle I,J\rangle \in \mathcal I_{M}^{c}}}\hspace{-.3cm}F_\mu(I,J)
|\langle f,\psi_{I}\rangle |^2\Big)^{\frac{1}{2}}
\\
&
\hskip115pt
\Big(  \sum_{J\in {\mathcal D}_{M}^{c}}
\sup_{\substack{I\in \mathcal D_{M}^{c}\\
\langle I,J\rangle \in \mathcal I_{M}^{c}}}\hspace{-.3cm}F_\mu(I,J)
|\langle g,\psi_{J}\rangle |^2\Big)^{\frac{1}{2}}
\\
&\lesssim \sum_{|e|\geq \log M}
2^{-|e|\alpha (\frac{\alpha +\delta}{\alpha+\delta/2}-1)}
\| f\|_{L^{2}(\mu)}\| g\|_{L^{2}(\mu)}
\lesssim 
M^{-\alpha (\frac{\alpha +\delta}{\alpha+\delta/2}-1)}< \epsilon, 
\end{align*}
by the choices of $\theta$ and $M$.

3) Now we work with $N_2$, for which we have 
$2^{\theta |e|}\leq k\leq 2^{|e|-2}$ with $\theta =\frac{\alpha}{\alpha +\delta/2}$. 
By Proposition \ref{twobumplemma2}, when $m=1$ and $k\geq 2^{\theta |e|}$, we have
\begin{align*}
|\langle T(\psi_I-\psi_{I,J}^{\rm full}),\psi_J\rangle |
&
\lesssim k^{-\delta}
\sum_{R\in \{I,I_p\}}
\Big(\frac{\mu(R\cap J)}{\mu(R)}\Big)^{\frac{1}{2}}
F_{2,\mu }(I,J)
\\
&+k^{-(\alpha+\delta)}
\frac{\mu(I)^{\frac{1}{2}}\mu(J)^{\frac{1}{2}}}{\ell(
I\smland J)^{\alpha }}F_{3}(I,J),
\end{align*}
with $F_{2,\mu}$ and $F_3$ are given in Proposition \ref{twobumplemma2}. 
The second term can be bounded using the same approach we used in the case 2) since the only difference between these two cases is the last factor, which is given by  
$F_3$ instead of $F_1$.
Then we focus on the first term. 


In $N_2$ we have $e\geq 0$, which implies  
$\ell(J)\leq \ell(I)$. Moreover, 
$F_{2,\mu }\leq F_{\mu}$ and so, the terms corresponding to this case can be bounded by a constant times 
\begin{align*}
\sum_{e\geq 0}&\sum_{k=2^{\theta e}}^{2^{e-2}}
\sum_{\substack{J\\ J_p\subset I_p}}
\sum_{I\in J_{e,1,k}}
|\langle f,\psi_{I}\rangle |
|\langle g, \psi_{J}\rangle |
\sum_{R\in \{I,I_p\}}\Big(\frac{\mu(R\cap J)}{\mu(R)}\Big)^{\frac{1}{2}}k^{-\delta}
F_{\mu}(I,J) .
\end{align*}
As before, we distinguish two cases:
$(I,J)\in \mathcal F_{M}$ 
and
$(I,J)\notin \mathcal F_{M}$.

3.a) In the first case, we have that $F_{\mu}(I,J)<\epsilon $. Moreover, since $e\geq 0$ the cardinality of $J_{e,1,k}$ is comparable to $n$ and 
there 
is only one $k\geq 2^{e\theta}$
such that $J_{e,1,k}$ is not empty, that is, $k_J$ is completely determined by $J$. 
Then, 
by Cauchy's inequality again, we can bound the terms  
of $N_2$
corresponding to this case by a constant times
\begin{align*}
&\epsilon \sum_{e\geq 0}
\Big( \sum_{I\in {\mathcal D}_{2M}^{c}}
|\langle f,\psi_{I}\rangle |^2
\sum_{k=2^{\theta e}}^{2^{e-2}}k^{-\delta}
\sum_{R\in \{I,I_p\}}
\sum_{J\in I_{-e,1,k}}
\frac{\mu(R\cap J)}{\mu(R)}
\Big)^{\frac{1}{2}}
\\
&
\hskip30pt \Big( \sum_{J}
|\langle g,\psi_{J}\rangle |^2
k_J^{-\delta}
\Big)^{\frac{1}{2}},
\end{align*}

Now, for fixed $J\in I_{-e,1,k}$, fixed $I'\in \child(I_p)$ such that $J\subset I'$, all $k\in \{ 2^{\theta e}, \ldots, 2^{e-2}\}$ and all $J\in I_{-e,1,k}$, since 
$\ell(J)=2^{-e}\ell(I)$ is fixed, 
we can define $I_k\in \mathcal C$ to be the cube 
such that $c(I_k)=c(I')$ and $\ell(I_k)=(2^{e}-k)\ell(J)\leq \ell(I)$. With this, we have
$$
I_{-e,1,k}\subset \{ t\in I : k\ell(J)<\dist(t,\mathfrak{D}_{I_p})\leq (k+1)\ell(J)\}
\subset I_k\setminus I_{k+1}
$$
Moreover, since the cubes $J$ in $I_{-e,1,k}$ are pairwise disjoint, we have for $R\in \{I,I_p\}$
$$
\sum_{J \in I_{-e,1,k}}\hspace{-.2cm}
\mu(R\cap J)\lesssim \sum_{I'\in \ch(I_p)}\mu(R\cap (I_k\setminus I_{k+1})).
$$
Then, using that the cardinality of 
$\ch(I_p)$ is $2^n$,  previous expression can be bounded by a constant times 
\begin{align*}
&\epsilon \sum_{e\geq 0}
\Big( \sum_{I\in {\mathcal D}_{2M}^{c}}
|\langle f,\psi_{I}\rangle |^2
\sum_{I'\in \ch(I_p)}
\sum_{k=2^{\theta e}}^{2^{e-2}}k^{-\delta}
\frac{\mu(R\cap I_k)-\mu( R\cap I_{k+1})}{\mu(R)}
\Big)^{\frac{1}{2}}
\\
&
\hskip30pt
\Big( \sum_{J\in {\mathcal D}_{M}^{c}}
|\langle g,\psi_{J}\rangle |^2k_J^{-\delta }\Big)^{\frac{1}{2}}.
\end{align*}


As before, we write $a_{k}=\mu(R\cap I_k)$ and 
evaluate the inner sum of the first factor by using 
Abel's formula:
\begin{align*}
\frac{1}{\mu(R)}\sum_{k=2^{\theta e}}^{2^{e-2}} \frac{a_{k}-a_{k+1}}{k^{\delta }}
&=\frac{a_{2^{\theta e}}}{2^{\theta \delta e} \mu(R)}
-\frac{a_{2^{e-2}+1}}{2^{(e-2)\delta } \mu(R)}
\\
&
+\frac{1}{\mu(R)}
\sum_{k=2^{\theta e}+1}^{2^{e-2}}
a_{k}\Big(\frac{1}{k^{\delta }}-\frac{1}{(k-1)^{\delta }}\Big).
\end{align*}
For the first term we have
$$
\frac{a_{2^{\theta e}}}{2^{\theta \delta e} \mu(R)}
\leq \frac{\mu(R\cap I')}{2^{\delta \theta e}\mu(R)}
\leq 2^{-\delta \theta e}.
$$
Similarly, 
the absolute value of the second term can be bounded by
$$
\frac{a_{2^{e-2}+1}}{2^{(e-2)\delta }\mu(R)}
 \lesssim \frac{\mu(R\cap I')}{2^{\delta e}\mu(R)}
 \leq 2^{-\delta e}.
 $$

The absolute value of the last term is bounded by 
\begin{align*}
\frac{1}{\mu(R)}
&
\sum_{k=2^{\theta e}+1}^{2^{e-2}}
a_{k}\frac{k^{\delta }-(k-1)^{\delta }}{(k-1)^{\delta }k^{\delta }}
\\
&
\lesssim \sum_{k=2^{\theta e}+1}^{2^{e-2}}
\frac{\mu(R\cap I_{k})}{\mu(R)}
\frac{k^{\delta -1}}{(k-1)^{\delta }k^{\delta }}
\\
&
\lesssim \sum_{k\geq 2^{\theta e}+1}\frac{1}{(k-1)^{ \delta+1}}\lesssim 
2^{-\delta \theta e}
\end{align*}

For the second factor, we just use that 
$k_J\geq 2^{\theta e}$. 
With this, the fact that the cardinality of 
$\ch(I_p)$ is $2^n$ and Lemma \ref{Planche}, we bound the terms in $N_2$
corresponding to this case by
\begin{align*}
\epsilon \sum_{e\geq 0}
\Big( 2^{-\theta \delta e}
&\sum_{I\in {\mathcal D}_{2M}^{c}}
|\langle f,\psi_{I}\rangle |^2\Big)^{\frac{1}{2}}
\Big( 2^{-\theta \delta e}
 \sum_{J\in {\mathcal D}_{M}^{c}}
|\langle g,\psi_{J}\rangle |^2\Big)^{\frac{1}{2}}
\\
&\lesssim \epsilon \sum_{e\geq 0}
2^{-\theta \delta e}
\| f\|_{L^{2}(\mu)}\| g\|_{L^{2}(\mu)}
\lesssim \epsilon ,
\end{align*}
since $0<\theta $.

3.b) 
When $(I,J)\notin \mathcal F_{M}$, 
as in case 2.b), we 
fix 
$e_{M}=\log M$.
Then, 
by the calculations developed in the sub-case 3.a) and 
$F_\mu(I,J)\lesssim 1$, we bound the relevant part of  (\ref{moduloin2}) by a constant times
\begin{align*}
&\sum_{|e|\geq e_M}\sum_{k=2^{\theta |e|}}^{2^{|e|-2}}
\sum_{\substack{J\\ J_p\subset I_p}}
\sum_{I\in J_{e,1,k}}
|\langle f,\psi_{I}\rangle |
|\langle g, \psi_{J}\rangle |
\hspace{-.1cm}
\sum_{R\in \{I,I_p\}}
\hspace{-.1cm}
\Big(\frac{\mu(R\cap J)}{\mu(R)}\Big)^{\frac{1}{2}}
\frac{F_\mu(I,J) }{k^{\delta}}
\\
&\lesssim \sum_{|e|\geq e_M}
\Big(
|\langle f,\psi_{I}\rangle |^2
\sum_{I'\in \child{(I)}}
\sum_{k=2^{\theta |e|}}^{2^{|e|-2}}k^{-\delta}
\hspace{-.1cm}
\sum_{R\in \{I,I_p\}}
\sum_{J\in I_{-e,1,k}}
\frac{\mu(R\cap J)}{\mu(R)}
\Big)^{\frac{1}{2}}
\\
&
\hskip45pt \Big( \sum_{J}
|\langle g,\psi_{J}\rangle |^2 k_J^{-\delta}
\Big)^{\frac{1}{2}}
\\
&\lesssim \sum_{|e|\geq \log M}2^{-|e|\theta \delta}
\| f\|_{L^{2}(\mu)}\| g\|_{L^{2}(\mu)}
\lesssim M^{-\frac{\alpha \delta }{\alpha+\delta/2}}
\leq \epsilon ,
\end{align*}
by the choice of $\theta $ and $M$.


4) We work now with the term $B_6$, which contains all cubes $I,J\in \mathcal D_M^c(Q)$ such that $1\leq k\leq  2^{\theta |e|}+1$, that is, $\inrdist(I,J)-1\leq \ec(I,J)^{-\theta}$. 

We remind the following notation used in Lemma \ref{zeroborder}:
for $N\in \mathbb N$, 
let 
$\mathcal{D}(Q)_{\geq N}=\{ I\in \mathcal{D}_M^c(Q)/ \ell(I)\geq 2^{-N}\ell(Q)\}$  
and $\mathcal{D}(Q)_{N}=\{ I\in \mathcal{D}_M^c(Q)/ \ell(I)=2^{-N}\ell(Q)\}$.
Moreover, for $I\in \mathcal{D}(Q)_{\geq N}$, let 
$I_{\theta}$ be the family of cubes $J\in \mathcal{D}(Q)_{\geq N}$
for which 
$1\leq k\leq  2^{\theta |e|}$.



Then we can rewrite $B_6$ as
$$
B_6=\sum_{I\in \mathcal{D}(Q)_{\geq N}}\sum_{J\in I_{\theta}}
\langle f,\psi_I\rangle \langle g,\psi_J\rangle
\langle T\psi_I,\psi_J \rangle .
$$
For each $I\in \mathcal{D}(Q)_{\geq N}$, let 
$I_{\max}$ be the family of cubes $J\in I_{\theta}$ that are maximal with respect to the inclusion. Then let 
$I_{\rm over}$ be the family of cubes 
$R\in \mathcal{D}(Q)_{\geq N}$ such that 
$J\subsetneq R$ for some $J\in I_{\max}$. 
We note that 
for all $I\in \mathcal{D}(Q)_{\geq N}$, 
either $Q\in I_{\max}$ (if $I\in Q_{\theta}$) or 
$Q\in I_{\rm over}$. So, we always have 
$Q\in I_{\theta}\cup I_{\rm over}$. 
We also note that 
all cubes in $I_{\rm over}$ satisfy that $k>2^{\theta |e|}$ with respect to $I$. 
Now to previous expression we  
add and subtract the term
$$
A=\sum_{I\in \mathcal{D}(Q)_{\geq N}}\sum_{J\in I_{\rm over}}
\langle f,\psi_I\rangle \langle g,\psi_J\rangle
\langle T\psi_I,\psi_J \rangle .
$$
With this we obtain
\begin{align}\label{b4teles}
|B_6|\lesssim |\sum_{I\in \mathcal{D}(Q)_{\geq N}}
\langle f,\psi_I\rangle 
\langle T\psi_I,\sum_{J\in I_{\theta}\cup I_{\rm over}}\langle g,\psi_J\rangle\psi_J \rangle |
+|A|.
\end{align}
Since all pair of cubes added satisfy that $k>2^{\theta |e|}$, we can apply the reasoning of any of the previous cases
(adding and subtracting the corresponding part of a paraproduct when needed)
to prove that the second term
in \eqref{b4teles} satisfies $|A|\lesssim \epsilon \|f\|_{L^2(\mu)}\|g\|_{L^2(\mu)}$. Then we only need to study the first term. 

By summing a telescoping series 
we have for each  $I\in \mathcal{D}(Q)_{\geq N}$,
$$
\sum_{J\in I_{\theta}\cup I_{\rm over}}\langle g,\psi_J\rangle\psi_J 
=\sum_{J\in \mathcal{D}(I)_{N}\cap I_{\theta}}
\langle g\rangle_J\chi_J-\langle g\rangle_{Q'}\chi_{Q'}
=\sum_{J\in \mathcal{D}(I)_{N}\cap I_{\theta}}
\langle g\rangle_J\chi_J, 
$$
where $Q'\in \mathcal D$ such that 
$\supp g\subset Q'$ and so, 
$\langle g\rangle_{Q'}=0$. With this, and Fubini's theorem, the first term in the right hand side of \eqref{b4teles} can be rewritten as 
\begin{align}\label{telescopy4f}
\nonumber
\sum_{I\in \mathcal{D}(Q)_{\geq N}}&\sum_{J\in \mathcal{D}(I)_{N}\cap I_{\theta}}
\langle f,\psi_I\rangle 
\langle g\rangle_J
\langle T\psi_I,\chi_J \rangle 
\\
&
=\sum_{J\in \mathcal{D}(Q)_{N}}
\langle g\rangle_J
\langle T(\sum_{I\in \mathcal{D}(Q)_{\geq N}\cap J_{\theta}}
\langle f,\psi_I\rangle \psi_I),\chi_J \rangle ,
\end{align}
where $J_\theta $ is defined as $I_\theta$ was defined before. 

We now remind the following definition: for $J\in \mathcal{D}(Q)_{N}$, 
$J^{\rm fr}$ denotes the family of dyadic cubes $J'\in \mathcal{D}(Q)_{N}$ such that $\ell(J')=\ell(J)$ and $\dist(J',J)=0$. Then 
we note that, since $J$ has minimal side length, the condition $I\in J_\theta $ implies that $J'\subset I$ for some $J'\in J^{\rm fr}$. 

Moreover, the cardinality of $J^{\rm fr}$ is $3^n$ and so, we can enumerate the cubes in $J^{\rm fr}$ as 
$\{ J_j\}_{j=1}^{3^n}$ by their fixed position with respect to $J$. 
Then, for each $j\in \{1, \ldots, 3^n\}$ 
the cubes $I\in \mathcal{D}(Q)_{\geq N}\cap J_{\theta}$ such that $J_j\subset I$ 
form an increasing chain of cubes
$J_j=J_{j,N}\subset J_{j,N-1}\subset \ldots \subset J_{j, k_{j}}$ parametrized by their size length   
$\ell(J_{j,k})=2^{-k}\ell(Q)$ with $k\in 
\{ k_{j},\ldots, N\}\subset 
\{0,\ldots, N\}$. Some chains may be empty.
Then for each fixed $J\in \mathcal{D}(Q)_{N}$, we have
\begin{align*}
\sum_{I\in \mathcal{D}(Q)_{\geq N}\cap J_{\theta}}
\langle f,\psi_I\rangle \psi_I
=\sum_{j=1}^{3^n}
\langle f\rangle_{J_{j}} \chi_{J_j}
-\sum_{j=1}^{3^n}\langle f\rangle_{J_{j,k_j}} \chi_{J_{j,k_j}}.
\end{align*}

With this \eqref{telescopy4f} can be written as 
\begin{align*}
\sum_{J\in \mathcal{D}(Q)_{N}}&
\sum_{j=1}^{3^n}
\langle f\rangle_{J_{j}} 
\langle g\rangle_J
\langle T\chi_{J_j},\chi_J \rangle
\\
&-\sum_{J\in \mathcal{D}(Q)_{N}}
\sum_{j=1}^{3^n}\langle f\rangle_{J_{j,k_j}} 
\langle g\rangle_J
\langle T\chi_{J_{j,k_j}},\chi_J \rangle
=S_1-S_2.
\end{align*}

We now use that the kernel operator is 
\begin{align*}
K_{\gamma,Q}(t,x)&=K(t,x)
(1-\phi(\frac{|t-x|}{\gamma})),
\end{align*}
where we have used that $t, x\in Q$.

By the definition of the kernel we have
$K_{\gamma, Q}(t,x)=0$ for $|t-x|<\gamma $. Then, 
if $\ell(I\smlor J)<\gamma/3$ and 
$\dist(I,J)< \gamma/3 $, 
we have 
$|t-x|<\gamma $ for all $t\in I$ and $x\in J$ and so,
$\langle T\chi_I,\chi_J \rangle =0$. 

Now, all cubes in
the first sum $S_1$ satisfy $\ell(J')=\ell(J)$ and 
$\dist(J',J)=0$. Moreover, since  
$N>N_0\geq \log\frac{6\ell(Q)}{\gamma }$ and 
$\ell(J')=2^{-N}\ell(Q)$, we have 
$\ell(J')<\gamma/3$, and so, each term in the sum $S_{1}$ equals zero.

We now focus on $S_{2}$. 
The cubes in that term satisfy 
$\dist(J_j,J)=0$. Moreover, since $1\leq k\leq 2^{\theta |e|}+1$, we have
\begin{align*}
\dist(J_{j,k_j},J_j)&< 2^{|e|\theta}\ell(J_j)-1
\leq \Big(\frac{\ell(J_j)}{\ell(J_{j,k_j})}\Big)^{1-\theta}
\ell(J_{j,k_j})\leq \ell(J_{j,k_j}).
\end{align*}
Then, since $\ell(J_j)\leq \ell(J_{j,k_j})$, we get 
$$
\dist(J_{j,k_j},J)\leq \dist(J_j,J)+\ell(J_j)+\dist(J_{j,k_j},J_j)
\leq 2\ell(J_{j,k_j}).
$$
With this, when 
$\ell(J_{j,k_j})<\gamma /6$, we have 
$\dist(J_{j,k_j},J)\leq \gamma /3$ and so 
$\langle T\chi_{I_{j,k_j}},\chi_{J} \rangle =0$. Therefore
the scales for which the dual pair is non-zero 
satisfy 
$\ell(J_{j,k_j})=2^{-k}\ell(Q)\geq \gamma /6$, that is,
$k\leq \log\frac{6\ell(Q)}{\gamma }\leq N_0$. And since 
$k\in \{0,\ldots, N\}$, that means that the non-zero terms in $S_2$ 
contain cubes $J_{j,k_j}$ of 
at most $N_0+1$ different size lengths (in fact in the $N_0+1$ largest scales, all of them in $\{0, 1, \ldots , N_0\}$).



Then we are going to rewrite the sum in 
\begin{align}\label{reparameter}
S_{2}=
\sum_{J\in \mathcal{D}(Q)_{N}}
\sum_{j=1}^{3^n}\langle f\rangle_{J_{j,k_j}} 
\langle g\rangle_J
\langle T\chi_{J_{j,k_j}},\chi_J \rangle
\end{align}
in terms of the cubes $J_{j,k_j}$ instead of the cubes $J$. 

We remind that in \eqref{reparameter}, 
for each $J\in \mathcal{D}(Q)_{N}$,  
each $J_j\in J^{\rm fr}$ with $j\in \{1,\ldots , 3^n\}$ and each scale
$k\in \{ 0,\ldots , N_0\}$,
there is an associated cube 
$J_{j,k}\in J_{\theta}$ with side length 
$\ell(J_{j,k})=2^{-k}\ell(Q)$. 
Now we 
re-parametrize the cubes we have up to now denoted by $J_{j,k}$ in the following way: 
for each scale $k\in \{ 0,\ldots , N_0\}$ and 
each $i\in \{ 1,\ldots , 2^{kn}\}$, 
we denote by $I_{i,k}$ the cubes 
such that $\ell(I_{i,k})=2^{-k}\ell(Q)$. 
We note that inside $Q$, for each 
$k\in \{ 0,\ldots , N_0\}$ there are in total $2^{kn}$ of such cubes.
Now, for each $I_{i,k}$ 
we define $\mathcal J_{i,k}$ as the family of cubes 
$J\in \mathcal D(Q)_N$ such that there exists  
$J'\in J^{\rm fr}$ 
associated with $I_{i,k}$. This means that 
$J'\subset I_{i,k}$, 
$\ell(I_{i,k})=2^{-k}\ell(Q)$ and 
$J'\in (I_{i,k})_{\theta}$, that is, 
$\dist(I_{i,k}, J')<
(2^{|e|\theta}+1)\ell(J')$. 
Finally, we denote 
$C_{i,k}=\bigcup_{J\in \mathcal J_{i,k}}3J$. We note that 
$\mu(3J)\leq \mu(C_{i,k})$. With this
\begin{align*}
S_{2}=\sum_{k=0}^{N_0-1}\sum_{i=0}^{2^{kn}} \sum_{J\in \mathcal J^{i,k}}
\langle f\rangle_{I_{i,k}} 
\langle g\rangle_J 
\langle T_1\chi_{I_{i,k}},\chi_J \rangle
\end{align*}

Now, for fixed $I_{i,k}$, let $\mathcal{J}^{0,0}$, $\mathcal{J}^{0,1}$, $\mathcal{J}^{1,0}$ and $\mathcal{J}^{1,1}$ be  
the collection of cubes in $J\in \mathcal J_{i,k}$
such that $\Re(\langle T_1\chi_{I_{i,k}},\chi_J \rangle)\geq 0$, 
$\Re(\langle T_1\chi_{I_{i,k}},\chi_J \rangle)<0$, 
$\Im(\langle T_1\chi_{I_{i,k}},\chi_J \rangle)\geq 0$, and $\Im(\langle T_1\chi_{I_{i,k}},\chi_J \rangle)>0$ respectively. Let also $S^{l_1,l_2}$ the union of the cubes in $\mathcal{J}^{l_1,l_2}$. Finally, we define 
$$
\tilde S
=\bigcup_{k=0}^{N_0-1}\bigcup_{i=0}^{2^{kn}}\bigcup_{J\in \mathcal{J}^{i,k}}J,
$$ 
that is, the union of 
all cubes $J\in \mathcal{D}(Q)_{N}$ such that 
$J\in (I_{i,k})_\theta$ for some $i,k$. 
We note that 
$$
\tilde S
=\bigcup_{k=0}^{N_0-1}\bigcup_{i=0}^{2^{kn}}\bigcup_{l_1,l_2\in \{0,1\}}\bigcup_{J\in \mathcal{J}^{l_1,l_2}}J.
$$

Before continuing, we remind that in the decomposition obtained in \eqref{fdecom} 
we are first working on estimates for 
$\langle P_{M}^{\perp}TP_{2M}^{\perp }f,g_1\rangle $. In this case, the cubes $J\in \tilde{\mathcal D}$ and so, they are open cubes. Therefore, 
$C_{i,k}$ and $\tilde S$ are open sets and they satisfy 
by the choice of $N$ in \eqref{Small} that 
$\mu(\tilde S)$ is sufficiently small.

With all this,
\begin{align*}
|S_{2}|&\leq \sum_{k=0}^{N_0-1}\sum_{i=0}^{2^{kn}} \sum_{J\in \mathcal J^{i,k}}
|\langle f\rangle_{I_{i,k}} |
|\langle g\rangle_J |
|\langle T_1\chi_{I_{i,k}},\chi_J \rangle|
\\
&\lesssim \| f\|_{L^{\infty }(\mu)}\| g\|_{L^{\infty }(\mu)}
\sum_{k=0}^{N_0-1}\sum_{i=0}^{2^{kn}}
\sum_{l=0}^1\Big( \sum_{J\in \mathcal{J}^{0,l}}(-1)^{l}
\Re\langle T_1\chi_{I_{i,k}},\chi_J \rangle
\\
&
\hskip100pt +\sum_{J\in \mathcal{J}^{1,l}}(-1)^{l}\Im\langle T_1\chi_{I_{i,k}},\chi_J \rangle\Big)
\\
&=\| f\|_{L^{\infty }(\mu)}\| g\|_{L^{\infty }(\mu)}
\sum_{k=0}^{N_0-1}\sum_{i=0}^{2^{kn}}
\sum_{l=0}^1 
\big((-1)^{l}\Re\langle T_1\chi_{I_{i,k}},\chi_{S^{0,l}} \rangle
\\
&
\hskip100pt +(-1)^{l}\Im\langle T_1\chi_{J_{i,k}},\chi_{S^{1,l}} \rangle \big)
\\
&\lesssim \| f\|_{L^{\infty }(\mu)}\| g\|_{L^{\infty }(\mu)}
\sum_{k=0}^{N_0-1}\sum_{i=0}^{2^{kn}}
\sum_{l_1,l_2\in \{0,1\}}
|\langle T_1\chi_{I_{i,k}},\chi_{S^{l_1,l_2}}\rangle |.
\end{align*}

Now we divide $S^{l_1,l_2}$ into $2n+1$ parts: 
$S^{l_1,l_2}=\cup_{j=0}^{2n}S_j^{l_1,l_2}$, where 
$S_j^{l_1,l_2}$ is the union of cubes $J\in \mathcal{D}(Q)_{N}$ such that 
$J\in (I_{i,k})_\theta$ for some $i,k$, and there is $I_{i,k}^j\in I_{i,k}^{\rm fr}$ with $J\subset I_{i,k}^j$.
This implies that $S_j^{l_1,l_2}\subset J_{i,k}^j$. 


We work with $S_j^{l_1,l_2}$ for every $j\in \{ 0, 1, \ldots, 2n\}$. By Lemma \ref{truncaT}, the truncated operator $T_{\gamma, Q}$ and also $T_1$ are bounded on $L^2(\mu)$ with bounds $\| T_{\gamma ,Q}\|_{2,2}
\leq 
\frac{\ell(Q)}{\gamma^\alpha }\leq 2^{N_0}$. Then, since 
$S_i^{l_1,l_2}\subset \tilde S$, we have
\begin{align*}
|\langle T\chi_{I_{i,k}},\chi_{S_i^{l_1,l_2}}\rangle |
&\leq \| T\|_{2,2}
\mu(I_{i,k})^{\frac{1}{2}}
\mu(S_i^{l_1,l_2})^{\frac{1}{2}}
\\
&\leq 2^{N_0}
\mu(I_{i,k})^{\frac{1}{2}}
\mu(\tilde S)^{\frac{1}{2}}
\end{align*}

With this,
\begin{align*}
|S_2|&\lesssim \mu(\tilde S)^{\frac{1}{2}}\| f\|_{L^{\infty }(\mu)}\| g\|_{L^{\infty }(\mu)}
2^{N_0}\sum_{k=0}^{N_0-1}\sum_{i=0}^{2^{kn}}
\mu(I_{i,k})^{\frac{1}{2}}
\\
&\lesssim \mu(\tilde S)^{\frac{1}{2}}\| f\|_{L^{\infty }(\mu)}\| g\|_{L^{\infty }(\mu)}
2^{N_0}
\mu(Q)^{\frac{1}{2}}
N_02^{N_0n}
\\
&
\lesssim \mu(\tilde S)^{\frac{1}{2}}\| f\|_{L^{\infty }(\mu)}\| g\|_{L^{\infty }(\mu)}
2^{N_0(n+3)}
<\epsilon ,
\end{align*}
In the second last inequality we used  $\mu(I_{i,k})\leq \mu(Q)\leq \rho (Q)\ell(Q)^\alpha \lesssim 
2^{N_0}$ and so, 
$\mu(Q)^{\frac{1}{2}}\leq 2^{N_0}$. 
The last inequality holds because 
$\tilde S\subset C_N$ and from the choice of $N$ in \eqref{Small}.

All this work finally proves the estimate for 
$\langle TP_{2M}^{\perp }P_{M_N}f,g_1\rangle $, the first term in \eqref{fdecom}. 

To deal with the second term in \eqref{fdecom}
$\langle Tf_1,g_{1,\partial}\rangle $ we note first that the reasoning to estimate
$D_1$, $N_i$, and $P_i$ can be applied unchanged to this new case. For the term $B_6$, we implement a small change. Since $B_6$ is completely symmetrical with respect the cubes $I$ and $J$, we can switch the roles played by these cubes,
$$
B_6=\sum_{J\in \mathcal{D}(Q)_{\geq N}}\sum_{I\in I_{\theta}}
\langle f,\psi_I\rangle \langle g,\psi_J\rangle
\langle T\psi_I,\psi_J \rangle .
$$
We now add and subtract the term
$$
A'=\sum_{J\in \mathcal{D}(Q)_{\geq N}}\sum_{I\in J_{\rm over}}
\langle f,\psi_I\rangle \langle g,\psi_J\rangle
\langle T\psi_I,\psi_J \rangle .
$$
which satisfies $|A'|\lesssim \|f\|_{L^2(\mu)}\|g\|_{L^2(\mu)}$
and we rewrite previous reasoning to obtain 
\begin{align*}
   & |B_6-A'|\lesssim \sum_{I\in \mathcal{D}(Q)_{N}}
\langle f\rangle_I
\langle T\chi_I,\sum_{J\in \mathcal{D}(Q)_{\geq N}\cap I_{\theta}}
\langle g,\psi_J\rangle \psi_J \rangle 
\\&
=\sum_{I\in \mathcal{D}(Q)_{N}}
\sum_{i=1}^{3^n}\langle f\rangle_I 
\langle g\rangle_{I_{i}}
\langle T\chi_I,\chi_{I_{i}} \rangle 
-\sum_{I\in \mathcal{D}(Q)_{N}}
\sum_{i=1}^{3^n}\langle f\rangle_I 
\langle g\rangle_{I_{i,k_i}}
\langle T\chi_I,\chi_{I_{i,k_i}} \rangle
\\&=S_{1}
-S_{2},
\end{align*}
in similar way as we did before. Again $S_{1}=0$, 
while we can reparametrize the sums in $S_2$ as we did before, 
to write
\begin{align*}
    |S_{2}|&\lesssim \| f\|_{L^{\infty }(\mu)}\| g\|_{L^{\infty }(\mu)}
\sum_{k=0}^{N_0-1}\sum_{j=0}^{2^{kn}}
\sum_{l_1,l_2=0}^1
|\langle T\chi_{S^{l_1,l_2}}, \chi_{J_{j,k}}\rangle |.
\end{align*}
Now we note that the cubes $I\in \tilde{\mathcal D}$ are open and so 
$N$ can be chosen large enough so that 
\begin{align*}
|S_2|\lesssim \mu(\tilde S)^{\frac{1}{2}}\| f\|_{L^{\infty }(\mu)}\| g\|_{L^{\infty }(\mu)}
2^{N_0(n+3)}
<\epsilon .
\end{align*}
This ends the estimate for $\langle Tf_1,g_{1,\partial}\rangle $. 

For the last term in \eqref{fdecom}  
$\langle Tf_{1,\partial},g_{1,\partial}\rangle $ we reason by reiteration. 
We first note that the supports of $f_{1,\partial}$ and $g_{1,\partial}$ are contained in the union of 
$\partial I$ for all $I\in \mathcal{D}^1(Q)$ with $\ell(I)\geq 2^{-N}\ell(Q)$. 
This set, which we denote by $\partial \mathcal{D}^1$, consists of finitely many euclidean affine spaces of dimension $n-1$, which are either pairwise parallel or pairwise perpendicular.  

Let now $\mathcal D=\mathcal D^2$. We now consider the families of cubes $\mathcal{D}^2(Q)$ 
and decompose $f_{1,\partial}=f_2+f_{2,\partial}$ 
where $f_{2,\partial}=f_{1,\partial}\chi_{\partial \mathcal{D}^2(Q)}$ and similar for $g_{1,\partial}$. Now, using the Haar wavelet system 
$(\psi_I)_{I\in \mathcal D^2}$ 
we decompose as before:
\begin{align}\label{fdecom2}
\langle Tf_{1,\partial},g_{1,\partial}\rangle 
&=\langle Tf_{1,\partial},g_2\rangle 
+\langle Tf_2,g_{2,\partial}\rangle 
+\langle Tf_{2,\partial},g_{2,\partial}\rangle 
\end{align}
Then we can apply all previous work to estimate the first two terms. 

For the third term, we note that the supports of $f_{2,\partial}$ and $g_{2,\partial}$ are now contained in $\partial \mathcal{D}^1(Q)\cap \partial \mathcal{D}^2(Q)$. Also that $\partial \mathcal{D}^2(Q)$ consists of finitely many euclidean affine spaces of dimension $n-1$, which are either pairwise parallel or pairwise perpendicular, and also either parallel or perpendicular
to every affine space of dimension $n-1$ of 
$\partial \mathcal{D}^1(Q)$. Then 
$\partial \mathcal{D}^1(Q)\cap \partial \mathcal{D}^2(Q)$ is 
a set consisting of finitely many euclidean affine spaces of dimension $n-2$. 

Then, by repeating the same argument $k=n-[\alpha ]+\delta(\alpha -[\alpha])$ times in total, we obtain   
$P_{2M}^{\perp}P_{M_N}f=\sum_{i=1}^{k}f_i+f_{i,\partial}$ and similar for $P_{M}^{\perp}P_{M_N}g$ such that the appropriate estimates hold for $|\langle Tf_i, \cdot \rangle |$ and $|\langle \cdot , T^*g_i\rangle |$ for all 
with $i\in \{1, \ldots ,k\}$  and 
the functions $f_{k,\partial }, g_{k,\partial }$,  are supported on 
$\bigcap_{i=1}^{k}\partial \mathcal{D}^i(Q)$. 
By repeating previous reasoning on parallel and perpendicular affine spaces, we conclude that this set consists of finitely many euclidean affine spaces of dimension $n-k=[\alpha]-\delta(\alpha -[\alpha])$, which are either pairwise parallel or pairwise perpendicular.

But now we can show that $\bigcap_{i=1}^{k}\partial \mathcal{D}^i(Q)$  
has measure zero with respect to $\mu$. 
Let $I$ be an arbitrary $n-k$ dimensional dyadic cube with side length $\ell(I)$. Let 
$(J_i)_{i=1}^{m}$ be a 
family of pairwise disjoint $n$-dimensional cubes $J_i$ with fixed side length $r$ such that 
$I\subset \cup_iJ_i$. This family has 
cardinality  
$m=(\frac{\ell(I)}{r})^{n-k}$. Then 
$$
\mu(I)\leq \sum_{i=1}^{m}\mu(J_{i})
\lesssim (\frac{\ell(I)}{r})^{n-k}r^{\alpha}
=\ell(I)^{n-k}r^{\alpha -n+k}. 
$$
Since 
$\alpha -n+k=\alpha-[\alpha]+\delta(\alpha -[\alpha])>0$,
we have  
$$
\mu(I)\lesssim \ell(I)^{n-k}
\lim_{r\rightarrow 0}r^{\alpha -n+k}=0
$$ 
for all cubes $I$ of dimension $n-k$. This shows that $\mu(\bigcap_{i=1}^{k}\partial \mathcal{D}^i(Q))=0$ and so, $\langle Tf_{k,\partial},g_{k,\partial}\rangle=0$. This  
finishes the proof of the first part of the theorem,  except for the last result in Proposition \ref{Ttrunc2T}.

The proof of the second part follows similar steps. As before, we first work with the classical $n$ dimensional dyadic grid $\mathcal D^n(Q)$ 
and decompose 
$P_M^{\perp}P_{M_N}f=f_1+f_{1,\partial}$, where 
$f_{1,\partial}=(P_M^{\perp}P_{M_N}f)\chi_{\partial D(Q)}$ and similar for 
$P_M^{\perp}P_{M_N}g$. 
With this
\begin{align*}
\langle TP_{2M}^{\perp }P_{M_N}f,P_{M}^{\perp }P_{M_N}g\rangle 
&=\langle TP_{2M}^{\perp }P_{M_N}f,g_1\rangle 
+\langle Tf_1,g_{1,\partial}\rangle 
\\&
+\langle TPf_{1,\partial},g_{1,\partial}\rangle . 
\end{align*}

Then we  use previous reasoning
to estimate the first two terms: 
$\langle TP_{2M}^{\perp }P_{M_N}f,g_1\rangle $ and 
$\langle Tf_1,g_{1,\partial}\rangle 
$. 

To control 
$\langle Tf_{1,\partial },g_{1,\partial}\rangle$
we note that the supports of $f_{1,\partial }$ and $g_{1,\partial}$ are contained in $\partial \mathcal D^n(Q)$. Then we consider the dyadic $n-1$ dimensional grid 
$\mathcal D^{n-1}_{\partial}(Q)$ and 
decompose 
$f_{1,\partial}=f_2+f_{2,\partial}$, 
with $f_{2,\partial}=f_{1,\partial}\chi_{\partial \mathcal D^{n-1}(Q)}$ and similar for $g$. 
Similarly as before, we use the Haar wavelets $(\psi_I)_{I\in \mathcal D^{n-1}(Q)}$ to estimate 
the first two terms $\langle Tf_{1,\partial},g_2\rangle $ and 
$\langle Tf_2,g_{2,\partial}\rangle 
$. 

To control 
$\langle Tf_{2,\partial },g_{2,\partial}\rangle$ 
we note that $f_{2,\partial }$ and $g_{2,\partial}$ are supported on $\partial \mathcal D^{n-1}(Q)$
and we reiterate the process. 

By repeating the same argument $k=n-[\alpha ]+\delta(\alpha -[\alpha])$ times, we obtain   
$P_{2M}^{\perp }P_{M_N}f=\sum_{i=1}^{k}f_i+f_{i,\partial}$ and similar for $P_{M}^{\perp }P_{M_N}g$ such that the appropriate estimates hold for $|\langle Tf_i, \cdot \rangle |$ and $|\langle \cdot , T^*g_i\rangle |$ for all $i\in \{1, \ldots ,k\}$ and 
the functions $f_{k,\partial }, g_{k,\partial }$,  are supported on 
$\partial \mathcal{D}^{n-k+1}(Q)$, for which we consider the $n-k$ dimensional grid $\mathcal D^{n-k}_{\partial}(Q)$.  
We prove as before that $\partial \mathcal{D}^{n-k+1}(Q)$ has measure zero with respect to $\mu$. 

Let $I\in \mathcal D^{n-k}_{\partial}(Q)$ be an arbitrary $n-k$ dimensional dyadic cube with side length $\ell(I)$. We cover $I$ with a family of  $n$ dimensional cubes $(J_i)_{i=1}^m$ with side length $r$ and cardinality 
$m=(\frac{\ell(I)}{r})^{n-k}$. Then again
$$
\mu(I)\leq \sum_{i=1}^{m}\mu(J_{i})
\lesssim (\frac{\ell(I)}{r})^{n-k}r^{\alpha}
=\ell(I)^{n-k}r^{\alpha -n+k}.
$$
As before, $\mu(I)\lesssim \ell(I)^{n-k}
\lim_{r\rightarrow 0}r^{\alpha -n+k}=0$ for every cube $I\in \mathcal D^{n-k}_{\partial}(Q)$ of dimension $n-k$. This shows that $\mu(\partial \mathcal{D}^{n-k+1}(Q))=0$ and so, $\langle Tf_{k,\partial},g_{k,\partial}\rangle=0$, 
finishing the proof of the second part of the theorem.
To completely finish the result, we need Proposition 
\ref{Ttrunc2T}. 

\end{proof}

The next result shows the way the truncated operators dominate the original operator. The proof is obtained by modifying a reasoning in the first chapter of \cite{ST}, where the measure is doubling and the original operator is assumed to be bounded. 

\begin{proposition}\label{Ttrunc2T}
Let $T_{\gamma ,Q}$ be the uniformly bounded smooth truncated operators of Definition \ref{TgammaQ}. 
Let $k=n-[\alpha]-\delta(\alpha-[\alpha])$.
Then for $f,g\in L^{2}(\mu)$ and $\epsilon >0$ 
there exist functions $(f_i)_{i=1}^k$, $(g_i)_{i=1}^k$, 
$(f_i)_{i=0}^{k}$, $(g_i)_{i=1}^{k}$
with 
$\|f_i\|_{L^{2}(\mu)}, \|f_{i,\partial}\|_{L^{2}(\mu)}\leq
\|f\|_{L^{2}(\mu)}$, $\|g_i\|_{L^{2}(\mu)}, \|g_{i,\partial}\|_{L^{2}(\mu)}\leq 
\|g\|_{L^{2}(\mu)}$ and 
$M_0\in \mathbb N$ such that 
for all $M>M_0$, 
\begin{align*}
|\langle TP_M^{\perp} f,P_M^{\perp}g\rangle |
&\lesssim \sup_{\gamma, Q}\sum_{i=1}^k
|\langle T_{\gamma  ,Q}f_{i-1,\partial }, g_i\rangle | +
|\langle T_{\gamma  ,Q}f_i, g_{i,\partial}\rangle |
\\
&
+\epsilon \|f\|_{L^{2}(\mu)}\|g\|_{L^{2}(\mu)}.
\end{align*}
\end{proposition}
\begin{proof}
Let $f,g\in L^{2}(\mu)$ and $\epsilon >0$ be fixed. 
We can assume $f,g$ are compactly supported and have integral zero. 

Let $Q\in \mathcal C$ such that 
$\sup f\cup \sup g\subset 2^{-1}Q$. 
Since $T_{\gamma ,Q}$ are uniformly bounded,  there exists a sequence $(\gamma_j)$ converging to zero and an  operator
$T_0$ bounded on $L^2(\mu)$ such that 
the operators $T_{\gamma_j ,Q}$ weakly converge to $T_0$ in $L^2(\mu)$, namely
$$
\lim_{j\rightarrow \infty }\langle (T_{\gamma_j ,Q}-T_0)f,g\rangle =0
$$
and $
\lim_{j\rightarrow \infty }\| (T_{\gamma_j ,Q}-T_0)f\|_{L^2(\mu)} =0 .
$

Let $M>M_0>0$, $N>0$ and $M_N$ the parameters fixed at the beginning of the proof of Theorem \ref{Mainresult2}. 
Let  the functions $P_{2M}^{\perp}P_{M_N}f$ and 
$P_{M}^{\perp}P_{M_N}g$. We remind the decompositions 
$P_{2M}^{\perp}P_{M_N}f=\sum_{i=1}^{k}f_i+f_{i,\partial}$, $P_{M}^{\perp}P_{M_N}g=\sum_{i=1}^{k}g_i+g_{i,\partial}$,
given in \eqref{fdecom} and 
\eqref{fdecom2} as follows: 
\begin{itemize}
\item First,
$P_{2M}^{\perp}P_{M_N}f=f_1+f_{1,\partial}$,  where 
$f_{1,\partial}=(P_M^{\perp}P_{M_N}f)\chi_{\partial D^1(Q)}$, and 
$P_M^{\perp}P_{M_N}g=g_1+g_{1,\partial}$, where 
$g_{1,\partial}=(P_M^{\perp}P_{M_N}g)\chi_{\partial D^1(Q)}$.
\item For $i\in \{1, \ldots, k\}$,
$f_{i,\partial}=f_{i+1}+f_{i+1,\partial}$ 
where $f_{i+1,\partial}=f_{i,\partial}\chi_{\partial \mathcal{D}^{i+1}(Q)}$, and  
$g_{i,\partial}=g_{i+1}+g_{i+1,\partial}$ 
where $g_{i+1,\partial}=g_{i,\partial}\chi_{\partial \mathcal{D}^{i+1}(Q)}$.
\end{itemize}

With this, if we denote $f_{0,\partial}=P_{2M}^{\perp}P_{M_N}f$,  
by the recursive definitions just provided we can write 
\begin{align}\label{usualdecomp}
\langle TP_{2M}^{\perp}P_{M_N}f, P_{M}^{\perp}P_{M_N}g\rangle 
&=\sum_{i=1}^{k}\langle Tf_{i-1,\partial }, g_i\rangle +
\langle Tf_i, g_{i,\partial}\rangle
\\
&
\nonumber
+\langle Tf_{k,\partial}, g_{k,\partial}\rangle ,
\end{align}
where $f_i, g_i$ are zero on $\partial \mathcal D^i(Q)$ and $\langle Tf_{k,\partial}, g_{k,\partial}\rangle=0$ as we saw before. 

We now prove that for $i\in \{ 1,\ldots, k\}$ and 
$g$ such that is zero on $\partial \mathcal D^i(Q)$
we have 
\begin{align}\label{boundeddif}
\langle Tf,g\rangle = \langle T_0f,g\rangle +
\langle a_if,g\rangle ,
\end{align}
where $a_i$ is 
such that $|\langle a_if,g\rangle|\lesssim \epsilon  \|f\|_{L^{2}(\mu)}\|g\|_{L^{2}(\mu)}$. Similarly  
$
\langle f,T^*g\rangle = \langle f,T_0^*g\rangle +
\langle f,b_ig\rangle
$
for every function $f$ that is zero on $\partial \mathcal D^i(Q)$, where $b_i$ is 
such that $|\langle f,b_ig\rangle|\lesssim \epsilon \|f\|_{L^{2}(\mu)}\|g\|_{L^{2}(\mu)}$. 

Assuming these equalities, we can prove the statement. Let $j_i$ large enough so that 
$|\langle (T_0-T_{\gamma_{j_i} ,Q})f_{i-1,\partial }, g_i\rangle|<\epsilon $. Then
\begin{align}\label{Tfirst}
|\langle Tf_{i-1,\partial }, g_i\rangle|&\leq |\langle T_{\gamma_{j_i} ,Q}f_{i-1,\partial }, g_i\rangle |
+|\langle (T_0-T_{\gamma_{j_i} ,Q})f_{i-1,\partial }, g_i\rangle|
\\
&
\nonumber
+|\langle a_if_{i-1,\partial }, g_i\rangle|
\\
\nonumber
&\leq \langle T_{\gamma_{j_i} ,Q}
f_{i-1,\partial }, g_i\rangle|
+2\epsilon \|f\|_{L^{2}(\mu)}\|g\|_{L^{2}(\mu)},
\end{align}
and similar for 
$|\langle Tf_{i,\partial }, g_{i,\partial}\rangle|$. Then
with \eqref{usualdecomp} and \eqref{Tfirst}, we get
\begin{align*}
|\langle TP_{2M}^{\perp}P_{M_N}f, P_{M}^{\perp}P_{M_N}g\rangle | 
&\lesssim
\sum_{i=1}^k
\sup_{\gamma, Q}|\langle T_{\gamma  ,Q}f_{i-1,\partial }, g_i\rangle | +
|\langle T_{\gamma  ,Q}f_i, g_{i,\partial}\rangle |
\\
&
+2k\epsilon \|f\|_{L^{2}(\mu)}\|g\|_{L^{2}(\mu)},
\end{align*}
which is comparable with the statement. 

We now work to prove \eqref{boundeddif}. Let $D=T-T_0$. 
Then we prove that $\langle Df,g\rangle=\langle a_i f,g\rangle$. 
Let $I'\in \mathcal C$ such that 
$2Q\subset I'$ and let $I\in \mathcal D(2^{-1}Q)$. We first show that, 
for all $g\in L^2(\mu)$ such that $g$ is zero on $\partial \mathcal D^i(Q)$, 
\begin{align}\label{lastclaim}
\langle D(\chi_I), g\rangle
=\langle \chi_{I}D(\chi_{I'}), g\rangle .
\end{align}
For this we proceed as follows. 
We first assume that $g$ satisfies the additional condition such that $\dist (\sup g, I)>0$. 
For $\epsilon' >0$, let $j_I\in \mathbb N$ be large enough so that 
$|\langle (T_{\gamma_{j_I},Q}-T_0)(\chi_I), g\rangle | <\epsilon' $ and 
$\dist (\sup g, I)>2\gamma_{j_I}$. Then 
$$
\langle D(\chi_I), g\rangle
=\langle (T-T_{\gamma_{j_I},Q})(\chi_I), g\rangle+
\langle (T_{\gamma_{j_I},Q}-T_0)(\chi_I), g\rangle .
$$
Since $\sup g\cap I=\emptyset $, $x\in \sup g\subset 2^{-1}Q$ and $t\in I\subset 2^{-1}Q$, we have
$$
\langle (T-T_{\gamma_{j_I},Q})(\chi_I), g\rangle
=\int \int_{I}K(t,x)\phi(\frac{|t-x|}{\gamma_{j_I}})d\mu(t)g(x)d\mu(x)
=0
$$
due to the facts that 
and $\supp \phi\subset [-2,2]$ and 
$|t-x|\geq \dist (I,\sup g)>2\gamma_{j_I}$ 
. Then 
$$
|\langle D(\chi_I), g\rangle|
=|\langle (T_{\gamma_{j_I},Q}-T_0)(\chi_I), g\rangle|
<\epsilon' .
$$
Since the inequality holds for all $\epsilon' >0$, we conclude that $\langle D(\chi_I), g\rangle=0$ 
for all $g\in L^2(\mu)$ such that $\dist (\sup g, I)>0$. 

Now for all $\lambda >0$ we denote 
$I_{\lambda }=\{ x\in \mathbb R^n / \dist(x,I)<\lambda \ell(I) \}
$. 
By previous reasoning, since 
$\dist (\sup (1-\chi_{I_\lambda }), I)\geq \lambda \ell(I)>0$, we have  
for all $g\in L^2(\mu)$ that are zero in $\partial \mathcal D(Q)$,
\begin{align}\label{break}
\langle D(\chi_I), g\rangle
&=\langle (1-\chi_{I_\lambda })D(\chi_I),g\rangle 
+\langle \chi_{I_\lambda }D(\chi_I)
,g\rangle 
\\
&
\nonumber
=\langle \chi_{I_\lambda }D(\chi_I), g\rangle .
\end{align}

Now, if we denote $I^c=I'\setminus I$, we can write
$$
\langle  \chi_{I_\lambda }D(\chi_I),g\rangle =\langle D(\chi_{I'}),g \chi_{I_\lambda } \rangle -\langle D(\chi_{I^c}), g \chi_{I_\lambda }\rangle .
$$ By previous reasoning we have that 
$\langle D(\chi_{I^c}), h\rangle =0$
for all $h\in L^2(\mu)$ such that $\dist (\sup h, I^c)>0$.
Then, as in \eqref{break},  
$$
\langle D(\chi_{I^c}), g\chi_{I_\lambda}\rangle 
=\langle \chi_{(I^c)_\lambda }D(\chi_{I^c}), g\chi_{I_\lambda} \rangle .
$$
With this, for $\lambda >0$ and $g\in L^2(\mu)$
that is zero on $\partial \mathcal D^i(Q)$, we have  
$$
\langle D(\chi_I), g\rangle
=\langle \chi_{I_\lambda }D(\chi_{I'}), g\rangle
-\langle \chi_{I_\lambda }\chi_{(I^c)_\lambda }D(\chi_I), g\rangle. 
$$
And since $\lambda >0$ is arbitrary, we get 
\begin{align}\label{limdif}
\langle D(\chi_I), g\rangle
=\lim_{\lambda \rightarrow 0}\langle \chi_{I_\lambda }D(\chi_{I'}), g\rangle
-\lim_{\lambda \rightarrow 0}\langle \chi_{I_\lambda }\chi_{(I^c)_\lambda }D(\chi_I), g\rangle .
\end{align}

For the first term, we reason as follows. 
By the testing condition on $T$ and the boundedness of 
$T_{\gamma ,Q}$, we have $\| \chi_{I'}D(\chi_{I'})\|_{L^{2}(\mu)}
\lesssim \mu(I')^{\frac{1}{2}}$. Then
$D(\chi_{I'})g$ 
is integrable on $I'$. 
Since $I_{\lambda }\subset I'$, 
$|\chi_{I_\lambda }D(\chi_{I'})g|\leq 
\chi_{I'}|D(\chi_{I'})g|\in L^1(\mu)$. Moreover, 
$\lim_{\lambda \rightarrow 0}\chi_{I_\lambda }
=\chi_{\bar I}$, 
where 
$\bar I\in \mathcal C$ is the closed cube defined by the closure of $I$.
Then, by Lebesgue's Dominated Theorem, 
$$
\lim_{\lambda \rightarrow 0}
\langle \chi_{I_\lambda }D(\chi_{I'}), g\rangle
=\langle \chi_{\bar I}D(\chi_{I'}), g\rangle, 
=\langle \chi_{I}D(\chi_{I'}), g\rangle .
$$
The last equality holds because $g$ is zero on $\partial \mathcal D^i(Q)$.

For the second term in \eqref{limdif}, we work  differently. 
Let $I_{-\lambda }=(1-\lambda )I$.
Then
\begin{align}\label{limdif2}
\langle \chi_{I_\lambda }\chi_{(I^c)_\lambda }D(\chi_I), g\rangle
=\langle \chi_{I\setminus I_{-\lambda }}D(\chi_I), g\rangle\
+\langle \chi_{I_\lambda \setminus \bar I}D(\chi_I), g\rangle .
\end{align}
The first new term can be treated as before:  
$\| \chi_{I}D(\chi_{I})\|_{L^{2}(\mu)}
\lesssim \mu(I)^{\frac{1}{2}}$ and so, 
$D(\chi_{I})g$ is integrable on $I$.
Moreover, $\lim_{\lambda \rightarrow 0}\chi_{I_{-\lambda}}
=\chi_{\partial I}$ and 
$|\chi_{I\setminus I_{-\lambda }}D(\chi_{I})g|\leq 
\chi_{I}|D(\chi_{I})g|\in L^1(\mu)$.
Then, by 
Lebesgue's Dominated Theorem
$$
\lim_{\lambda \rightarrow 0}
\langle \chi_{I\setminus I_{-\lambda }}D(\chi_{I}), g\rangle
=\langle \chi_{\partial I}D(\chi_{I}), g\rangle =0,
$$
since $g$ is zero on $\partial \mathcal D^i(Q)$. 

For the second term in \eqref{limdif2}, we proceed as follows. 
Let $S_{r}=\{x\in I_\lambda \setminus \bar I: 
2^{-(r+1)}\ell(I_\lambda \setminus \bar I)< \dist(x,I)\leq 2^{-r}\ell(I_\lambda \setminus \bar I)\}$.
Then since 
$I_\lambda \setminus \bar =\bigcup_{r=0}^{\infty }S_r$, we have
\begin{align*}
\langle \chi_{I_\lambda \setminus \bar I}
D(\chi_I), g\rangle
=\langle \lim_{R\rightarrow \infty }
\sum_{r=0}^R\chi_{S_r}
D(\chi_I), g\rangle
\end{align*}
Then, by Fatou's lemma,
\begin{align}\label{Fatou}
|\langle \chi_{I_\lambda \setminus \bar I}
D(\chi_I), g\rangle |
&\leq \langle \lim_{R\rightarrow \infty }
\sum_{r=0}^R\chi_{S_r}
|D(\chi_I)|, |g|\rangle
\\
\nonumber
&\leq  \liminf_{R\rightarrow \infty } 
\sum_{r=0}^R\langle \chi_{S_r}
|D(\chi_I)|, |g|\rangle
\end{align}

Given 
$\epsilon'>0$, we choose for $j_r$, dependent on $r$, $\lambda $, $\epsilon'$, $I$ and $g$ such that
$2\gamma_{j_r}<2^{-(r+1)}\ell(I_{\lambda}\setminus \bar I)$ and 
$\|(T_{\gamma_{j_{r}},Q}-T_0)(\chi_I)\|_{L^2(\mu)} \| g\|_{L^2(\mu)}
<\epsilon'2^{-r}$. 
Now we decompose as
\begin{align*}
\langle \chi_{S_r}
|D(\chi_I)|, |g|\rangle
&\leq \langle \chi_{S_r}
|(T-T_{\gamma_{j_{r}},Q})(\chi_I)|, |g|\rangle
\\&
+\langle \chi_{S_r}
|(T_{\gamma_{j_{r}},Q}-T_0)(\chi_I), |g|\rangle. 
\end{align*}

The second term can be bounded by 
\begin{align*}
\|(T_{\gamma_{j_{r}},Q}-T_0)(\chi_I)\|_{L^2(\mu)} \| g\|_{L^2(\mu)}
<\epsilon'2^{-r} 
\end{align*}
For the first term, we denote 
$D_r=T-T_{\gamma_{j_r},Q}$. Since 
$S_r\cap I=\emptyset $, 
\begin{align*}
\langle \chi_{S_r}
|D_r(\chi_I)|, |g|\rangle
=\int_{S_r}
\Big|\int_I K(t,x)\phi(\frac{|t-x|}{\gamma_{j_r}})d\mu(t)\Big| |g(x)|
d\mu(x)=0
\end{align*}
since  
$2\gamma_{j_r}<^{-(r+1)}\ell(I_{\lambda}\setminus \bar I)< \dist(x, I)\leq |t-x|$. 
With both estimates we continue the estimate in \eqref{Fatou} as 
\begin{align*}
|\langle \chi_{I_\lambda \setminus \bar I}
D(\chi_I), g\rangle |
&\leq  \sum_{r=0}^{\infty }
\epsilon' 2^{-r}
\lesssim \epsilon'
\end{align*}
for all $\epsilon'>0$. Then
$\langle \chi_{I_\lambda \setminus \bar I}
D(\chi_I), g\rangle =0$. Finally, by combining the decompositions in \eqref{limdif} and \eqref{limdif2}, we have 
$
\langle D(\chi_I), g\rangle
=\langle \chi_{I}D(\chi_{I'}), g\rangle 
$, 
which is the equality claimed in \eqref{lastclaim}. 

Now we use \eqref{lastclaim} to prove \eqref{boundeddif}. By telescoping sums, we can write that $f_{i,\partial}
=\sum_{j}\langle f\rangle_{I_j}\chi_{I_j}$ with $I_j\in (\mathcal D^i)_M^c(Q)$ such that 
$I_j\subset Q\subset 2^{-1}I'$ and small enough so that
$F_T(I_j)<\epsilon $. Then we are going to consider only functions with the described decomposition. 
In that case, by \eqref{lastclaim},
\begin{align*}
\langle D(f), g\rangle
&=\sum_{j}\langle f\rangle_{I_j}\langle D(\chi_{I_j}), g\rangle
=\sum_{j}\langle f\rangle_{I_j}\langle \chi_{I_j}D(\chi_{I'}), g\rangle
\\
&
=\langle fD(\chi_{I'}), g\rangle 
=\langle a_if, g\rangle.
\end{align*}
with $a_i=D(\chi_{I'})$. We now show 
that 
\begin{align*}
|\langle a_if, g\rangle|
\lesssim \epsilon \|f\|_{L^{2}(\mu)}\|g\|_{L^{2}(\mu)}.
\end{align*}
Let $I,J\in \mathcal D^i(Q)$, $\tilde J\in \mathcal {\tilde D}^i(Q)$ with $\ell(I)=\ell(J)$ and such that 
$\tilde J$ is the interior of $J$. By the definition of 
$a_i$ and \eqref{lastclaim} again, we have
\begin{align}\label{last}
\langle a_i\chi_I, \chi_{\tilde J}\rangle
&=\langle \chi_{I\cap J}D(\chi_{I'}), \chi_{\tilde J}\rangle 
=\langle D(\chi_{I\cap J}), \chi_{\tilde J}\rangle
\end{align}
Now, if $I\cap J=\emptyset $, we get 
$\langle a_i\chi_I, \chi_{\tilde J}\rangle =0$. 
Otherwise, since $\ell(I)=\ell(J)$, we have $I=J$. Then, by \eqref{last}, the testing condition on $T$ and the uniform boundedness of $T_{\gamma, Q}$, we get
\begin{align*}
|\langle a_i\chi_I, \chi_{\tilde J}\rangle|
&=|\langle D\chi_{J}, \chi_{\tilde J}\rangle|
\leq |\langle T\chi_{J}, \chi_{\tilde J}\rangle|
+|\langle T_{\gamma, Q}\chi_{J}, \chi_{\tilde J}\rangle|
\\&
\lesssim (\|\chi_{J}T\chi_{J}\|_{L^{2}(\mu)}+\|\chi_{J}T_{\gamma, Q}\chi_{J}\|_{L^{2}(\mu)})\mu(\tilde J)
\\
&
\lesssim F_T(J)\mu(\tilde J)<\epsilon \mu(\tilde J).
\end{align*}
We denote now  
$F=\sum_{I\in \mathcal D^i(Q)_{N}}\langle f\rangle_{I} \chi_{I}$ and 
$G=\sum_{\tilde I\in \mathcal {\tilde D}^i(Q)_{N}}\langle g\rangle_{\tilde I} \chi_{\tilde I}$, with 
$F_T(I)<\epsilon $ and $\tilde I$ being the interior of $I$. Then 
\begin{align*}
|\langle a_iF, G\rangle|
&
=|\sum_{I\in \mathcal D^i(Q)_{N}}\sum_{\tilde J\in \mathcal {\tilde D}^i(Q)_{N}}\langle f\rangle_I \langle g\rangle_{\tilde J}\langle a_i\chi_I, \chi_{\tilde J}\rangle|
\\
&
\leq \sum_{I\in \mathcal D^i(Q)_{N}}|\langle f\rangle_I | |\langle g\rangle_{\tilde I} |
|\langle D\chi_{I}, \chi_{\tilde I}\rangle|
\\
&
\lesssim \epsilon \sum_{I\in \mathcal D^i(Q)_{N}}|\langle f\rangle_I | |\langle g\rangle_{\tilde I} |\mu(\tilde{I})
\\
&
\lesssim \epsilon \Big(\sum_{I\in \mathcal D^i(Q)_{N}}|\langle f\rangle_I |^2
\mu(I)^{\frac{1}{2}}\Big)^{\frac{1}{2}}
\Big(\sum_{\tilde I\in \mathcal {\tilde D}^i(Q)_{N}}|\langle g\rangle_{\tilde I} |^2
\mu(\tilde I)^{\frac{1}{2}}\Big)^{\frac{1}{2}}
\\
&
\lesssim \epsilon \|f\|_{L^{2}(\mu)}\|g\|_{L^{2}(\mu)}.
\end{align*}


\end{proof}

\bibliographystyle{amsplain}
\bibliography{refs}

\end{document}